%% file: AguillonHoernschemeyerSainte-Marie.tex
\documentclass{article}
\usepackage[utf8]{inputenc}
\usepackage[english]{babel}
\usepackage{amssymb}
\usepackage{amsmath}
\usepackage{amsthm}
\usepackage{epsfig}
\usepackage{xcolor}
\usepackage{hyperref}
\usepackage{ulem}
\definecolor{rwth}   {RGB}{  0  84 159}
\hypersetup{
    colorlinks=true,
    linkcolor=rwth,
    citecolor=rwth,
    }
\usepackage{bbm}
\usepackage{cases}
\usepackage{subcaption}
\usepackage{authblk}
\usepackage{enumerate}

\usepackage{algorithm}
\usepackage{algpseudocode}

\usepackage{tikz}
\usepackage{pgfplots}

\usetikzlibrary{shapes,arrows,positioning}
\usetikzlibrary{shapes.geometric,calc, fit}

\definecolor{rwth-75}{RGB}{ 64 127 183}
\definecolor{rwth-50}{RGB}{142 186 229}
\definecolor{rwth-25}{RGB}{199 221 242}
\definecolor{rwth-10}{RGB}{232 241 250}

\definecolor{orange-125}{RGB}{196 134 0}
\definecolor{orange}    {RGB}{246 168 0}

\usepackage[a4paper]{geometry}
\graphicspath{{Figures/}}

\newcommand{\R}{\mathbb{R}}

\newtheorem{theorem}{Theorem}[section]
\newtheorem{lemma}[theorem]{Lemma}
\newtheorem{defn}[theorem]{Definition}
\newtheorem{proposition}[theorem]{Proposition}

\newtheorem{remark}[theorem]{Remark}

\definecolor{Red}{rgb}{1,0,0}
\definecolor{Blue}{rgb}{0,0,1}

\definecolor{red}{rgb}{1,0,0}
\definecolor{blue}{rgb}{0,0,1}
\newcommand{\revised}[1]{\textcolor{black}{#1}} % orange

\newcommand{\ha}{h_{\alpha}}

\newcommand{\ua}{u_{\alpha}}
\newcommand{\va}{v_{\alpha}}
\newcommand{\phia}{\varphi_{\alpha}}
\newcommand{\Ta}{T_{\alpha}}
\newcommand{\la}{l_{\alpha}}

\newcommand{\uj}{u_{j}}

\newcommand{\lj}{l_{j}}
\newcommand{\zb}{z_{b}}
\newcommand{\sa}{\sigma_{\alpha}}
\newcommand{\ta}{\theta_{\alpha}}

\newcommand{\Ec}{E^c}
\newcommand{\Eca}{\Ec_{\alpha}}
\newcommand{\Ecasa}{\Eca \sa}
\newcommand{\baru}{\bar{u}}
\newcommand{\barv}{\bar{v}}
\newcommand{\barT}{\bar{T}}

\newcommand{\CFL}{\mathrm{CFL}}

\newcommand{\Dtn}{\Delta t^{n}}
\newcommand{\dtk}{\delta t^{k}}

\newcommand{\Fha}{F^{\ha}}
\newcommand{\Fh}{F^{h}}

\newcommand{\Fhaua}{F^{\ha \ua}}
\newcommand{\FhaTa}{F^{\ha \Ta}}
\newcommand{\FEcasa}{F^{\Eca \sa}}

\newcommand{\Fhajp}{\Fha_{j+\frac12}}
\newcommand{\Fhajm}{\Fha_{j-\frac12}}
\newcommand{\Fhajpm}{\Fha_{j\pm\frac12}}
\newcommand{\Fhjp}{\Fh_{j+\frac12}}
\newcommand{\Fhjpm}{\Fh_{j\pm\frac12}}
\newcommand{\Fhjm}{\Fh_{j-\frac12}}

\newcommand{\Fhauajp}{\Fhaua_{j+\frac12}}
\newcommand{\Fhauajm}{\Fhaua_{j-\frac12}}
\newcommand{\FhaTajp}{\FhaTa_{j+\frac12}}
\newcommand{\FhaTajm}{\FhaTa_{j-\frac12}}
\newcommand{\FEcasajp}{F^{\Eca \sa}_{j+\frac12}}
\newcommand{\FEcasajm}{F^{\Eca \sa}_{j-\frac12}}

\newcommand{\uap}{u_{\alpha+\frac12}}
\newcommand{\uam}{u_{\alpha-\frac12}}
\newcommand{\vap}{v_{\alpha+\frac12}}
\newcommand{\vam}{v_{\alpha-\frac12}}
\newcommand{\Tap}{T_{\alpha+\frac12}}
\newcommand{\Tam}{T_{\alpha-\frac12}}

\newcommand{\hjn}{h_{j}^{n}}
\newcommand{\hjpn}{h_{j+1}^{n}}
\newcommand{\hjd}{h_{j}^{n+\frac12}}
\newcommand{\hjnp}{h_{j}^{n+1}}

\newcommand{\han}{h_{\alpha}^{n}}
\newcommand{\haj}{h_{\alpha,j}}
\newcommand{\hajn}{h_{\alpha,j}^{n}}
\newcommand{\hajp}{h_{\alpha,j+1}}
\newcommand{\hajpn}{h_{\alpha,j+1}^{n}}
\newcommand{\hajm}{h_{\alpha,j-1}}

\newcommand{\hajs}{h_{\alpha,j}^{\star}}
\newcommand{\has}{h_{\alpha}^{\star}}
\newcommand{\had}{h_{\alpha}^{n+\frac12}}

\newcommand{\hajd}{h_{\alpha,j}^{n+\frac12}}

\newcommand{\uan}{u_{\alpha}^{n}}
\newcommand{\uaj}{u_{\alpha,j}}
\newcommand{\uajn}{u_{\alpha,j}^{n}}
\newcommand{\uajnp}{u_{\alpha,j}^{n+1}}
\newcommand{\uajpn}{u_{\alpha,j+1}^{n}}

\newcommand{\uajm}{u_{\alpha,j-1}}
\newcommand{\uajs}{u_{\alpha,j}^{\star}}
\newcommand{\uas}{u_{\alpha}^{\star}}
\newcommand{\uad}{u_{\alpha}^{n+\frac12}}
\newcommand{\uajps}{u_{\alpha+1,j}^{\star}}

\newcommand{\uajms}{u_{\alpha-1,j}^{\star}}
\newcommand{\uajpds}{u_{\alpha+\frac12,j}^{\star}}
\newcommand{\uajmds}{u_{\alpha-\frac12,j}^{\star}}
\newcommand{\uajd}{u_{\alpha,j}^{n+\frac12}}

\newcommand{\uasjp}{u_{\alpha, {j+\frac12}}^{n} }
\newcommand{\uasjm}{u_{\alpha, {j-\frac12}}^{n} }
\newcommand{\uajpdnd}{u_{\alpha+\frac12,j}^{n+\frac12}}
\newcommand{\uajmdnd}{u_{\alpha-\frac12,j}^{n+\frac12}}

\newcommand{\Tan}{T_{\alpha}^{n}}
\newcommand{\Tajn}{T_{\alpha,j}^{n}}
\newcommand{\Tajnp}{T_{\alpha,j}^{n+1}}

\newcommand{\Tas}{T_{\alpha}^{\star}}
\newcommand{\Tad}{T_{\alpha}^{n+\frac12}}
\newcommand{\Tajd}{T_{\alpha,j}^{n+\frac12}}

\newcommand{\Gap}{G_{\alpha+\frac12}}
\newcommand{\Gam}{G_{\alpha-\frac12}}
\newcommand{\Gapm}{G_{\alpha \pm \frac12}}
\newcommand{\Gapj}{G_{\alpha+\frac12,j}}

\newcommand{\Gamj}{G_{\alpha-\frac12,j}}

\newcommand{\fhk}{f^{h,k}}
\newcommand{\fhuk}{f^{h \baru,k}}

\newcommand{\fhsak}{f^{h \sa,k}}

\newcommand{\fhTak}{f^{h \Ta,k}}

\newcommand{\Sjndk}{S_{j}^{k}}

\newcommand{\barujdk}{\baru_{j}^{n+\frac12,k}}

\newcommand{\fhkjp}{\fhk_{j+\frac12}}
\newcommand{\fhkjm}{\fhk_{j-\frac12}}
\newcommand{\fhkjpm}{\fhk_{j\pm\frac12}}
\newcommand{\fhukjp}{\fhuk_{j+\frac12}}
\newcommand{\fhukjm}{\fhuk_{j-\frac12}}
\newcommand{\fhukjpm}{\fhuk_{j\pm\frac12}}

\newcommand{\fhsakjp}{\fhsak_{j+\frac12}}
\newcommand{\fhsakjm}{\fhsak_{j-\frac12}}

\newcommand{\fhTakjp}{\fhTak_{j+\frac12}}
\newcommand{\fhTakjm}{\fhTak_{j-\frac12}}

\newcommand{\fhjp}{\mathcal{F}^{h}_{j+\frac12}}
\newcommand{\fhjm}{\mathcal{F}^{h}_{j-\frac12}}

\newcommand{\fhsajp}{\mathcal{F}^{h \sa}_{j+\frac12}}
\newcommand{\fhsajm}{\mathcal{F}^{h \sa}_{j-\frac12}}

\newcommand{\fhTajp}{\mathcal{F}^{h \Ta}_{j+\frac12}}
\newcommand{\fhTajm}{\mathcal{F}^{h \Ta}_{j-\frac12}}

\newcommand{\hjdkp}{h_{j}^{n+\frac12,k+1}}
\newcommand{\hjdk}{h_{j}^{n+\frac12,k}}

\newcommand{\sajn}{\sigma_{\alpha,j}^n}
\newcommand{\sajnp}{\sigma_{\alpha,j}^{n+1}}
\newcommand{\sajd}{\sigma_{\alpha,j}^{n+\frac12}}

\newcommand{\sumaN}{\sum_{\alpha=1}^{N}}

\newcommand{\sumja}{\sum_{j=1}^{\alpha}}

\newcommand{\dt}{\frac{\partial}{\partial t}}
\newcommand{\dx}{\frac{\partial}{\partial x}}
\newcommand{\dy}{\frac{\partial}{\partial y}}

\newcommand{\ddt}[1]{\frac{\partial #1}{\partial t}}
\newcommand{\ddx}[1]{\frac{\partial #1}{\partial x}}
\newcommand{\ddy}[1]{\frac{\partial #1}{\partial y}}

\newcommand{\pos}[1]{\left( #1 \right)^+}
\newcommand{\nega}[1]{\left( #1 \right)^-}

\title{Barotropic-Baroclinic Splitting for Multilayer Shallow Water Models with Exchanges}

\author{Nina Aguillon$^{1}$, Sophie H\"ornschemeyer$^{2}$, Jacques Sainte-Marie$^{1}$}
 
\affil{$^{1}$ Sorbonne Universit\'e, CNRS, Inria, Laboratoire J.-L. Lions, F-75005 Paris, France \and $^{2}$ IGPM, RWTH Aachen University, Germany}

\date{}

\pgfplotsset{compat=1.17}

\begin{document}

\maketitle

%%%%%%%%%%%%%%%%%%%

% \tableofcontents
\begin{abstract}    
    This work presents the numerical analysis of a barotropic-baroclinic splitting in a nonlinear multilayer framework with exchanges between the layers in terrain-following coordinates. The splitting is formulated as an exact operator splitting. The barotropic step handles free-surface evolution and depth-averaged velocity via a well-balanced one-layer model, while the baroclinic step manages vertical exchanges between layers and adjusts velocities to their mean values.
    We show that the barotropic-baroclinic splitting preserves total energy conservation and \revised{satisfies} both a discrete maximum principle and a discrete entropy inequality.
    Several numerical experiments are presented showing the gain in computational cost, particularly in low-Froude simulations, with no loss of accuracy. The benefits of using a well-balancing strategy in the barotropic step to preserve the geostrophic equilibrium are inherited \revised{by} the overall scheme.
\end{abstract}

\vskip0.5cm

% \paragraph{Ethics and Integrity Statement} We, the authors, declare that we have no conflicts of interest regarding this research. All findings are based on our original work, and no external data sources were utilized. In line with our commitment to transparency, we will make the code available upon publication to ensure reproducibility of our results. This research was partially funded by Deutsche Forschungsgemeinschaft (DFG grant 320021702 / GRK2326) and the ANR – FRANCE (project NASSMOM ANR-23-CE56-0005-01).

\addcontentsline{toc}{section}{Introduction}
\section*{Introduction}

Ocean global circulation models (OGCM) are numerical cores able to simulate the ocean with atmospheric forcing on the whole planet for very long periods of time. They are used in the CMIP simulations (Coupled Model Intercomparison Project) that inform IPPC reports (Inter-governmental Panel on Climate Change). In such simulations the primary focus is on the evolution of stratification and ocean currents through time, and the space and time scales are enormous. The underlying equations are the free surface Navier Stokes equations with hydrostatic pressure and the Boussinesq assumption. Without any specific adjustment the time step is constrained by the fastest wave speed, namely of the surface gravity waves with order of magnitude of $200\, \mathrm{ms}^{-1}$. The fastest current velocity is of about $3\,\mathrm{ms}^{-1}$.

Such a severe constraint is unacceptable and from the very beginning of global ocean circulation simulation efforts have been made to alleviate this constraint. The old rigid lid approximation~\cite{B69} quickly showed its limits and was replaced either by the implicit surface height method~\cite{DR94} or by the barotropic-baroclinic decomposition~\cite{KSP91} under study in the present work.
The main idea is that the surface gravity waves arise from a vertically integrated (barotropic) two-dimensional system resembling to the shallow water model. The other \revised{phenomena} reflect the deviation from this barotropic state and correspond to a fully three-dimensional system with much slower behavior (baroclinic). We also refer the reader to~\cite{H99} and~\cite{SM05} for different vertical coordinates and to~\cite{DDMLBE19},~\cite{LJWG22} and references therein for more recent developments.

In this paper we revisit the barotropic-baroclinic splitting in the setting of an inviscid terrain-following vertical coordinate ($\sigma$-model) and in the constant density case for simplicity. We propose a nonlinear framework where it reads as an operator splitting. We study the hyperbolic properties of both parts of the splitting, \revised{prove a (discrete) entropy inequality,} and illustrate the ability of the splitting in considerably reducing the computational cost without losing accuracy. This is achieved in the framework of multilayer (or layer-averaged) shallow water models (\cite{ABPSM11a}, \cite{ABPSM11b}, \cite{ABF19} and references therein). Such models are also used for regional applications for which the barotropic-baroclinic splitting can also have an advantage as soon as the physical phenomenon under scrutiny is not about the free surface waves but rather on a slower material time scale.

\revised{The multilayers shallow water models with mass exchanges are obtained from the inviscid Euler equations with the two classical kinematic boundary conditions at the bottom and at the surface. The column of water is divided into $N$ layers, which heights $\ha$ are proportional to the total water depth $h$ with $\ha=\la h$. This choice corresponds to the terrain-following vertical coordinate in ocean models. The evolution of $h$ and of the horizontal velocities $\ua$ and $\va$ in each layer $L_{\alpha}$ is obtained by vertically integrating the Euler equations under the shallow water assumption that the mean of the velocity squared is the square of the mean velocity.}
We obtain the system
\begin{equation}\label{eq:MSW}
    \begin{cases}
        % conservation of mass
        \ddt{\ha} + \dx (\ha \ua) + \dy (\ha v_\alpha)  = \Gap - \Gam,  \\ %\label{eq:haMSW} 
         % hu, middle layers
        \dt (\ha \ua) + \dx \left(\ha \ua^2 \right) + g \ha \ddx{h} + \dy (\ha \ua \va) = -g \ha \ddx{\zb} + \uap\Gap - \uam\Gam,   \\ % \label{eq:huMSW}
        % hv, middle layers
        \dt (\ha \va) + \dx (\ha \ua \va) + \dx \left(\ha \va^2 \right) + g \ha \ddy{h} = -g \ha \ddy{\zb} + \vap\Gap - \vam\Gam,    \\ %\label{eq:hvMSW}
        \dt (\ha \Ta) + \dx (\ha \ua \Ta) +  \dy (\ha \va \Ta) = \Tap\Gap- \Tam\Gam %\label{eq:Ta}
    \end{cases}
\end{equation}
\revised{where, $\Gapm$ are  mass exchange terms between the layers, $u_{\alpha \pm \frac12}$ and $v_{\alpha \pm \frac12}$ are the interface velocities.} 
A no-flux condition at the surface and bottom $G_{\frac12}= G_{N+\frac12}=0$ is imposed. The notation \revised{is} gathered \revised{in} Figure~\ref{fig:notations_mc}.
\revised{We consider the simplest possible case by omitting forcings and diffusive terms which are discussed later on in Section~\ref{S:testcases}. Most importantly, the density of the fluid is constant both along the vertical and the horizontal direction, thus the pressure terms only depend on the height of the water column. We included a passive tracer $T$ which is simply advected by the flow. The interface tracer concentrations are denoted by $T_{\alpha \pm \frac12}$. This is useful to study some important numerical analysis properties on the scheme such as maximum principles, and it also paves the way to the varying density case since the temperature and the salinity are also advected by the flow.}

The barotropic-baroclinic splitting found in the ocean community corresponds to the separation of two different types of waves: the fast surface gravity waves that arise in a depth-independent shallow water model and the adjustment with respect to this barotropic evolution in a fully baroclinic model with much slower wave speed. 
Here, we adopt a presentation very close to~\cite{LJWG22} in the framework of multilayers models~\eqref{eq:MSW}, adapting it to the case of a constant density fluid. 
Firstly, the summation over the layers $\alpha = 1, \dots, N$ yields the barotropic system with baroclinic forcing
\begin{equation} \label{eq:oceanoBT}
    \begin{cases} 
        % conservation of mass
        \ddt{h} + \dx (h \baru) + \dy (h \bar v) = 0, \\ 
         % hu
        \dt(h \baru) + \dx \left(h \baru^2 +  \frac{g}{2}h^2 \right) + \dy (h \baru \barv) = -gh \ddx{\zb} + \dx \big(h \baru^2- h \overline{u^2} \big) + \dy (h \baru \bar v - h \overline{uv}),\\
        % hv
        \dt (h \barv) + \dx (h \baru \barv) + \dy \left(h \barv^2 + \frac{g}{2}h^2 \right) = -gh \ddy{\zb} + \dx (h \baru \barv-  h \overline{uv}) + \dy \big(h \barv^2 - h \overline{v^{2}} \big)\\
    \end{cases} 
\end{equation}
where we denote 
$\overline{u^{2}} := \frac{1}{h} \sumaN \ha \ua^{2}$, $\overline{v^{2}} := \frac{1}{h} \sumaN \ha \va^{2}$ and $\overline{u v} := \frac{1}{h} \sumaN \ha \ua \va$. \revised{This system is not closed since $\overline{u^{2}} \neq \overline{u}^{2}$ and $\overline{u v} \neq \overline{u} \, \overline{v}$.}

Secondly, some straightforward manipulations yield the baroclinic equation
\begin{equation} \label{eq:oceanoBC}
    \begin{cases}
        \dt (\ua-\baru) + \ua \ddx \ua + \va \ddy \ua + \frac{\baru}{h} \left( \dx (h \baru) + \dy (h \barv) \right) - \frac{1}{h} \left( \dx (h \overline{u^{2}}) + \dy (h \overline{u v}) \right)\\  
        \qquad \qquad =  \frac{1}{\ha}\big((\uap-\ua)\Gap - (\uam-\ua)\Gam \big), \\
        \dt (\va-\barv) + \ua \ddx \va + \va \ddy \va + \frac{\barv}{h} \left( \dx (h \baru) + \dy (h \barv) \right) - \frac{1}{h} \left( \dx (h \overline{uv}) + \dy (h \overline{v^{2}}) \right) \\
        \qquad \qquad =  \frac{1}{\ha} \big( (\vap-\va)\Gap - (\vam-\va)\Gam \big).
    \end{cases}
\end{equation}
This latter system describes the evolution of the deviation of the individual velocities from the mean vertical averaged velocities. It is a fully \revised{three-dimensional} system but no pressure or bathymetry terms appear, thus it can be solved with a large time step denoted by $\Delta t$.  In parallel, the barotropic system~\eqref{eq:oceanoBT} is used to predict the evolution of $h$ and $(\baru, \bar v)$ by \revised{performing} many small time steps $\delta t$ constrained by the velocity of the gravity waves until the large baroclinic time step $\Delta t$ is reached. The left-hand side is \revised{frozen during the large timestep $\Delta t$} at its initial value and~\eqref{eq:oceanoBT} reads as the \revised{two-dimensional} shallow water equations with source term. Thus, it remains computationally affordable despite the many time \revised{steps}.

A major issue is that at time $\Delta t$ one has to correct the update of the individual velocities $(\ua, \va)$ obtained in the baroclinic step~\eqref{eq:oceanoBC} so that their mean value $(\baru, \barv)$ coincides with the one obtained in the barotropic step~\eqref{eq:oceanoBT}. In order to stabilize the splitting OGCMs also incorporate some time filtering in the barotropic step. The mean velocity $(\baru, \barv)$ predicted at time $\Delta t$ is not taken as the final substep, but as an average of many substeps. Two particular filters are described in~\cite{DDMLBE19}: one uses all the substeps in the interval $[0, 2 \Delta t]$ and the other the ones between $[\Delta t/2, 3 \Delta t/2]$. See also~\cite{H96} for isopycnal models.

In this paper, we reformulate the barotropic-baroclinic splitting as an operator splitting in Section~\ref{S:multilayer}. The evolution in time in~\eqref{eq:MSW} is decomposed into two parts in the same spirit of the decomposition~\eqref{eq:oceanoBT}-\eqref{eq:oceanoBC}
\begin{itemize}
    \item a barotropic part gathering the transport at the mean speed $(\baru, \barv)$ and the pressure terms;
    \item a baroclinic part gathering the deviation from the mean velocity and the vertical exchange terms.
\end{itemize}
The sum of the two contributions exactly gives the source term and \revised{spatial} derivatives of~\eqref{eq:MSW}. Similarly to~\eqref{eq:oceanoBT}-\eqref{eq:oceanoBC}, the barotropic part is essentially depth-independent and contains the surface gravity waves, while the baroclinic part is fully \revised{three-dimensional} but has \revised{a} much slower dynamic. The main differences are that first, the barotropic step does not have any baroclinic forcing but encompasses a redistribution of the evolution of the mean velocities over the individual velocities and second the mean velocity also evolves in the baroclinic step. \revised{As a consequence, the need of reconciliation between the two steps disappears. }

The numerical treatment presented in Section~\ref{S:scheme} is classically based on a Lie splitting (we only present first-order schemes in this article) with a subcycling procedure in the fast barotropic step, see~\cite{DL11} and~\cite{AB03} for a presentation on such techniques. The numerical treatment of the exchange terms in the baroclinic step is similar to~\cite{ABS18}. We exhibit the time step restrictions of all the parts of the scheme and prove that the total energy decreases at each time step. In other words, the scheme fulfills a discrete entropy inequality, an important property that ensures the convergence towards physical solutions of the scheme (provided that the sequence of approximation does have a limit, Lax theorem).

In Section~\ref{S:testcases}, we present a variety of testcases to illustrate the good properties of the scheme. As expected, the computational cost are reduced drastically for low Froude number simulations. 
Finally, we take full advantage of proximity between the barotropic system and the simple one-layer shallow water equations to implement state of the art well-balanced schemes for this part of the system. As a result the overall scheme preserves the lake-at-rest equilibrium and the positivity of the water height with almost no additional work (see~\cite{ABBKP04},~\cite{CN17}). We also show that recent works to better maintain the geostrophic equilibrium such as~\cite{GCGP24} and~\cite{ADDGNP21} provide more accurate results for coarse-mesh long-time integration in the multilayer environment in the presence of the Coriolis force.

\section{Barotropic-baroclinic splitting at the continuous level}

This section is devoted to the presentation and the analysis of a splitting operator of the layer-averaged Euler equations~\eqref{eq:MSW} of~\cite{ABPSM11a} that shares the same properties as the barotropic-baroclinic splitting of the ocean model~\eqref{eq:oceanoBT}-\eqref{eq:oceanoBC} described in the introduction. We present the procedure on the 1D $(x,z)$-model to alleviate the notation, see Figure~\ref{fig:notations_mc}.
\begin{equation}\label{eq:MSW1D}
    \begin{cases}
        % conservation of mass
        \ddt{\ha} + \dx (\ha \ua)  = \Gap- \Gam,  \\ %\label{eq:haMSW} 
         % hu, middle layers
        \dt (\ha \ua) + \dx \left(\ha \ua^{2} \right) + g \ha  \ddx{h} = -g \ha \ddx{\zb} + \uap\Gap - \uam\Gam,   \\ % \label{eq:huMSW}
        \dt (\ha \Ta) + \dx (\ha \ua \Ta) = \Tap\Gap- \Tam\Gam. %\label{eq:Ta}
    \end{cases}
\end{equation}
The addition of a second horizontal direction $y$ brings no further difficulties. 

Before presenting the splitting let us give some more details on model~\eqref{eq:MSW1D}. Written in this form the exchange terms are unknown and the system is unclosed with $4N-1$ unknowns $(\ha, \ua, \Ta)_{1 \leq \alpha \leq N}$ and $(\Gap)_{1 \leq \alpha \leq N-1}$.

There are two families of multilayer models. One can consider that there are no exchange terms. In that case we obtain a family of shallow water equations stacked on each other and coupled through the pressure term. This model is ill-posed in the present constant density case, but is useful when the stratification is strong with respect to the velocity shear and corresponds to the so called ``isopycnal models'' (\cite[Chapter 3]{Vallis} or~\cite{CDV17}). A second possibility is to enforce a constraint on the water heights $\ha$. For example the inner separations between the layers $z_{\alpha+\frac12}$, $1 \leq \alpha \leq N-1$, are horizontal ($z$-coordinates in ocean models), or all the layers are a given fraction of the total water height (``terrain-following models'' or $\sigma$-coordinates).

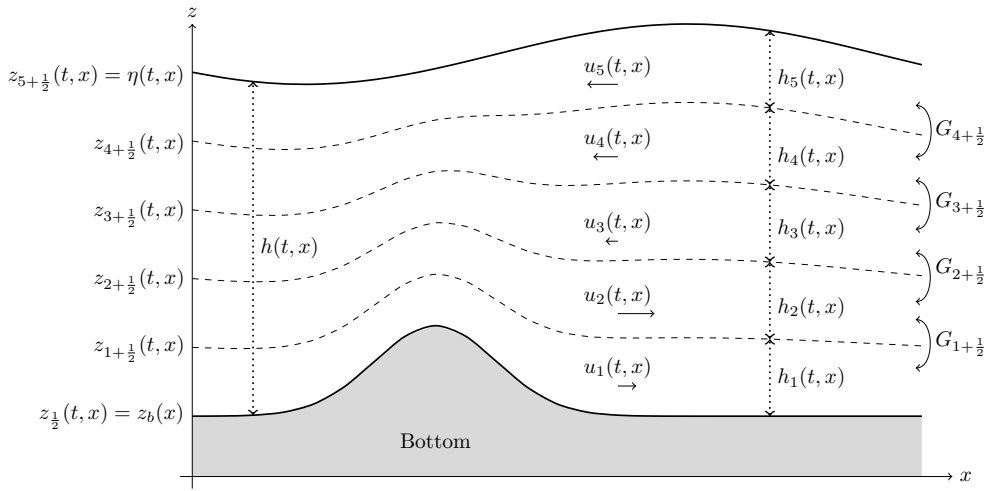
\begin{figure}[hbtp]
    \centering
    \scalebox{0.8}{
    \begin{tikzpicture}[domain=0:12]

      % bottom topography and free surface
      \fill [gray!30, domain=0:12] plot ({\x},{1.5*exp(-1/2*(\x-4)^2)}) |- (0,-1);
      \draw (4,-0.4) node {Bottom}; 
      \draw[smooth, thick] plot (\x,{0.5*sin(1/2*(\x-5) r) + 6});
      \draw (0,5.6) node[left] {$z_{5+\frac12}(t,x) = \eta(t,x)$};
      
      \draw[smooth, thick] node[left] {$z_{\frac12}(t,x) = z_b(x)$} plot (\x,{1.5*exp(-1/2*(\x-4)^2)});
      
      % intermediate layers
      \foreach \k in {1,...,4}
        \draw[smooth, dashed] plot(\x,{\k/5*(0.5*sin(1/2*(\x-5) r) + 6 - 1.5*exp(-1/2*(\x-4)^2)) + 1.5*exp(-1/2*(\x-4)^2)});
        
      \draw (0,1.12) node[left] {$z_{1+\frac12}(t,x)$};
      \draw (0,2.24) node[left] {$z_{2+\frac12}(t,x)$};
      \draw (0,3.36) node[left] {$z_{3+\frac12}(t,x)$};
      \draw (0,4.48) node[left] {$z_{4+\frac12}(t,x)$};
        
      % total water height h
      \draw[<->, dotted, thick] (1,0.016663494807363458) -- (1,5.5453512865871595);
      \draw (1,2.8) node[right] {$h(t,x)$};
      
      % layers h_alpha
      \draw[<->, dotted, thick] (9.5,0) -- (9.5,1.2778076436382126);
      \draw (9.5,0.63) node[right] {$h_1(t,x)$};
      \draw[<->, dotted, thick] (9.5,1.2778076436382126) -- (9.5,2.55561488233965);
      \draw (9.5,1.8) node[right] {$h_2(t,x)$};
      \draw[<->, dotted, thick] (9.5,2.55561488233965) -- (9.5,3.8334221210410875);
      \draw (9.5,3.1) node[right] {$h_3(t,x)$};
      \draw[<->, dotted, thick] (9.5,3.8334221210410875) -- (9.5,5.1112293597425245);
      \draw (9.5,4.3) node[right] {$h_4(t,x)$};
      \draw[<->, dotted, thick] (9.5,5.1112293597425245) -- (9.5,6.389036598444);
      \draw (9.5,5.6) node[right] {$h_5(t,x)$};
      
      % individual velocities u_alpha
      \draw[->] (7,0.5) node[above] {$u_1(t,x)$} -- (7.3,0.5);
      \draw[->] (7,1.7) node[above] {$u_2(t,x)$} -- (7.6,1.7);
      \draw[<-] (6.8,2.9) -- (7,2.9) node[above] {$u_3(t,x)$};
      \draw[<-] (6.6,4.3) -- (7,4.3) node[above] {$u_4(t,x)$};
      \draw[<-] (6.5,5.5) -- (7,5.5) node[above] {$u_5(t,x)$};
      
      % mass exchanges
      \draw[<->, bend angle=90, bend right] (11.9,0.8) to (11.9,1.6);
      \draw (12.1,1.2) node[right] {$G_{1+\frac12}$};
      \draw[<->, bend angle=90, bend right] (11.9,1.9) to (11.9,2.7);
      \draw (12.1,2.4) node[right] {$G_{2+\frac12}$};
      \draw[<->, bend angle=90, bend right] (11.9,3.1) to (11.9,3.9);
      \draw (12.1,3.5) node[right] {$G_{3+\frac12}$};
      \draw[<->, bend angle=90, bend right] (11.9,4.3) to (11.9,5.15);
      \draw (12.1,4.7) node[right] {$G_{4+\frac12}$};
      
      % axis
      \draw[->] (-0.2,-1) -- (12.5,-1) node[right] {$x$};
      \draw[->] (0,-1.2) -- (0,6.5) node[above] {$z$};
      
    \end{tikzpicture}}
    \caption{Notations for the multilayer approach.}
    \label{fig:notations_mc}
\end{figure}

We focus on the latter family and place our footsteps in the family of multilayer models developed in~\cite{ABPSM11a},~\cite{ABPSM11b},~\cite{ABF19} and references therein. The notation is gathered in Figure~\ref{fig:notations_mc}. The weights $(\la)_{\alpha \in \{1, \dots, N \} }$ with $\sumaN \la =1$ and $\la>0$ are fixed once and for all.  In such models with mass exchanges, the unknowns in system~\eqref{eq:MSW1D} are the total water height $h$ and the individual velocities $(\ua, \va)_{\alpha \in \{1, \dots, N\}}$. The individual water heights $\ha$ verify the constraint $\ha= \la h$. The evolution of the total water height is given by the summation  over all the layers of the equation on~$\ha$ in~\eqref{eq:MSW1D}, namely
\begin{equation}
    \ddt{h} + \ddx{(h \baru)} = 0. \label{eq:h} %+  \frac{\partial (h \barv) }{\partial x } 
\end{equation}
The exchange terms $\Gapm$ are Lagrange multipliers enforcing the constraints $\ha= \la h$ at all time. In other words the individual equation on $\ha$ should be understood as a definition of the exchange terms $\Gapm$. Multiplying equation~\eqref{eq:h} by $\la$ and substracting it from the equation on $\ha$ in~\eqref{eq:MSW1D} to cancel the time derivative we obtain
\begin{equation} \label{eq:G}
    \Gap-\Gam = \la \dx  (h (\ua-\baru)). %+ \frac{\partial }{\partial y }  (h (\va-\barv))
\end{equation}
The final closed system reads
\begin{equation} \label{eq:MSW1D-closed}
    \begin{cases}
        % conservation of mass
        \ddt{h} + \dx (h \baru) = 0,  \\ 
         % hu
        \dt (h \ua) + \dx \left(h \ua^{2} +  \frac{g}{2}h^{2}\right) = -gh \ddx{\zb} + \frac{1}{\la} \big( \uap \Gap - \uam \Gam \big) ,   \\ 
        \dt (h \Ta) + \dx \left(h \ua \Ta \right) = \frac{1}{\la} \big( \Tap \Gap - \Tam \Gam \big),   \\ 
         \Gap  =  \sumja \lj \dx (h (\uj-\baru)).
    \end{cases}
\end{equation}
The interface velocities $\uap$ and interface tracers $\Tap$ are defined using an upwinding strategy
\begin{equation} \label{eq:interu}
    \uap = 
    \begin{cases}
        \ua, & \mbox{if } \;\Gap \leq 0 \\
        u_{\alpha+1}, & \mbox{if } \;\Gap > 0
    \end{cases}
    \quad \text{ and } \quad
    \Tap =
    \begin{cases}
        \Ta, & \mbox{if } \;\Gap \leq 0\\
        T_{\alpha+1}, & \mbox{if } \;\Gap > 0
    \end{cases}.
\end{equation}
We now propose a barotropic-baroclinic spitting for the set of equations~\eqref{eq:MSW1D-closed}.
\subsection{Barotropic-baroclinic splitting for the multilayer shallow water model with mass exchanges} \label{S:multilayer}

\begin{defn}
    The barotropic step corresponds to the  advection of $h$, $\ua$ and $\Ta$ at the mean velocity~$\baru$ together with the pressure force and the topography source term. It writes 
    \begin{equation}  \label{eq:BT-closed}
        \begin{cases} 
            \ddt{h} + \dx (h \baru) = 0, \\ 
            \dt (h \ua) + \dx \left(h \ua \baru + \frac{g}{2} h^{2} \right) = -g h \ddx{\zb},  \\
            \dt (h\Ta) + \dx (h\baru \Ta) = 0. 
        \end{cases}
    \end{equation}
    The baroclinic step gathers the nonlinear adjustment around the mean vertical velocity $\baru$ and the exchange terms to maintain the constraint $\ha=\la h$.
    
    It  writes
    \begin{equation} \label{eq:BC}
        \begin{cases}
            \ddt h = 0,  \\
            \dt (h \ua) + \dx \left(h \ua (\ua -\baru) \right) = \frac{1}{\la} \big( \uap \Gap - \uam \Gam \big), \\
            \dt (h \Ta) + \dx \left(h \Ta (\ua -\baru) \right) = \frac{1}{\la} \big( \Tap \Gap - \Tam \Gam \big), \\
            \Gap =  \sumja \lj \dx (h (\uj-\baru))
        \end{cases}
    \end{equation}
    where the definition of the interface velocities is given by~\eqref{eq:interu}. 
\end{defn}
An important feature of this splitting is that it is an operator splitting.
\begin{proposition}
    Write the multilayer operator~\eqref{eq:MSW1D-closed} as
    $$ \partial_t \begin{pmatrix} h \\ h \ua \\ h \Ta \end{pmatrix} = \operatorname{ML} \begin{pmatrix} h \\ h \ua \\ h \Ta \end{pmatrix} $$
    and similarly the barotropic operator~\eqref{eq:BT-closed} and the baroclinic operator~\eqref{eq:BC} respectively as
    $$ \partial_t \begin{pmatrix} h \\ h \ua \\ h \Ta \end{pmatrix} = \operatorname{BT} \begin{pmatrix} h \\ h \ua \\ h \Ta \end{pmatrix} \ \text{ and } \  \partial_t \begin{pmatrix} h \\ h \ua \\ h \Ta \end{pmatrix} = \operatorname{BC} \begin{pmatrix} h \\ h \ua \\ h \Ta \end{pmatrix}. $$
    Then $\operatorname{ML} = \operatorname{BT} + \operatorname{BC}$.
\end{proposition}
\begin{proof} 
    In the original model~\eqref{eq:MSW1D-closed} the first term $\ua$ in the equation on $h \ua$ is written as $\ua= \baru + (\ua-\baru)$. Then the mean term in $\baru$ yields the barotropic part~\eqref{eq:BT-closed} while the deviation term $\ua-\baru$ yields the baroclinic part~\eqref{eq:BC}.
\end{proof}
We now emphasize the proximity of the barotropic step~\eqref{eq:BT-closed} with the one layer shallow water model.
\begin{proposition}
    System~\eqref{eq:BT-closed} in variables $(h, h \ua, h \Ta)_{\alpha \in \{1, \ldots, N\}}$ is equivalent to the system in variables $(h, h \baru, (h \sa)_{\alpha \in \{1, \ldots, N-1\}}, (h \Ta)_{\alpha \in \{1, \ldots, N\}})$ with $\sa= \ua-\baru$:
    \begin{equation}
        \begin{cases}{\label{eq:BT}}
            \ddt{h} + \dx (h \baru) = 0, \\ %\label{eq:hBT}
            \dt (h \baru) + \dx \left(h \baru^2 +  \frac{g}{2}h^2\right) = -gh \ddx{\zb}, \\ %\label{eq:huBT}
            \dt (h \sa) + \dx \left(h \baru \sa \right)= 0, \\%\label{eq:hu-ubarBT}  
            \dt (h \Ta) + \dx \left(h \baru \Ta \right)= 0.
        \end{cases}
    \end{equation}
    The evolution of $h \sigma_{N}$  is similar as in the other layers, but this quantity is deduced from the others as $\sumaN \la h \sa = 0$.
\end{proposition} \label{prop:rewriting_BT}

\begin{proof}
    The equation on $h \baru$ is obtained by summing the equations on $h \ua$ weighted by $\la$. The substraction of the individual equation on  $h \ua$ with the new equation on $h \baru$ gives the equation on $h \sa$. 
\end{proof}

\begin{remark}
Another possibility is to treat the tracer in a similar fashion, separating the evolution of its mean value $\barT = \sumaN\la \Ta$ and of the deviations $\ta=\Ta-\barT$:
    \begin{equation*}
        \begin{cases}%{\label{eq:BT}}
            \ddt{h} + \dx (h \baru) = 0, \\ %\label{eq:hBT}
            \dt (h \baru) + \dx \left(h \baru^2 +  \frac{g}{2}h^2\right) = -gh \ddx{\zb}, \\ %\label{eq:huBT}
            \dt (h \barT) + \dx \left(h \baru \barT \right) = 0, \\ 
            \dt (h \sa) + \dx \left(h \baru \sa \right)= 0, \\%\label{eq:hu-ubarBT}  
            \dt (h \ta) + \dx \left(h \baru \ta \right)= 0.
        \end{cases}
    \end{equation*}
This form could be useful if $T$ was an active tracer, meaning that the pressure would depend on the values of $(\Ta)_{\alpha \in \{1, \ldots, N\}}$. The extension to this case is left for future work.
\end{remark}

Let us take a closer look on the individual equation on $\ha$.
\begin{remark}
    For the barotropic system~\eqref{eq:BT-closed} the individual equation on $\ha$ is given by
    $$ \ddt{\ha} + \dx (\ha \baru) = 0. $$
    Together with the equation on $h$ it holds
    $$ \dt (\ha- \la h) + \dx ((\ha- \la h) \baru) = 0 $$
    and it is clear that if the constraint $ \ha= \la h$ initially holds this property is preserved through time. This explains why there are no exchange terms in this step.
    
    For the baroclinic step the definition of the exchange terms yields $\Gap - \Gam = \la \dx (h (\ua-\baru))$ and
   $$ \ddt{\ha} + \dx \left(\ha (\ua -\baru) \right) = \Gap - \Gam.$$
    Once again the exchange terms can be seen as Lagrange multipliers enforcing the constraints $\ha=\la h$.
\end{remark}

We now turn to the classical question of energy conservation or dissipation for both parts of the splitting. 
In the unsplit multilayer system~\eqref{eq:MSW1D-closed}-\eqref{eq:interu} the total energy $E=   \frac{g h^{2}}{2} + g h \zb + \sumaN \frac{\ha \ua^{2}}{2}$ verifies the energy inequality~\cite{ABPSM11a}
\begin{equation}\label{eq:entropyMSW}
    \dt E+ \dx \left( g h^{2} \baru + g h \zb \baru + \sumaN \frac{\ha \ua^{3} }{2} \right) \leq 0.
%    = - \frac12 \sumaN (\ua-u_{\alpha+1})^{2}|\Gap|. 
\end{equation}
\begin{proposition} \label{P:energyequality} 
    The barotropic-baroclinic splitting~\eqref{eq:BT-closed}-\eqref{eq:BC} is associated with a splitting of~\eqref{eq:entropyMSW}: the smooth solutions of the barotropic step~\eqref{eq:BT-closed} verify the energy equality
    \begin{equation} \label{eq:EIBT}
        \dt E+ \dx \left( g h^{2} \baru + g h \zb \baru + \sumaN \dfrac{\ha \ua^{2} \baru}{2} \right) = 0,
    \end{equation}
    while smooth solutions of the baroclinic step~\eqref{eq:BC} with interface velocities~\eqref{eq:interu} verify the energy inequality
    \begin{equation} \label{eq:EIBC}
        \dt E + \dx \left( \sumaN \frac{\ha \ua^{2} (\ua-\baru)}{2} \right) \leq 0. %= - \sumaN \frac12(\ua-u_{\alpha+1})^{2}|\Gap|.
    \end{equation}
\end{proposition} \label{prop:entropy}
\begin{proof}
    The proof can be found in Section~\ref{sec:continuous_Entropy_Proofs}.
\end{proof}

\subsection{Two hyberbolicity results}
In this section, we investigate the hyperbolicity of the systems~\eqref{eq:BT-closed} and~\eqref{eq:BC}. We ignore the passive tracer $T$ and focus on the nonlinear part. \revised{The hyperbolicity of multilayer models is a difficult question in general due to the non-conservative exchange terms. 
It is known from the seminal paper~\cite{ABPSM11a} that the two-layer shallow water system is hyperbolic for any choice of interface velocity $u_{\frac12}$ lying between $u_1$ and $u_2$. For the general $N$-layer system~\eqref{eq:MSW1D} with upwind interface velocities~\eqref{eq:interu} numerical investigations indicate that there are always two extremal eigenvalues resembling $\baru \pm \sqrt{gh}$ (with a perturbation due to the shear). However "in the middle" complex eigenvalues arise in general, but no effect on the stability of numerical schemes has ever been detected. The latest work devoted to the subject is~\cite{DMEHGGSM25}, where the structure of the spectrum is described in detail, both at the continuous level of the free-surface Euler equations (or Benney model) in Lagrangian formulation and for multilayer vertical discretizations. In this section, we study the hyperbolicity of both parts of the splitting, starting with the barotropic part~\eqref{eq:BT-closed} that contains the two large gravitational eigenvalues.}
\begin{proposition}  
    The barotropic system~\eqref{eq:BT-closed} with flat topography $\zb=0$ is hyberbolic with eigenvalues $\baru - \sqrt{gh}$, $\bar u$ with $N-1$ multiplicity and $\baru+\sqrt{gh}$.
\end{proposition}
\begin{proof} 
    We study the hyperbolicity in the variables $(h, \baru, \sigma_{1}, \cdots, \sigma_{N-1})$, see~\eqref{eq:BT}. The quasilinear form of this system is
    $$
    \partial_{t} \begin{pmatrix} h \\ \baru \\ \sigma_{1} \\ \sigma_{2} \\ \vdots \\ \sigma_{N-1} \end{pmatrix}
    + \begin{pmatrix} \bar u & h & 0 & \dots & \dots & 0 \\ g & \bar u & 0 & \dots & \dots & 0 \\ 0 & 0 & \baru & 0 & \dots & 0 \\ \vdots & \vdots & 0 & \baru & \ddots& \vdots \\ \vdots & \vdots & \vdots & \ddots & \ddots & 0 \\ 0 & 0 & 0 & \dots & 0 & \baru  \end{pmatrix} \partial_{x} \begin{pmatrix} h \\ \baru \\ \sigma_{1} \\ \sigma_{2} \\ \vdots \\ \sigma_{N-1} \end{pmatrix} =  \begin{pmatrix} 0 \\ 0 \\ 0 \\ \vdots\\ \vdots \\ 0 \end{pmatrix}
    $$
    and the result immediately follows.
\end{proof}
The study of the hyperbolicity of the baroclinic system~\eqref{eq:BC} shares many similarities with the one of the original system~\eqref{eq:MSW1D}. 

The following proposition concerns the hyperbolicity of the baroclinic part \emph{without mass exchange}. At first glance this system seems somehow artificial but:
\begin{itemize}
    \item it is very close to the multilayer semi-Lagrangian model (5.2) in~\cite{DMEHGGSM25},
    \item this system is solved in the numerical approximation of the baroclinic step.
\end{itemize}
\revised{Even though the eigenvalues are not explicit, this proposition shows that they are small and controlled by the material velocities $\ua$. This point will be numerically crucial: the fully three-dimensional baroclinic part~\eqref{eq:BC} will be updated with a much larger time step than the barotropic step~\eqref{eq:BT-closed}.}
\begin{proposition} \label{prop:eigenvalues_prediction}
    Consider the following system with the $2N$ unknowns $(\ha, \ua)_{\alpha \in \{1,\ldots,N\}}$ 
    \begin{equation} \label{eq:BCnoME}
        \begin{cases}
            \ddt{\ha} + \dx (\ha (\ua - \baru)) = 0, \\
            \dt (\ha \ua) + \dx \left( \ha \ua \left(\ua-\baru \right) \right) = 0.
        \end{cases}
    \end{equation}
    The total water height $h$ and the mean velocity $\baru$ are defined using the other unknowns with $h=\sumaN \ha$ and $\baru = \frac{1}{h} \sumaN \ha \ua$.% but the individual water heights $\ha$ vary without the constraints $\ha= \la h$ in the absence of exchange terms.
    
    Suppose that all the velocity deviations $\ua-\bar u$ are different, that none of them vanishes and that $\sumaN \frac{\ha}{\ua - \baru} \neq 0$. Then system~\eqref{eq:BCnoME} is strictly hyperbolic with the $2N$ distinct eigenvalues 
    $$  
    \left\{ u_1 - \baru, \dots, u_N - \baru \right\} \cup \left\{ \lambda \in \R, \ h = \sumaN \dfrac{(\ua - \baru ) \ha }{\ua - \baru - \lambda} \right\}.
    $$
    Moreover, all the eigenvalues belong to the interval $[\min_{\alpha} \ua - \baru, \max_{\alpha} \ua-\baru]$.
\end{proposition}
%\begin{remark}
%The condition resembles to the condition obtained in~\cite[Proposition 5.5]{DSMEHGGSM23} where the authors study the hyperbolicity of a semi-Lagrangian formulation of the hydrostatic free-surface Euler system.
%\end{remark}
\begin{proof}
    Combining the two equations of~\eqref{eq:BCnoME}, for regular solutions we obtain the transport equation on $\ua$, namely
    \begin{equation} %\label{eq:BCtransportu}
        \ddt{\ua} + (\ua-\baru) \ddx{\ua} = 0. 
    \end{equation}
    We now work on the equation on $\ha$. With the chain rule and $\baru = \frac{\sum_{k=1}^N h_k u_k}{\sum_{k=1}^N h_k}$ we get
    $$ \begin{aligned}
    \frac{\partial}{\partial x } (\ha (\ua- \baru)) &= (\ua- \baru) \frac{\partial \ha}{\partial x } + \ha  \frac{\partial \ua}{\partial x } - \ha \frac{\partial }{\partial x } \left( \dfrac{\sum_{k=1}^N h_k u_k}{\sum_{k=1}^N h_k} \right) \\
    &=  (\ua- \baru) \frac{\partial \ha}{\partial x } + \ha  \frac{\partial \ua}{\partial x } - \dfrac{\ha}{h} \left( \sum_{k=1}^N (u_k - \baru) \frac{\partial h_k}{\partial x }  \right) - \dfrac{\ha}{h} \left( \sum_{k=1}^N h_k \frac{\partial u_k}{\partial x } \right).
    \end{aligned} $$
    We define $H = (h_1, \dots, h_N)^T$, $U = (u_1, \dots, u_N)^T$, $\sa = \ua - \baru$, $\Sigma = (\sigma_1, \dots, \sigma_N)^T$, $D_\Sigma  = \operatorname{diag}(\Sigma)$ and $L = (\ell_1, \dots, \ell_N)^T= \frac{1}{h}(h_{1}, \dots, h_{N})^{T}$. 
    Then the quasilinear linear form of~\eqref{eq:BCnoME} is
    $$
    \dt \left(\begin{array}{c} H \\ \hline  U \end{array} \right) 
    + 
    \left( \begin{array}{c|c}
        D_\Sigma - L \Sigma^T & h L L^T \\ \hline
        0 & D_\Sigma 
    \end{array} \right)
    \dx \left( \begin{array}{c} H \\ \hline  U \end{array} \right) 
    = 
    \dx \left( \begin{array}{c} 0 \\ \hline  0 \end{array}\right).
    $$
    All the $\sa$ are eigenvalues of the matrix. We look for the values of $\lambda \in \R$ such that $\revised{D_{\Sigma}} -L \Sigma^{T} - \lambda I $ is non invertible. As long as $\lambda \neq \sigma_{k}$ the matrix determinant lemma yields
    $$ \det(\revised{D_{\Sigma}} - \lambda I -L \Sigma^{T})= (1- \Sigma^T (\revised{D_{\Sigma}}-\lambda I)^{-1} L) \det(\revised{D_{\Sigma}}-\lambda I). $$
    We now prove that the following function has $N$ roots, denoted by $\lambda_{1}< \dots < \lambda_{N}$
    $$ 
     \varphi(\lambda)= 1- \Sigma^{T} (\revised{D_{\Sigma}}-\lambda I)^{-1} L
     = 1 - \sumaN\frac{\sa \ell_{\alpha}}{\sa- \lambda}.
    $$
    The fact that $\sumaN \ha \sa = 0$ implies the existence of \revised{at least one positive and one negative} $\sa$. We order the velocity deviations: for some $k \in \{1, \dots, N-1\}$ and some permutation $(p_{1}, \dots, p_{N})$ of $\{1, \dots, N\}$, 
    $$ \sigma_{p_{1}}< \cdots < \sigma_{p_{k}}<0< \sigma_{p_{k+1}}< \dots < \sigma_{p_N}.$$
    The limits
    $$ \lim_{\lambda \uparrow \sigma_{p_{i}}} \varphi(\lambda) = 
    \begin{cases}
    + \infty & \text{ if } i \leq  k\\
    - \infty & \text{ if }  i \geq  k+1
    \end{cases},
    \ 
     \lim_{\lambda \downarrow \sigma_{p_{i}}} \varphi(\lambda) = 
    \begin{cases}
    - \infty & \text{ if } i \leq  k\\
    + \infty & \text{ if }  i \geq k+1
    \end{cases}
    \ \text{ and } \
    \lim_{\lambda \to \pm \infty} \varphi(\lambda) = 1
    $$
    imply the existence of a root in between each $N-2$ consecutive pairs $(\sigma_{p_{1}}, \sigma_{p_{2}})$, \dots , $(\sigma_{p_{k-1}}, \sigma_{p_{k}})$, $(\sigma_{p_{k+1}}, \sigma_{p_{k+2}})$, \dots , $(\sigma_{p_{N-1}}, \sigma_{p_{N}})$ sharing the same sign. On the interval $[\sigma_{p_{k}}, \sigma_{p_{k+1}}]$ the function $\varphi$ has two roots. Indeed its limits on the boundaries are $-\infty$ while at $\lambda=0$ the function $\varphi$ vanishes with derivative $\varphi'(0)= \sumaN \frac{\ell_{\alpha}}{\sa}= \frac{1}{h}\sumaN \frac{\ha}{\ua-\baru} \neq 0$. Alltogether the system has $2N$ distinct eigenvalues
    $$  \sigma_{p_{1}} < \lambda_1 < \sigma_{p_{2}} < \dots <\sigma_{p_{k}} < \lambda_k< \lambda_{k+1}  < \sigma_{p_{k+1}}< \lambda_{k+2} < \sigma_{p_{k+2}} < \dots < \lambda_{N} < \sigma_{p_{N}} .$$
\end{proof}
\begin{remark}
    If $\ua-\baru=0$ for some $\alpha$ or if $\sumaN \frac{\ha}{\ua-\baru}=0$ the system is no longer strictly hyperbolic in general. The most obvious examples are the case  $\ua-\baru=0$ for all $\alpha$ with the apparition of the Jordan block $hLL^T$. In the case $N=2$, $u_1=-u$, $u_2=u$, $h_1=h_2$ the matrix reduces to
    $$ \begin{pmatrix} -\frac{u}{2} &-\frac{u}{2} & \frac{h}{4}& \frac{h}{4}\\
    \frac{u}{2}& \frac{u}{2} &  \frac{h}{4}& \frac{h}{4}\\
    0 & 0 & -u & u \\
    0 & 0 & 0 & u\end{pmatrix}$$
and $0$ is a double eigenvalue with only one eigenvector.
\end{remark}

%%%%%%%%%%%%%%%%%%%%%%%%%%%%%%
\section{Numerical scheme} \label{S:scheme}
%%%%%%%%%%%%%%%%%%%%%%%%%%%%%%

We propose a numerical discretization of the multilayer system~\eqref{eq:MSW1D} based on the barotropic-baroclinic splitting~\eqref{eq:BC}-\eqref{eq:BT} presented above. We use a finite volume framework and discretize the computational domain by a grid with equidistant distributed nodes $x_j$. We denote by $C_j = [x_{j-\frac12}, x_{j + \frac12}]$ the cell of length $\Delta x_j = x_{j+\frac12} - x_{j-\frac12}$ with $x_{j+\frac12} = \frac{x_j + x_{j+1}}{2}$. 
At the beginning of the time step the quantities \revised{$(h,\ha\ua, \ha\Ta)_j^n$} at time $t^{n}$ are known.
They are updated to time $t^{n+1}=t^{n}+ \Dtn$. The baroclinic~\eqref{eq:BC} and the barotropic~\eqref{eq:BT} systems are discretized one after the other with a Lie splitting. More precisely,
\begin{enumerate}
    \item \emph{Baroclinic step:} System~\eqref{eq:BC} is solved on a large time step $\Dtn$. We denote the variables obtained at the end of this step with the superscript  $n+\frac12$.
    
    \item \emph{Barotropic step:} System~\eqref{eq:BT} is then solved on the time interval $[0, \Dtn]$ with $K$ small time substeps $\delta t^k$ using $\hjd$, $\uajd$ and $\Tajd$ as initial data. The last subtimestep gives the values $\hjnp$, $\uajnp$ and $\Tajnp$. It serves as the initial data for step $1$ (the baroclinic step) at the next time step.
\end{enumerate}

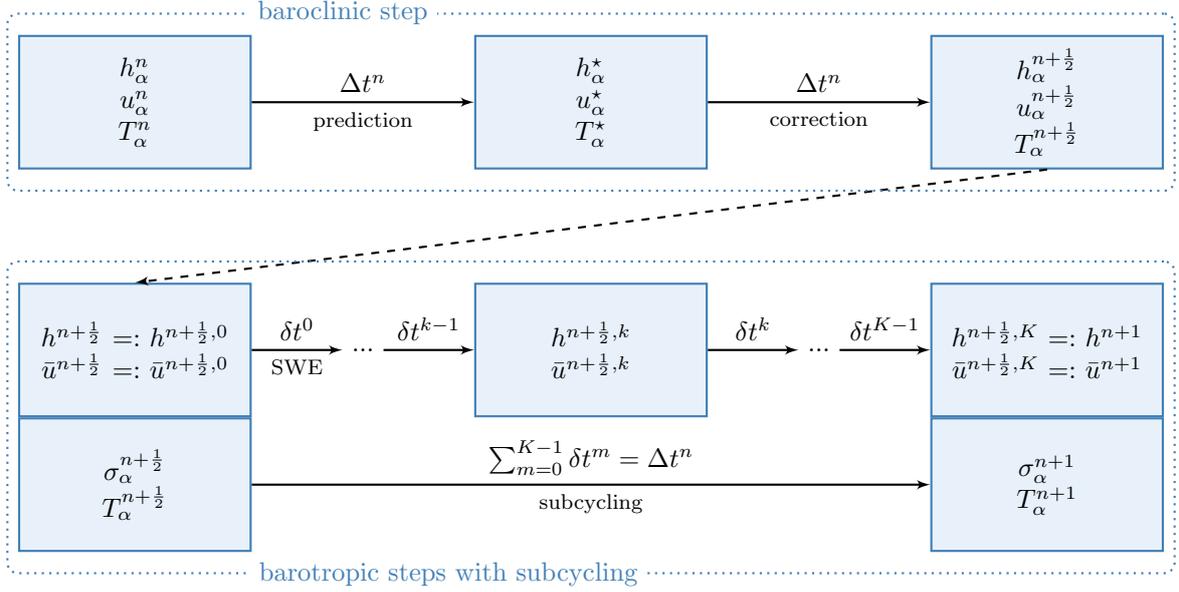
\begin{figure}[H]
    \centering
    \begin{tikzpicture}
      [auto,
       block/.style ={rectangle, draw=rwth-75, thick, fill=rwth-10,
                        text width=8em, align=center,
                        minimum height=5em},
       line/.style ={draw, thick, -latex'}]
    
      \matrix [column sep=12mm]
      {
        % baroclinic
          \node [block] (initial_data_BC) {$\han$ \\ $\uan$ \\ $\Tan$}; & & 
          \node [block] (BC_after_prediction) {$\has$ \\ $\uas$ \\ $\Tas$}; & & 
          \node [block] (BC_after_correction) {$\had$ \\ $\uad$ \\ $\Tad$}; \\ [1.5cm,between borders]
                
        % barotropic
        % SWE
          \node [block] (initial_data_BT) {$h^{n+\frac12} =: h^{n+\frac12,0}$ \\ $\bar{u}^{n+\frac12} =: \bar{u}^{n+\frac12,0}$}; &
          \node (dummy) {...}; & 
          \node [block] (intermediate_BT) {$h^{n+\frac12,k}$ \\ $\bar{u}^{n+\frac12,k}$}; & 
          \node (dummy2) {...}; & 
          \node [block] (BT_final) {$h^{n+\frac12,K} =: h^{n+1}$ \\ $\bar{u}^{n+\frac12,K} =: \bar{u}^{n+1}$}; \\
          
        % transport
          \node[block] (initial_data_subcycling) {$\sa^{n+\frac12}$ \\ $\Ta^{n+\frac12}$}; & & & & 
          \node[block] (subcycling_final) {$\sa^{n+1}$ \\ $\Ta^{n+1}$}; \\
      };
      \begin{scope}[every path/.style=line]
        \path (initial_data_BC)     -- (BC_after_prediction) node [above,midway] {$\Delta t^n$} node [below,midway] {\footnotesize prediction};
        \path (BC_after_prediction) -- (BC_after_correction) node [above,midway] {$\Delta t^n$} node [below,midway] {\footnotesize correction};
        
        \path [dashed] (BC_after_correction.south) -- (initial_data_BT.north);
        
        \path (initial_data_BT)     -- (dummy) node [above,midway] {$\delta t^0$} node [below,midway] {\footnotesize SWE};
        \path (dummy)               -- (intermediate_BT) node [above,midway] {$\delta t^{k-1}$};
        \path (intermediate_BT)     -- (dummy2) node [above,midway] {$\delta t^k$};
        \path (dummy2)              -- (BT_final) node [above,midway] {$\delta t^{K-1}$};
        \path (initial_data_subcycling)  -- (subcycling_final)  node [above,midway] {$\sum_{m=0}^{K-1} \delta t^m = \Delta t^n$} node [below,midway] {\footnotesize subcycling};
      \end{scope}
      
        \node[draw=rwth-75, thick, dotted, rounded corners,  inner xsep=0.4em, inner ysep=0.8em, fit=(initial_data_BC) (BC_after_correction)] (BC) {};
    	\node[fill=white,text=rwth-75] at (BC.160) {baroclinic step};
    	
    	\node[draw=rwth-75, thick, dotted, rounded corners,  inner xsep=0.4em, inner ysep=0.8em, fit=(initial_data_BT) (subcycling_final)] (BT) {};
    	\node[fill=white,text=rwth-75] at (BT.228) {barotropic steps with subcycling};
    		
    \end{tikzpicture}
    \caption{Overview of the scheme. The large baroclinic step with $\Dtn$ consists of a prediction step where system~\eqref{eq:BCnoME} is solved and a correction step where the mass exchange terms are defined and applied. The computed data of the baroclinic step $(\ha,\ua,\Ta)^{n+\frac12}$ is used as initial data for the barotropic step. We perform many small barotropic time steps $\delta t^k$ until we reach the final time $\Dtn$. During this loop we only update the vertically averaged quantities $h$ and $h\baru$ \revised{with the shallow water equation (SWE). The transported variables $\sa$ and $\Ta$ are updated only once with a large time step $\Dtn$.}}
    \label{fig:sketch_bb}
\end{figure}

% \revised{\sout{We detail these two steps in Section~\ref{S:overview}. The timestep restrictions on the timesteps are discussed in Section~\ref{S:CFL}. We prove that both steps ensure a discrete maximum principle on the velocity and on the tracer and a discrete entropy inequality. Special choices of the numerical fluxes are discussed in Section~\ref{S:choices}.}}

% \subsection{Overview of the scheme} \label{S:overview}
An overview of the scheme with the notation and the most important steps is sketched in Figure~\ref{fig:sketch_bb} and in Algorithm~\ref{alg:bb-splitting}. Compared to the original scheme without splitting \revised{\cite{ABPSM11a}}, in which all variables are updated with a small time step $\delta t^n$, there are two advantages in terms of computational cost. Firstly, the exchange terms are only calculated in each large time step $\Dtn$ in the baroclinic step. Secondly, only the water height $h$ and the mean velocity $\baru$ are calculated with small time steps. The transported variables $\sa$ and $\Ta$, on the other hand, are updated with only one large time step $\Dtn$.

The details of the baroclinic and the barotropic steps are given in the following subsections.

\begin{algorithm}[H]
    \begin{algorithmic}[1]
        \Require Initial data $(h^0, \ua^0, \Ta^0)$, final time $t^\text{final}$
        \Ensure $(h^\text{final}, \ua^\text{final}, \Ta^\text{final})$
        \State $t \gets 0$
        \State $(h, \ua, \Ta) \gets (h^0, \ua^0, \Ta^0)$
        \While{$t < t^\text{final}$}
            \State $\Delta t \gets \min(\operatorname{computeDT\_baroclinic}(h, \ua), t^\text{final}-t)$ \Comment{see CFL conditions~\eqref{eq:CFL_prediction},~\eqref{eq:positivity_prediction},~\eqref{eq:CFLcor}}
            \State $(h, \ua, \Ta) \gets \operatorname{baroclinic}(h, \ua, \Ta, \Delta t)$ \Comment{see Section~\ref{subsec:baroclinic}}
            \State $\baru \gets \sumaN \la \ua$
            \State $\tau \gets 0$
            \While{$\tau < \Delta t$}
                \State $\delta t \gets \min(\operatorname{computeDT\_barotropic}(h, \baru), \Delta t - \tau)$ \Comment{see CFL condition~\eqref{eq:CFL_SWE}}
                \State $(h, \baru) \gets \operatorname{barotropic\_shallow\_water}(h, \baru,  \delta t)$ \Comment{see Section~\ref{subsec:barotropic}}
                \State $\tau \gets \tau + \delta t$
            \EndWhile
             \State $(\sa, \Ta) \gets \operatorname{barotropic\_adjustment}(h, \baru, \sa, \Ta, \Delta t)$ \Comment{see update~\eqref{eq:subcycling}}
            \State $\ua \gets \baru + \sa$
            \State $t \gets t + \Delta t$
        \EndWhile
        \State return $(h, \ua, \Ta)$
    \end{algorithmic}
\caption{Barotropic-baroclinic splitting. Lines $4$ to $6$ correspond to the baroclinic part~\eqref{eq:BC}. Lines $8$ to $13$ correspond to the barotropic part~\eqref{eq:BT}}\label{alg:bb-splitting}
\end{algorithm}

\subsection{Baroclinic step} \label{subsec:baroclinic}
% \subsubsection{Baroclinic step} \label{subsec:baroclinic}
This step corresponds to~\eqref{eq:BC}. At the begining of the timestep the values of $\hjn$, $\uajn$ and $ \Tajn$ are known. We follow the prediction/correction strategy of~\cite{ABS18}. 
An overview of the baroclinic step is given in Algorithm~\ref{alg:BC}.

\begin{algorithm}[H]
    \begin{algorithmic}[1]
        \Require $h^n, \ua^n, \Ta^n, \Delta t^n$
        \Ensure $h^{n+\frac12}, \ua^{n+\frac12}, \Ta^{n+\frac12}$
        \State $(h^{n+\frac12},\ha^\star, \ua^\star, \Ta^\star) \gets \operatorname{prediction}(h^n, \ua^n, \Ta^n, \Dtn)$ \Comment{see~\eqref{eq:predictionstep},~\eqref{eq:hjd}}
        \State $\Gapm^\star \gets \operatorname{exchanges}(h^{n+\frac12},\ha^\star)$ \Comment{see~\eqref{eq:def_exchange_terms}}
        \State $(\ua^{n+\frac12}, \Ta^{n+\frac12}) \gets \operatorname{correction}(h^\star,\ua^\star,\Ta^\star, \Gapm^\star, \Dtn)$ \Comment{see~\eqref{eq:correctionstep}}
        \State return $h^{n+\frac12}, \ua^{n+\frac12}, \Ta^{n+\frac12}$
    \end{algorithmic}
\caption{Baroclinic step}\label{alg:BC}
\end{algorithm}

\subsubsection*{Prediction of the total water height}
The prediction step updates the total water height $\hjn$ to $\hjd$ and defines the exchange terms in a way that the constraint $\ha= \la h$ is preserved through time. It consists in solving system~\eqref{eq:BCnoME} without mass exchanges with $2N$ unknowns studied in Proposition~\ref{prop:eigenvalues_prediction}.
We recall that in the absence of exchange terms each individual water height $\ha$ has its own evolution. 
The numerical fluxes are denoted \revised{by the} letter $F$. At the end of \revised{the prediction} step we obtain 
\begin{equation} \label{eq:predictionstep}
    \begin{cases}
        \hajs = \hajn - \frac{\Dtn}{\Delta x} \left( \Fhajp - \Fhajm \right), \\
        (\ha \ua)_j^\star = (\ha \ua)_j^n - \frac{\Dtn}{\Delta x} \left( \Fhauajp - \Fhauajm \right), \\
        (\ha \Ta)_j^\star = (\ha \Ta)_j^n - \frac{\Dtn}{\Delta x} \left( \FhaTajp - \FhaTajm \right) 
    \end{cases}
\end{equation}
with $\Dtn$ being constrained by the CFL condition motivated by Proposition~\ref{prop:eigenvalues_prediction} about the eigenvalues of the baroclinic prediction step
\begin{align} \label{eq:CFL_prediction}
    \frac{\Dtn}{\Delta x} \max_{\alpha,j} |u_{\alpha,j}^n - \bar{u}_j^n| \leq \CFL 
\end{align}
for a constant CFL-number.
\revised{Note that even though the eigenvalues of the baroclinic system without exchanges~\eqref{eq:BCnoME} are not computed explicitly, it is sufficient to use the upper bound $\max_{\alpha,j} |\sajn|$ proven in Propostion~\ref{prop:eigenvalues_prediction} to ensure stability due to a CFL condition.
}

%Two specific choices for the numerical mass flux $F_{j\pm\frac12}^{\ha}$ are discussed later on in Propositions~\ref{P:choice} and~\ref{P:Rusanov} to ensure several desirable properties of the overall scheme. 
On the other hand, we remark that $\ua$ and $\Ta$ are transported at speed $\sa = \ua - \baru$ and choose the numerical fluxes~\cite{AB03}
\begin{equation}\label{eq:fluxdec}
    F^{\ha \varphi_\alpha}_{j+\frac12} = \varphi_{\alpha,j}^n \pos{\Fhajp} - \varphi_{\alpha,j+1}^n \nega{\Fhajp} \quad \text{for } \varphi_\alpha \in \{ \ua, \Ta \},
\end{equation}
where we denote for any real number $a$, its positive part by $a^{+} := \max(a,0)$ and its negative part by $a^{-} :=-\min(a,0)$, thus \revised{$a = a^+ - a^-$}.
The total water height is defined as 
\begin{equation} \label{eq:hjd}
    \hjd = \sumaN \hajs = \hjn - \frac{\Dtn}{\Delta x} \left( \Fhjp - \Fhjm \right) \quad \text{ where } \quad \Fhjpm = \sumaN \Fhajpm.
\end{equation}

\subsubsection*{Definition of the exchange terms}
The individual water heights $\hajs$ obtained in the prediction step~\eqref{eq:predictionstep} are not proportional to each other, \revised{thus} $\hajs \neq \la \hjd$. The exchange terms are defined to restore this constraint.

\begin{proposition}\label{P: exchanges}
    Define the exchange terms as $G_{\frac12,j}^{\star}:=0$ and for all $\alpha \in \{ 1, \dots, N-1\}$ 
    \begin{equation} \label{eq:def_exchange_terms}
        \Gapj^\star := \frac{1}{\Dtn} \left( \la h_j^{n+\frac12} - \hajs \right) + \Gamj^\star.
    \end{equation}
    Then $G_{N+\frac12,j}^{\star}=0$ and for any  $\alpha \in \{ 1, \dots, N\}$ it holds
    \begin{equation} \label{eq:hajd}
        \hajd = \hajs + \Dtn \left(\Gapj^{\star}-\Gamj^{\star} \right) = \la \hjd.
    \end{equation}
\end{proposition}

\begin{proof}
    Equality~\eqref{eq:hajd} is straightforwardly deduced from Definition~\eqref{eq:def_exchange_terms}. Definition~\eqref{eq:hjd} yields
    $$ \Dtn G_{N+\frac12,j}^{\star} = \sumaN \left( \la \hjd-\hajs \right) + \Dtn G_{\frac12,j}^{\star} = \hjd - \sumaN \hajs = 0. $$
\end{proof}

\subsubsection*{Correction step} 
The exchange terms defined in Proposition~\ref{P: exchanges} are applied to correct the velocities and tracer values:
\begin{equation} \label{eq:correctionstep}
    (\ha \phia)_j^{n+\frac12} = (\ha \phia)_j^\star + \Dtn \left( \varphi_{\alpha+\frac12,j}^{\sharp} \Gapj^{\star} - \varphi_{\alpha-\frac12,j}^{\sharp} \Gamj^{\star}\right) \qquad  \varphi \in \{ u, T \},
\end{equation}
with either $\sharp=\star$ (explicit scheme) or $\sharp=n+\frac12$ (implicit scheme).
The interface velocities and tracers are defined in accordance with~\eqref{eq:interu} by
\begin{equation} \label{eq:decu}
    \varphi_{\alpha+\frac12,j}^{\sharp} =
    \begin{cases}
        \varphi_{\alpha,j}^{\sharp} & \mbox{if } \;\Gapj^{\star} \leq 0,\\
        \varphi_{\alpha+1,j}^{\sharp} & \mbox{if } \;\Gapj^{\star} > 0,
    \end{cases} \qquad \varphi \in \{ u, T \}.
\end{equation}

\subsubsection*{Properties of the baroclinic step}
We exhibit two timestep restrictions that guarantee the non-negativity of the water heights and a maximum principle on \revised{the velocity} $\ua$ and \revised{the tracer} $\Ta$ in the baroclinic step.
\begin{proposition} \label{prop:maxprincipleprediction}
    Consider the \revised{baroclinic} prediction step~\eqref{eq:predictionstep} with the decentered choice~\eqref{eq:fluxdec} for the momentum flux. Suppose that the time step $\Dtn$ is small enough so that for all $\alpha$ and for all $j$
    \begin{equation} \label{eq:positivity_prediction}
        \hajn - \frac{\Dtn}{\Delta x} \left( \pos{\Fhajp} + \nega{\Fhajm} \right) \geq 0.
    \end{equation}
    Then the water heights $\hajs$ are non-negative and a discrete maximum principle on $\ua$ and $\Ta$ holds:
    $$ \min(\varphi_{\alpha, j-1}^{n}, \varphi_{\alpha, j}^{n}, \varphi_{\alpha,  j+1}^{n})  \leq \varphi_{\alpha,  j}^{\star} \leq \max(\varphi_{\alpha, j-1}^{n}, \varphi_{\alpha,  j}^{n}, \varphi_{\alpha,  j+1}^{n})  \qquad \forall \varphi \in \{ u, T\}.$$
\end{proposition}
\revised{
\begin{remark} 
    Condition~\eqref{eq:positivity_prediction} is not as explicit as Condition~\eqref{eq:CFL_prediction}. We give in Propositions~\ref{P:choice} and~\ref{P:Rusanov} two explicit choices where~\eqref{eq:CFL_prediction} implies~\eqref{eq:positivity_prediction}.
    For other choices if the non-negativity of the water height
    should be violated under the CFL condition~\eqref{eq:CFL_prediction}, it can be enforced by modifying the mass fluxes using the draining time technique \cite{BollermannLukacovaNoelle2011}. This sets the mass flux to zero when the water height reaches zero within a time step. Note that it does not reduce the overall time step, but is an improved quadrature rule for the time integral of the numerical flux. For the terrain-following coordinates, this will only happen at wet-dry fronts.
\end{remark}}

\begin{proof}
    We reproduce the proof of~\cite[Theorem 4.3]{AB03}. The positivity of the water height is easy:
    \begin{align*}
        \hajs &=\hajn- \frac{\Dtn}{\Delta x} \left( \Fhajp - \Fhajm \right) \\
        &= \hajn- \frac{\Dtn}{\Delta x} \left( \pos{\Fhajp} - \nega{\Fhajp} - \pos{\Fhajm} + \nega{\Fhajm} \right) \\
        & \geq  \hajn- \frac{\Dtn}{\Delta x} \left( \pos{\Fhajp} + \nega{\Fhajm} \right)
    \end{align*}
    which is non-negative under hypothesis~\eqref{eq:positivity_prediction}.
    We now prove the lower bound for $\varphi \in \{u, T\}$ updated with~\eqref{eq:fluxdec}
    \begin{align*}
        (h \varphi )_{\alpha,j}^{\star} &= (h\varphi)_{\alpha,j}^{n} - \frac{\Dtn}{\Delta x} \left[
        \varphi_{\alpha,j}^{n} \pos{\Fhajp} - \varphi_{\alpha,j+1}^{n} \nega{\Fhajp} -  \varphi_{\alpha,j-1}^{n} \pos{\Fhajm} + \varphi_{\alpha,j}^{n} \nega{\Fhajm} \right]\\
        & = \varphi_{\alpha, j}^{n} \left[\hajn- \frac{\Dtn}{\Delta x} \left( \pos{\Fhajp} - \nega{\Fhajm} \right) \right]  
         + \varphi_{\alpha,j+1}^{n} \frac{\Dtn}{\Delta x} \nega{\Fhajp} 
         + \varphi_{\alpha,j-1}^{n} \frac{\Dtn}{\Delta x} \pos{\Fhajm}   \\
        & \geq  \min(\varphi_{\alpha,j-1}^{n}, \varphi_{\alpha,j}^{n}, \varphi_{\alpha,j+1}^{n}) \left( \hajn - \frac{\Dtn}{\Delta x} \left[ \pos{\Fhajp} - \nega{\Fhajp} - \pos{\Fhajm} + \nega{\Fhajm} \right] \right)\\
        & = \min(\varphi_{\alpha,j-1}^{n}, \varphi_{\alpha,j}^{n}, \varphi_{\alpha, j+1}^{n}) \hajs.
    \end{align*}
    The inequality holds because the coefficient in front of $\varphi_{\alpha, j}^{n}$ is non-negative as soon as~\eqref{eq:positivity_prediction} holds. The upper bound is proved in a similar fashion.
\end{proof}

\begin{proposition} \label{P:sumsa}
    The \revised{baroclinic} correction \revised{step} on the velocities and the tracers~\eqref{eq:correctionstep} with the decentered choice~\eqref{eq:decu} for the interface velocities verifies the maximum principle 
    $$ \min_{\beta}(\varphi_{\beta,j}^{\star})  \leq \varphi_{\alpha,j}^{n+\frac12} \leq \max_{\beta}(\varphi_{\beta,j}^{\star}), \qquad \varphi_\alpha \in \{ u, T \} $$
    for any time step $\Dtn$ in the implicit case $\sharp=n+\frac12$.
    In the explicit case $\sharp=\star$ it verifies the maximum principle
    $$ \min(\varphi_{\alpha-1,j}^{\star}, \varphi_{\alpha,j}^{\star}, \varphi_{\alpha+1,j}^{\star})  \leq \varphi_{\alpha,j}^{n+\frac12} \leq \max(\varphi_{\alpha-1,j}^{\star}, \varphi_{\alpha,j}^{\star},  \varphi_{\alpha+1,j}^{\star}) $$
    if the time step satisfies the bound
    \begin{equation} \label{eq:CFLcor}
     \Dtn \left(\Gapj^{\star}\right)^-  + \Dtn   \left(\Gamj^{\star}\right)^+  \leq \revised{\hajs}. 
    \end{equation}
\end{proposition}
\begin{proof}
    We refer the reader to~\cite[Lemma 3.11]{ABS18} for the implicit case.
    For the explicit case the proof is the same as in Proposition~\ref{prop:maxprincipleprediction}, with $\alpha$ playing the role of $j$ and $G^{\star}$ the role of $F^{\ha}$.
\end{proof}

\subsection{Barotropic step} \label{subsec:barotropic}
% \subsubsection{Barotropic step} \label{subsec:barotropic}

The barotropic step corresponds to System~\eqref{eq:BT}.
Its first two lines  correspond to the shallow water equations on $(h, h \baru)$ and have a severe CFL restriction on the timestep. Thus the timestep $\Dtn$ set in the baroclinic step is decomposed in $K$ small sub-timesteps. We initialize the barotropic step with
$$h_{j}^{n+\frac12,0} := \hjd, \quad \baru_j^{n+\frac12,0}:= \sumaN \la \uajd  \quad \text{ and } \quad \sajd:=(\ua- \baru)_j^{n+\frac12}.$$
For $0 \leq k \leq K-1$, the $k$-th time subtimestep updates the values \revised{$(h,h\baru)_j^{n+\frac12,k}$} at time $\tau^{k}$ to the values \revised{$(h,h\baru)_j^{n+\frac12,k+1}$} at time $\tau^{k+1} = \tau^k + \dtk$. 
At the last sub-timestep $K$ we have $\sum_{k=0}^{K-1} \dtk=\tau^{K}=\Dtn$ and we define $h_{j}^{n+1} := h_{j}^{n+\frac12,K}$, $\baru_{\alpha,j}^{n+1} := \baru_{\alpha,j}^{n+\frac12,K}$. The deviation $\sa$ and the tracer $\Ta$ are updated directly from \revised{$(\ha\sa,\ha\Ta)_j^{n+\frac12}$} to \revised{$(\ha\sa,\ha\Ta)_j^{n+1}$}. 

The shallow water numerical fluxes for the density $h$ and the mean momentum $h \baru$ at interfaces $j\pm\frac12$ are denoted respectively with $\fhkjpm$ and $\fhukjpm$ and the discretization of the topography source term with $\Sjndk$. The scheme reads as 
\begin{equation} \label{eq:FV-SW}
    \begin{cases}
      \hjdkp = \hjdk - \frac{\dtk}{\Delta x} (\fhkjp-\fhkjm), \\
      (h \baru)_j^{n+\frac12,k+1} = (h \baru)_j^{n+\frac12,k} - \frac{\dtk}{\Delta x} (\fhukjp-\fhukjm) + \dtk \Sjndk, \\
    \end{cases}
\end{equation}
with timestep $\dtk$ being constrained by a classical CFL condition of the form
\begin{equation} \label{eq:CFL_SWE}
    \dtk \leq \frac{\CFL \, \Delta x}{\max_j |\barujdk|+ \sqrt{g \hjdk } } 
\end{equation}
for some constant Courant--Friedrichs--Lewy number $\CFL$ depending on the scheme. \revised{When choosing Method~\eqref{eq:FV-SW}, the reader may use their preferred shallow water solver. We specify our choice in the beginning of Section~\ref{S:testcases}.}

\revised{It remains to update the transported unknowns $\sa$ and $\Ta$.}
The simplest method would be to update \revised{them} as frequently as $h$ and $h\baru$ with
\begin{equation*} \label{eq:nosubcycling}
    \begin{cases}
      (h \sa)_j^{n+\frac12,k+1} = (h \sa)_j^{n+\frac12,k} - \frac{\dtk}{\Delta x} (\fhsakjp-\fhsakjm), \\
      (h \Ta)_j^{n+\frac12,k+1} = (h \Ta)_j^{n+\frac12,k} - \frac{\dtk}{\Delta x} (\fhTakjp-\fhTakjm). 
    \end{cases}
\end{equation*}
However, it is preferable both from the computational cost point of view and to reduce the numerical diffusion to take advantage of the fact that the equations on $h \Ta$ and $h \sa$ are decoupled from the ones on $(h, h\baru)$ in~\eqref{eq:BT}. \revised{We follow the subcycling strategy of~\cite{AB03} and update $(\ha\sa, \ha\Ta)_j^{n+\frac12}$ only once with a large time step $\Dtn$.}
\revised{For this purpose, we gather} the mass fluxes of all the \revised{barotropic} sub-timesteps in 
\begin{equation} \label{eq:summed_massflux}
    \Dtn \fhjp := \sum_{k=0}^{K-1} \dtk \fhkjp \ \text{ so that } \ \hjnp = \hjd - \dfrac{\Dtn}{\Delta x} \left( \fhjp - \fhjm \right). 
\end{equation} 
\revised{Now}, we decenter the numerical fluxes as before with
\begin{equation} \label{eq:fluxdecBT}
    \mathcal{F}_{j + \frac12}^{h \varphi_\alpha} :=  \varphi_{\alpha,j}^{n+\frac12} \pos{\fhjp} - \varphi_{\alpha,j+1}^{n+\frac12} \nega{\fhjp}
    \quad \text{for} \quad \varphi_\alpha \in \{ \sa, \Ta \}
\end{equation} 
and update $\sa$ and $\Ta$ with only one large timestep $\Dtn$ \revised{by}
\begin{equation} \label{eq:subcycling}
    \begin{cases}
      (h \sa)_j^{n+1} = (h \sa)_j^{n+\frac12} - \frac{\Dtn}{\Delta x} (\fhsajp-\fhsajm), \\
      (h \Ta)_j^{n+1} = (h \Ta)_j^{n+\frac12} - \frac{\Dtn}{\Delta x} (\fhTajp-\fhTajm). 
    \end{cases}
\end{equation}
The strategy is detailed in Algorithm~\ref{alg:BT-subcycling}. A condition on $\Dtn$ ensuring the maximum principle on $\sa$ and $\Ta$ is given in Proposition~\ref{P:maxBT} below. We observe in practice that this condition is always fulfilled for the baroclinic timestep $\Dtn$, see Section~\ref{S:testcases}. For the sake of completeness, we however describe a strategy in \revised{Remark~\ref{rmk:non-negativity-subcycling} and Algorithm~\ref{alg:BT-subcycling-negativity-check}} handling all the cases.

\begin{algorithm}[htbp]
    \begin{algorithmic}[1]
        \Require Initial data for barotropic loop $(h, \ua, \Ta)^{n+\frac12}$, final time for barotropic loop $\Dtn$
        \Ensure $(h, \ua, \Ta)^{n+1}$
        \State $\tau \gets 0$
        \State $\Delta t \mathcal{F}^h \gets 0$
        %\State $(h,\baru,\sa,\Ta) \gets (h, \baru, \sa \Ta)^{n+\frac12}$
        \State $(h,\baru,\sa,\Ta) \gets \left(h^{n+\frac12}, \sumaN \la \ua^{n+\frac12}, \ua^{n+\frac12}-\baru^{n+\frac12}, \Ta^{n+\frac12} \right)$
        \While{$t < \Dtn$}
            \State $\delta t \gets \min(\operatorname{computeDT\_barotropic}(h, \baru), \Dtn - \tau)$ \Comment{see~\eqref{eq:CFL_SWE}}
            \State $(h, \baru, f^h) \gets \operatorname{SWE\_solver}(h, \baru, \delta t)$ \Comment{see~\eqref{eq:FV-SW}}
            \State $\Delta t \mathcal{F}^h \gets \Delta t \mathcal{F}^h + \delta t f^h$ \Comment{see~\eqref{eq:summed_massflux}}
            \State $\tau \gets \tau + \delta t$
        \EndWhile 
        \State $(\sa,\Ta) \gets \operatorname{update\_transported\_quantities} (\sa,\Ta,\Delta t \mathcal{F}^h)$ \Comment{see~\eqref{eq:subcycling}}
        \State return $(h, \ua= \sa+\baru,\Ta)$
    \end{algorithmic}
\caption{Barotropic loop with subcycling strategy}\label{alg:BT-subcycling}
\end{algorithm}

\subsubsection*{Properties of the barotropic step} \label{S:propBT}
We first check that the weighted sum of the \revised{deviations} $\sumaN \la \sa= \sumaN \la(\ua-\baru)$ remains \revised{zero} during the entire barotropic step.
%We first check that the update of $h \baru$ with the shallow water equation coincides with the update of $\sumaN \la h \ua$. In other words, we recover at the discrete level that the sum over all layers on the last equation of~\eqref{eq:BT} vanishes. Or in other words the weighted sum of the deviations $\sumaN \la \sa$ remains null during the entire barotropic step.
\begin{proposition} \label{P:sumsigma}
    The decentered choice~\eqref{eq:fluxdecBT} for the numerical flux ensures that $\sumaN \la \sajnp =0$ holds for all $j$ and for all $n$.%for all $0 \leq k \leq K$ and for all $j$, $\sumaN \la \sajdkp =0$ 
\end{proposition}
\begin{proof}
   At the beginning of the barotropic step we have $\sumaN \la \sajd =0$. Then
   $$ 
   \sumaN \la \mathcal{F}_{j + \frac12}^{h \sigma_\alpha}=  \pos{\fhjp} \sumaN \la \sigma_{\alpha,j}^{n+\frac12}  -  \nega{\fhjp} \sumaN \la \sigma_{\alpha,j+1}^{n+\frac12} =0 
   $$
   and the result follows by summing~\eqref{eq:subcycling} over all the layers.
\end{proof}

Second, we prove the non-negativity of the water height $\hjnp$ and a maximum principle for the velocity and the tracer.

\begin{proposition} \label{P:maxBT}
    Suppose that
    \begin{equation} \label{eq:non-negativity}
        h_j^{n+\frac12} - \frac{\Dtn}{\Delta x} \left( \pos{\fhjp} - \nega{\fhjm} \right) \geq 0
    \end{equation}
    holds. Then $\hjnp$ is non-negative and the following maximum principle on $\varphi \in \{ \Ta, \sa\}$ obtained with~\eqref{eq:subcycling} holds:
    \begin{equation*}
        \min \left( \varphi_{j-1}^{n+\frac12}, \varphi_{j}^{n+\frac12}, \varphi_{j+1}^{n+\frac12} \right) 
        \leq \varphi_j^{n+1} 
        \leq \max \left( \varphi_{j-1}^{n+\frac12}, \varphi_{j}^{n+\frac12}, \varphi_{j+1}^{n+\frac12} \right).
    \end{equation*}
\end{proposition}

\begin{proof}
    The proof is the same as for Proposition~\ref{prop:maxprincipleprediction}.
\end{proof}

\begin{remark}\label{rmk:non-negativity-subcycling}
    In practice, there is no guarantee that condition~\eqref{eq:non-negativity} \revised{holds}. We therefore check condition~\eqref{eq:non-negativity} using the summed flux $\sum_{k=0}^m \delta t^k f^{h,k}$ instead of $\Dtn \mathcal{F}^h = \sum_{k=0}^{K-1} \dtk f^{h,k}$ for each $m$, $1 \leq m \leq K-1$. If condition~\eqref{eq:non-negativity} is not fulfilled for one $m$, we update the transported quantities $(h\sa)^{n+\frac12,m}$ and $(h\Ta)^{n+\frac12,m}$ in between according to~\eqref{eq:subcycling}, but using the mass fluxes summed up to $m$ and restart the summation of the mass fluxes from here. The strategy is the same as in \cite{AB03} and is illustrated in more detail in Algorithm~\ref{alg:BT-subcycling-negativity-check}. The main argument of the proof is once again the one of Proposition~\ref{prop:maxprincipleprediction}.
\end{remark}

\begin{algorithm}[htbp]
    \begin{algorithmic}[1]
        \Require Initial data for barotropic loop $(h, \ua, \Ta)^{n+\frac12}$, final time for barotropic loop $\Dtn$
        \Ensure $(h, \baru, \sa, \Ta)^{n+1}$
        \State $\tau \gets 0$
        \State $\Delta t \mathcal{F}^h \gets 0$
        \State $(h,\baru,\sa,\Ta) \gets \left(h^{n+\frac12}, \sumaN \la \ua^{n+\frac12}, \ua^{n+\frac12}-\baru^{n+\frac12}, \Ta^{n+\frac12} \right)$
        \State $h_\text{old} \gets h$
        \While{$\tau < \Dtn$}
            \State $\delta t \gets \min(\operatorname{computeDT\_barotropic}(h, \baru), \Dtn - \tau)$ \Comment{see~\eqref{eq:CFL_SWE}}
            \State $(f^{h},f^{h\baru}) \gets \operatorname{compute\_numerical\_SWE\_fluxes}(h,h\baru)$
            \If{$h_{\text{old},j} - \frac{1}{\Delta x} \left( \pos{\Delta t \mathcal{F}^h_{j+\frac12} + \delta t f^h_{j+\frac12}} + \nega{\Delta t \mathcal{F}^h_{j-\frac12} + \delta t f^h_{j-\frac12}} \right) \geq  0$ for all $j$} \Comment{see~\eqref{eq:non-negativity}}
                \State $(h, \baru) \gets \operatorname{SWE\_solver}(h, \baru, f^{h}, f^{h\baru}, \delta t)$ \Comment{see~\eqref{eq:FV-SW}}
                \State $\Delta t \mathcal{F}^h \gets \Delta t \mathcal{F}^h + \delta t f^h$ \Comment{see~\eqref{eq:summed_massflux}}
                \State $\tau \gets \tau + \delta t$
                % \If{$t = \Dtn$}
                %     \State $(\sa,\Ta) \gets \operatorname{update\_transported\_quatities}(\sa,\Ta,\Delta t \mathcal{F}^h)$ \Comment{see~\eqref{eq:subcycling}}
                % \EndIf
            \Else 
                \State $(\sa,\Ta) \gets \operatorname{update\_transported\_quantities}(\sa,\Ta,\Delta t \mathcal{F}^h)$ \Comment{see~\eqref{eq:subcycling}}
                \State $h_\text{old} \gets h$ 
                \State $\Delta t \mathcal{F}^h \gets 0$
            \EndIf
        \EndWhile
        \State $(\sa,\Ta) \gets \operatorname{update\_transported\_quantities}(\sa,\Ta,\Delta t \mathcal{F}^h)$ \Comment{see~\eqref{eq:subcycling}}
        \State return $(h, \ua= \sa+\baru,\Ta)$
    \end{algorithmic}
\caption{Barotropic loop (with subcycling and non-negativity check)}\label{alg:BT-subcycling-negativity-check}
\end{algorithm}

\revised{
\begin{remark}
The subcycling strategy presented here results in a fully explicit scheme with an explicit numerical flux for the shallow water system as a building block. Many are available and we study in the next section how the properties of this building block affect the properties of the overall scheme. An alternative could be to solve~\eqref{eq:BT} with the same large timestep $\Delta t$ with an implicit-explicit scheme~\cite{CGK16},~\cite{BALN14} or~\cite{C14} for example.
\end{remark}
}

\subsection{Properties of the entire scheme} \label{S:choices}
% \subsubsection{Properties of the entire scheme} \label{S:choices}

We now consider the global scheme  based on the barotropic-baroclinic splitting, going from $(h, \ua, \Ta)^{n}_{j}$ to $(h, \ua, \Ta)^{n+1}_{j}$ with the timestep $\Dtn$, consisting of
\begin{itemize}
    \item the baroclinic part with
    \begin{itemize}
        \item the prediction step~\eqref{eq:predictionstep}-\eqref{eq:hjd} with Choice~\eqref{eq:fluxdec} for the momentum flux;
        \item the correction step~\eqref{eq:correctionstep}-\eqref{eq:decu} with Definition~\eqref{eq:def_exchange_terms} of the exchange terms;
    \end{itemize}
    \item the baroclinic part with
     the update of the shallow water system~\eqref{eq:FV-SW} and the update~\eqref{eq:subcycling} of the deviation $\sa$ and the tracer $\Ta$ with the choice of flux~\eqref{eq:fluxdecBT}.
\end{itemize}
We suppose that  the CFL conditions~\eqref{eq:positivity_prediction},~\eqref{eq:CFLcor} and~\eqref{eq:non-negativity} are fulfilled. 
%The entropy inequality for the continuous solutions~\eqref{eq:entropyMSW} also holds at the discrete level. 
The entropy flux in the entropy inequality~\eqref{eq:entropyMSW} for the continuous solutions is denoted by $F^{E}=g h^{2} \baru + g h \zb \baru + \sumaN \frac{\la h \ua^{3}}{2}$. We now show that this inequality also holds at the discrete level.
\begin{theorem} \label{thm:DEI}
    Suppose that the scheme verifies the following additional assumptions:
    \begin{itemize}
        \item in the prediction step, there exists a numerical potential energy flux $F^{E^{p}}_{j\pm\frac12}$ consistent with \revised{the zero flux} such that
        \begin{equation} \label{eq:Ep}
            (E^{p})_{j}^{*} \leq (E^{p})_{j}^{n} - \dfrac{\Dtn}{\Delta x} \left( F^{E^{p}}_{j+\frac12} - F^{E^{p}}_{j-\frac12}  \right) \ \text{ where } \ E^{p}= g \frac{h^{2}}{2} + g h \zb ;
        \end{equation}
        % the decay of the potential energy $ \partial_{t} \left(  g \frac{h^{2}}{2} + g h \zb  \right) \leq 0$ holds (see Lemma~\ref{L:pred} for a precise definition). 
        %This hypothesis is automatically verified if $\sumaN \Fhajp= 0$.
        %the mass flux~\eqref{eq:predictionstep} in the prediction step is such that $\sumaN \Fhajp= 0$;
        \item the shallow water scheme~\eqref{eq:FV-SW} satisfies a discrete entropy inequality (see Lemma~\ref{L:EIBT} for a precise definition) under the subtimestep restriction~\eqref{eq:CFL_SWE} on $\dtk$.
    \end{itemize}
    Then  there exists a numerical entropy flux $F^{E}_{j \pm \frac12}$ consistent with the continuous entropy flux $F^{E}$ such that, under the same restrictions on $\Dtn$ and $\dtk$, the following discrete entropy inequality holds:
    $$ E_{j}^{n+1} \leq E_{j}^{n} + \frac{\Dtn}{\Delta x} \left( F^{E}_{j + \frac12} - F^{E}_{j -\frac12}\right). $$
\end{theorem}

\begin{remark} \label{rmk:sumnull}
    Let us recall that in the baroclinic step~\eqref{eq:BC} the total water height and thus the potential energy $E^{p} = g \frac{h^{2}}{2} + g h \zb $ remain constant. If the total mass flux $\sumaN F^{\ha}_{j\pm\frac12}$ vanishes, we automatically get $\hjd=\hjn$ and an exact numerical counterpart of $\partial_{t} h =0$ and $\partial_{t} E^{p} =0$ at the numerical level. In this case~\eqref{eq:Ep} is an equality with $F^{E^{p}}_{j\pm\frac12}=0$.
    When the numerical discretization introduces a variation in the potential energy, Inequality~\eqref{eq:Ep} is sufficient to obtain the result.
    %The mass flux in the prediction step $\Fhajp$ is consistent with $\ha (\ua-\baru)$ the continuous mass flux in~\eqref{eq:BCnoME}. Thus the first hypothesis implies that the total water height does not vary through the baroclinic step: $\hjd=\hjn$. In other words, the first equation $\partial_{t} h$ of~\eqref{eq:BC} is exactly solved at the numerical level. 
    %
    %The Rusanov flux
    %$$ \Fhajp = \dfrac{h_{\alpha, j}^{n} \sigma_{\alpha, j}^{n} + h_{\alpha, j+1}^{n} \sigma_{\alpha, j+1}^{n}}{2} - A_{j+\frac12} \dfrac{h_{\alpha, j+1}^{n}-h_{\alpha, j}^{n} }{2} $$
    %does not fulfill this hypothesis. Indeed by definition of the deviation $\sa$ one has $\sumaN h_{\alpha, j}^{n} \sigma_{\alpha, j}^{n}= \sum h_{\alpha, j+1}^{n} \sigma_{\alpha, j+1}^{n}$ but the diffusif part of the flux yields
    %$$ \hjd = \hjn- \frac{\Dtn}{\Delta x} \sumaN \left( \Fhajp -  \Fhajp \right)=  \frac{\Dtn}{\Delta x}  \dfrac{ A_{j+\frac12} ( \hjpn-h_{ j}^{n} ) - A_{j-\frac12} ( h_{j}^{n}-h_{ j-1}^{n} )}{2} \neq 0 $$
\end{remark}
\revised{
\begin{remark}
In the presence of a bathymetry source term the second condition is challenging but is verified by Godunov schemes and some approximate Riemann Solver fluxes~\cite{BC16} at the cost of solving a nonlinear equation. See also the recent work~\cite{MC25} for a simpler choice and references therein. In our numerical tests below we used the hydrostatic reconstruction technique~\cite{ABBKP04,CN17}. It only fulfills a semi-discrete entropy inequality which is sufficient to prove the convergence of the scheme~\cite{BL21}. %Thus, the second hypothesis could probably be relaxed by allowing an error in $O(\Delta x^{2})$ in the inequality.
\end{remark}}

\begin{proof}[Proof of Theorem~\ref{thm:DEI}]
The proof is detailed in Section~\ref{S:DEIproofs} and heavily relies on the splitting of the total entropy inequality~``\eqref{eq:entropyMSW}=\eqref{eq:EIBT}+\eqref{eq:EIBC}'' obtained in Proposition~\ref{P:energyequality}. We obtain a discrete counterpart of~\eqref{eq:EIBC} for the baroclinic part in Lemma~\ref{L:pred} and in Lemma~\ref{L:cor} for the prediction and correction steps respectively. Lemma~\ref{L:EIBT} mimics~\eqref{eq:EIBT} through the barotropic step with the subcycling strategy. The three inequalities are then put together to obtain the decay of the entropy over the entire timestep.
\end{proof}

\begin{proposition} \label{P:choice}
    Suppose that the shallow water scheme~\eqref{eq:FV-SW} preserves 
    both 
    the non-negativity of the water height under a CFL condition of the form~\eqref{eq:CFL_SWE} 
    and the lake at rest equilibrium with $\baru = 0$ and constant free surface $h+ \zb$.
    % \revised{and is well-balanced for some steady state}.
    Consider the baroclinic mass flux for~\eqref{eq:predictionstep} given by
    \begin{equation} \label{eq:choice}
        \Fhajp = (\ha\sa)_{j+\frac12}^n 
        := \begin{cases}
            (\ha \sa)_j^n, & \text{ if } \hajn<\hajpn, \\
            (\ha \sa)_{j+1}^n, & \text{ if } \hajn \geq \hajpn.
        \end{cases}
    \end{equation}
    Then the overall scheme:
    \begin{enumerate}[(i)]
        \item preserves the non-negativity of the water height under the CFL condition $ \Dtn \leq \dfrac{0.5 \Delta x}{\max_{\alpha,j} |\sigma_{\alpha,j}^{n}|}$; \label{item:proof_choice_non-negativity}
        
        \item preserves the shallow water equation solutions:
        $$ (\forall \alpha, \ \forall j, \ \uajn=\baru_{j}^{n }) \Longrightarrow (\forall \alpha, \ \forall j, \ \uajnp=\baru_{j}^{n+1} ); $$ \label{item:proof_choice_SWE}
        
        \item preserves the lake at rest equilibrium; 
        % \revised{well-balancing property of the shallow water solver}; 
        \label{item:proof_choice_wb}
        
        \item verifies the first hypothesis of Theorem~\ref{thm:DEI}. \label{item:proof_choice_DEI}
        
    \end{enumerate}
\end{proposition}

\begin{proof}
    \eqref{item:proof_choice_non-negativity} We \revised{first prove}
    % now turn to 
    the non-negativity condition~\eqref{eq:positivity_prediction}. We treat the worst case for the non-negativity, the case $\Fhajp>0$, $\Fhajm<0$. Then
    \begin{align*}
        \hajn - \frac{\Dtn}{\Delta x} \left( \pos{\Fhajp} + \nega{\Fhajm} \right)
        &= \hajn- \dfrac{\Dtn}{\Delta x} \left( h_{\alpha,j+\frac12}^{n} |\sigma_{\alpha,j+\frac12}^{n}| + h_{\alpha,j-\frac12}^{n} | \sigma_{\alpha,j-\frac12}^{n}| \right) \\
        & \geq \hajn \left( 1 - \dfrac{\Dtn}{\Delta x} (|\sigma_{\alpha,j+\frac12}^{n}| + | \sigma_{\alpha,j-\frac12}^{n}|) \right)
    \end{align*}
    which is non-negative as soon as
    $$ \Dtn \leq \dfrac{0.5 \Delta x}{\max_{\alpha,j} |\sigma_{\alpha,j}^{n}|}.$$
    
    \eqref{item:proof_choice_SWE} We now turn to the case of a purely barotropic setting, where the velocity is uniform along every vertical column of fluid (but may vary along the horizontal): for all $\alpha$ and for all $ j$, \revised{it holds} $\uajn=\baru_{j}^{n}$ or in other words $\sigma_{\alpha,j}^{n}=0$ identically. 
    \revised{Then the mass flux~\eqref{eq:choice} $\Fhajp$ vanishes and thus $\hjd=\hjn$, $F^{\ha \ua}_{j+\frac12}=0$, and $\uajs=\uajn$.}
    % Then once again in the baroclinic step $\Fhajp=0$, yielding $F^{\ha \ua}_{j+\frac12}=0$, $\hajs=\hajn$ and $u_{\alpha,j}^{\star}=u_{\alpha,j}^{n}$.
    The exchange terms vanish and the correction step does nothing. In the barotropic step, $h$ and $\baru$ vary in the resolution of the shallow water system~\eqref{eq:FV-SW}. The redistribution over the layers~\eqref{eq:subcycling} gives $\sigma_{\alpha,j}^{n+1}=0$ and thus $u_{\alpha,j}^{n+1} = \bar{u}_j^{n+1}$.
    
    \eqref{item:proof_choice_wb} %\revised{Because of~\eqref{item:proof_choice_SWE} the overall scheme preserves the shallow water solutions. If this solution is a steady state preserved by the shallow water solver by hypothesis then this steady state is preserved by the overall scheme.}
    \revised{The lake at rest equilibrium is a solution of the shallow water equation. Because of~\eqref{item:proof_choice_SWE} it is preserved by the overall scheme.}
    % Consider a lake at rest initial data: there exists $C$ such that for all $j$, $\hjn+(\zb)_{j}= C$ and for all $\alpha$ and for all $j$, $\uajn=0$. Then all the deviations $\sigma_{\alpha,j}^{n}$ are equal to $0$. The decentered choice~\eqref{eq:fluxdec} for the fluxes yields $F^{\ha \ua}_{j+\frac12}=0$ and thus $\uajs=0$. Moreover, with the mass flux~\eqref{eq:choice} we not only have $\hjd=\hjn$ but also $\hajs=\hajn=\la \hjd$. Thus all the exchange terms~\eqref{eq:def_exchange_terms} vanish and we also have $\uajd=0$ after the correction step~\eqref{eq:correctionstep}. The lake at rest equilibrium is thus preserved through the baroclinic step. It is also preserved in the shallow water part by hypothesis. The redistribution~\eqref{eq:fluxdecBT}-\eqref{eq:subcycling} yields $\sajnp=0$ and all the individual velocities are still null after the barotropic step.
    
    \eqref{item:proof_choice_DEI} \revised{To prove that the potential energy decreases in the baroclinic prediction step,} we recall that  $\sumaN \hajn \sigma_{\alpha,j}^{n} =0$
    %$$\sumaN \hajn \sigma_{\alpha,j}^{n} = \sumaN \la \hjn \left(\uajn-\sumbN \lb u_{\beta,j}^{n} \right) = 0, $$
    which yields $\sumaN \Fhajp=0$ and $\hjd = \hjn$, see Remark~\ref{rmk:sumnull}. In particular~\eqref{eq:Ep} holds with  $F^{E^{p}}_{j\pm\frac12}=0$.
\end{proof}

\revised{
\begin{remark}
    It is often desired that numerical schemes are able to maintain the fundamental equilibria of the hyperbolic system. For this purpose a wide range of so-called \emph{well-balanced} schemes has been developed. There is a wide literature on the preservation of the classical lake at rest equilibrium with $\baru = 0$ and $h+\zb = \text{const}$. But also moving flow equilibria have been considered, such as $h\baru = \text{const}$ and $\frac12 \baru^2 + g(h+\zb) = \text{const}$ in the one-dimensional case, see for example~\cite{NXS07}, or the geostrophic equilibrium $g \nabla (h+b) = f (\barv,-\baru)^\top$ with Coriolis parameter $f$, see for example~\cite{GCGP24}.
    Since the barotropic-baroclinic splitting scheme preserves the solutions of the shallow water equations, see Proposition~\ref{P:choice} \eqref{item:proof_choice_SWE}, it inherits not only the preservation of the lake at rest equilibrium but all the well-balancing properties of the chosen shallow water solver~\eqref{eq:FV-SW}.
\end{remark}}

\begin{proposition} \label{P:Rusanov}
    Proposition~\ref{P:choice} is also valid for the Rusanov mass flux
    \begin{equation} \label{eq:Rusanov}
        \Fhajp = \dfrac{h_{\alpha, j}^{n} \sigma_{\alpha, j}^{n} + h_{\alpha, j+1}^{n} \sigma_{\alpha, j+1}^{n}}{2} - A_{j+\frac12} \dfrac{h_{\alpha, j+1}^{n} -h_{\alpha, j}^{n} }{2} \quad \text{with} \quad A_{j+\frac12}=\max_{\substack{1 \leq \alpha \leq N, \\ 0 \leq k \leq 1}}  |\sigma_{\alpha, j+k}^{n}|
    \end{equation}
    with the difference that $\Dtn \leq \dfrac{\Delta x}{\max_{\alpha,j} |\sigma_{\alpha,j}^{n}|}$ is enough to ensure the non-negativity of \revised{the} water height.
\end{proposition}

\begin{proof}
    \eqref{item:proof_choice_non-negativity} \revised{We start with} %Now we investigate 
    the non-negativity condition~\eqref{eq:positivity_prediction}. The worst case is when $\Fhajp>0$ and $\Fhajm<0$. Then
    \begin{align*}
        % \hajs &= 
        &\hajn - \dfrac{\Dtn}{\Delta x} \left( \pos{\Fhajp} + \nega{\Fhajm} \right) \\ 
        = &\, \hajn - \dfrac{\Dtn}{\Delta x} \left( h_{\alpha, j}^{n} \frac{\sigma_{\alpha, j}^{n}+A_{j+\frac12}}{2} + h_{\alpha, j+1}^{n} \dfrac{\sigma_{\alpha, j+1}^{n}-A_{j+\frac12}}{2} %\right. \\
        % & \qquad \left. 
        - h_{\alpha, j-1}^{n} \dfrac{\sigma_{\alpha, j-1}^{n}+A_{j-\frac12}}{2} - h_{\alpha, j}^{n} \dfrac{\sigma_{\alpha, j}^{n}-A_{j-\frac12}}{2} \right) \\
        = &\, \hajn \left( 1 - \dfrac{\Dtn}{\Delta x} \dfrac{A_{j+\frac12}+A_{j-\frac12}}{2} \right) + h_{\alpha, j+1}^{n} \dfrac{\Dtn}{\Delta x}  \dfrac{A_{j+\frac12}-\sigma_{\alpha, j+1}^{n}}{2} %\\
        % & \qquad 
        +h_{\alpha, j-1}^{n} \dfrac{\Dtn}{\Delta x} \dfrac{\sigma_{\alpha, j-1}^{n}+A_{j-\frac12}}{2}
    \end{align*}
    With $A_{j\pm1/2} = \max( |\sigma_{\alpha, j}^{n}|,  |\sigma_{\alpha, j\pm1}^{n}|)$ the coefficients in front of $h_{\alpha, j+1}^{n}$ and $h_{\alpha, j-1}^{n}$ are non-negative.  The coefficient in front of $\hajn$ is non-negative as soon as
    \begin{equation} \label{eq:dt_rusanov}
        \Dtn \leq \dfrac{\Delta x}{\max_{\alpha,j} |\sigma_{\alpha,j}^{n}|} \leq  \dfrac{2 \Delta x}{A_{j+\frac12}+A_{j-\frac12}}.
    \end{equation}
    
    \eqref{item:proof_choice_SWE},~\eqref{item:proof_choice_wb} \revised{We consider again a purely barotropic setting. Then, it holds $\sigma_{\alpha,j}^n=0$, $A_{j\pm \frac12}=0$ and hence $\Fha_{j\pm\frac12}=0$. Consequently,} the water height does not vary during the baroclinic step 
    % on the lake at rest or 
    % shallow water solution because $\sigma_{\alpha,j}^n=0$ and $A_{j\pm \frac12}=0$ in this case. As in the proof before, $\Fha_{j\pm\frac12}=0$ and thus 
    \revised{In addition,} $F^{\ha\ua}_{j\pm\frac12}$ and the exchange terms vanish. Thus, also the velocity does not change during the baroclinic step. The barotropic step preserves 
    % the lake at rest solution and
    the shallow water solution by hypothesis,
    % . In the case of the shallow water solution 
    $h$ and $\bar{u}$ vary but since the redistribution gives $\sigma_{\alpha,j}^{n+1} = 0$ it holds
    % $u_{\alpha,j}^{n+1} = \uajn = 0$ in the case of the lake at rest and 
    $u_{\alpha,j}^{n+1} = \bar{u}_j^{n+1}$.
    % in the case of the shallow water solution.
    \revised{As in the proof before, the overall scheme preserves the lake at rest equilibrium.}
    %inherits the well-balancing properties of the shallow water solver.}
    
    \eqref{item:proof_choice_DEI} We now turn to the variation of potential energy~\eqref{eq:Ep} in the prediction step. Under the hypothesis~\eqref{eq:dt_rusanov}
    \begin{align*}
      h_j^\star&= \hjn \left( 1 - \dfrac{\Dtn}{\Delta x} \dfrac{A_{j+\frac12}+A_{j-\frac12}}{2} \right)  + \hjpn \dfrac{\Dtn}{\Delta x}  \dfrac{A_{j+\frac12}}{2} +h_{ j-1}^{n} \dfrac{\Dtn}{\Delta x} \dfrac{A_{j-\frac12}}{2}
    \end{align*}
    is a convex combination of $\hjn$, $\hjpn$ and $h_{j-1}^{n}$. Hence we get with Jensen's inequality
    \begin{align*}
      (h_j^\star)^{2}&\leq (\hjn)^{2} \left( 1 - \dfrac{\Dtn}{\Delta x} \dfrac{A_{j+\frac12}+A_{j-\frac12}}{2} \right)  + (\hjpn)^{2} \dfrac{\Dtn}{\Delta x}  \dfrac{A_{j+\frac12}}{2}  +(h_{j-1}^{n})^{2} \dfrac{\Dtn}{\Delta x} \dfrac{A_{j-\frac12}}{2} \\
      &= (\hjn)^{2} + \dfrac{\Dtn}{\Delta x} \left[ A_{j+\frac12} \dfrac{(\hjpn)^{2}-(\hjn)^{2}}{2} -  A_{j-\frac12} \dfrac{(\hjn)^{2}-(h_{ j-1}^{n})^{2}}{2} \right]
    \end{align*}
    and we obtain~\eqref{eq:Ep} with $F^{E^{p}}_{j+\frac12}= A_{j+\frac12} \dfrac{ E^{p,n}_{j+1}-E^{p,n}_{j}}{2}$.
\end{proof}

\section{Numerical experiments} \label{S:testcases}
In this section we perform several numerical test cases. We investigate the convergence behaviour of the splitting towards an analytical solution of the hydrostatic Euler equations, the reduction of computational cost and the well-balancing properties of the scheme for the lake at rest and for the geostrophic equilibrium.

The results presented here are all obtained using the Rusanov mass flux~\eqref{eq:Rusanov} \revised{in the baroclinic prediction step}. In addition, we ran the same test cases again with the mass flux~\eqref{eq:choice}, which gave us very similar results.

\revised{For the barotropic step, we use the hydrostatic reconstruction~\cite{ABBKP04,CN17} to discretize the numerical fluxes $\fhkjpm$, $\fhukjpm$ and the bottom topography source term $\Sjndk$ in the shallow water solver~\eqref{eq:FV-SW}. This ensures the non-negativity of water heights, well-balancing of the lake at rest and a semi-discrete entropy inequality. 
% The intermediate bottom topography at the cell interface is defined as $z_{j+\frac12} := \max(z_{b,j}, z_{b,j+1})$. In addition, we define the left and right interface water heights as 
% \begin{equation*}
%     \hjdkL := \max \left\{ \hjdk + z_{b,j} - z_{j+\frac12}, 0 \right\}
%     \quad \text{and} \quad 
%     \hjdkR := \max \left\{ \hjpdk + z_{b,j+1} - z_{j+\frac12}, 0 \right\}.
% \end{equation*}
% The left and right interface velocities remain the ones from the cell: $\bar{u}_{j+\frac12-}^{n+\frac12,k} := \barujdk$, $\bar{u}_{j+\frac12-}^{n+\frac12,k} := \bar{u}_{j+1}^{n+\frac12,k}$.
% We use these reconstructed variables to compute the numerical fluxes $\fhkjpm$, $\fhukjpm$ using the HLL-flux~\cite{Toro}.
% The source term is discretized as 
% \begin{equation*}
%     \Sjndk := \begin{pmatrix} 0 \\ \frac{g}{2} \left( \hjdkL \right)^2 - \frac{g}{2} \left( \hjdkR \right)^2 \end{pmatrix}.
% \end{equation*}
}

\subsection{Convergence study: analytical solution for the 2d hydrostatic Euler equations} \label{sec:analytical_sol_2d_euler}
According to \cite{BBSM2013} stationary solutions of the 2d ($x,z$) hydrostatic and incompressible Euler equations are given by
\begin{align*}
    z_b(x) &= \bar{z_b} - h(x) - \frac{\alpha^2\beta^2}{2g\sin^2(\beta h(x))}, \\
    u(x,z) &= \frac{\alpha \beta}{\sin(\beta h(x))} \cos(\beta (z-z_b(x))), \\
    w(x,z) &= \alpha \beta \left( \frac{\partial z_b}{\partial x} \frac{\cos(\beta (z-z_b(x)))}{\sin(\beta h(x))} + \frac{\partial h}{\partial x} \frac{\sin(\beta (z-z_b(x))) \cos(\beta h(x))}{\sin^2(\beta h(x))} \right)
\end{align*}
for $\alpha, \beta \in \R$ and $h \in C^1(\Omega)$ non-negative such that $\sin(\beta h(x)) \neq 0$ for all $x \in \Omega$. 
For our testcase we choose the domain $\Omega = [-5,5]$, the parameters $\alpha = 0.1$, $\beta = 1$, $\bar{z_b} = 2$ and $h(x) = 2-e^{-x^2}$ and use this as initial data. In this setting the Froude number is $\mathrm{Fr}=0.04$. We compute until final time $t=60 \, \mathrm{s}$ using the barotropic-baroclinic splitting with implicit treatment of the exchange terms and subcycling strategy and investigate the behaviour of the numerical solution when refining the cells in $x$-direction and refining the layers. Figure~\ref{fig:2d_euler_velocity_profiles} illustrates the velocity field and the vertical velocity profile of $u$ at various positions $x$. The upper section of the figure displays the velocity field of the solution, while the red lines indicate the locations where the convergence behavior of the vertical velocity profile of $u$ is investigated. This convergence behavior is presented in the second row of Figure~\ref{fig:2d_euler_velocity_profiles}. The $L^1$-errors and experimental orders of convergence are given in Table~\ref{tab:EOC_euler}. \revised{The total error is plotted in Figure~\ref{fig:2d_Euler_conv} and first order convergence can be observed.}
% In Figure~\ref{fig:convergence_plot_ref_x} we plotted the reduction of the total error $\| (h_\text{err},hu_\text{err})^\top \|_2$ when only $\Delta x$ is reduced and the number of layers is kept constant and vice versa (see Figure~\ref{fig:convergence_plot_ref_z}). As expected, the scheme only converges with first order if the cells and the layers are refined both at the same time (see Figure~\ref{fig:convergence_plot_ref_xz}). 

% If we modify this testcase such that it becomes a non stationary testcase by adding a smooth perturbation of the water height, the scheme still converges with first order. 

\begin{figure}[htbp]
    \centering
    \includegraphics[width=0.75\textwidth]{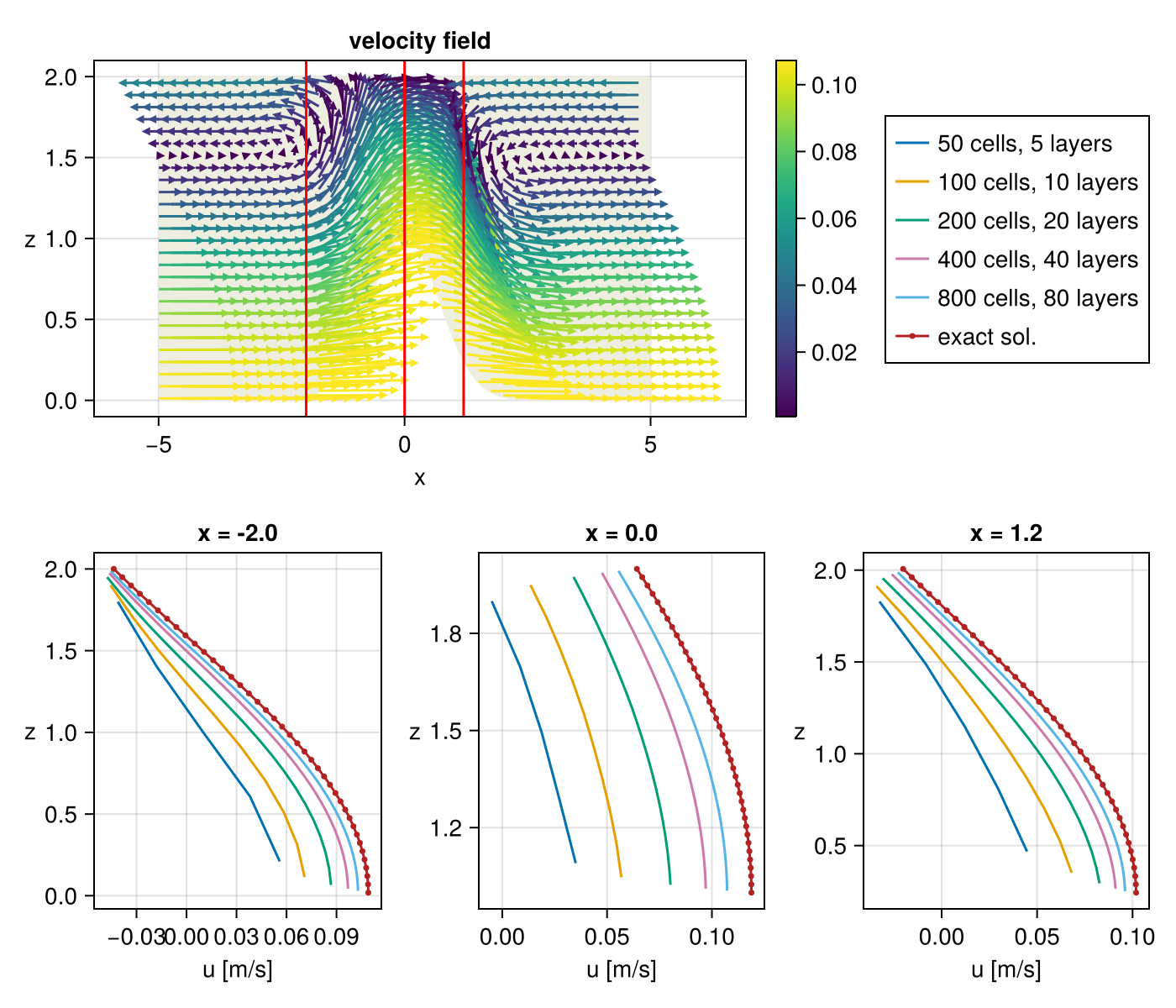}
    \caption{Upper row: The velocity field for testcase~\ref{sec:analytical_sol_2d_euler}, along with the highlighted positions for a detailed examination of the vertical velocity profile of 
    $u$. Lower row: The convergence behavior of the vertical velocity profiles of $u$ at various positions $x$.}
    \label{fig:2d_euler_velocity_profiles}
\end{figure}

\begin{table}[htbp]
    \centering
    \begin{tabular}{ccccc}
        \hline
        \input{EOC_table_combined_L1_splitting_until_1600NX_160layers_t=60.0_Fr=0.04}\\
        \hline
    \end{tabular}
    \caption{$L^1$-errors and experimental orders of convergence for testcase~\ref{sec:analytical_sol_2d_euler}.}
    \label{tab:EOC_euler}
\end{table} 

\begin{figure}
    \centering
    \includegraphics[width=0.4\textwidth]{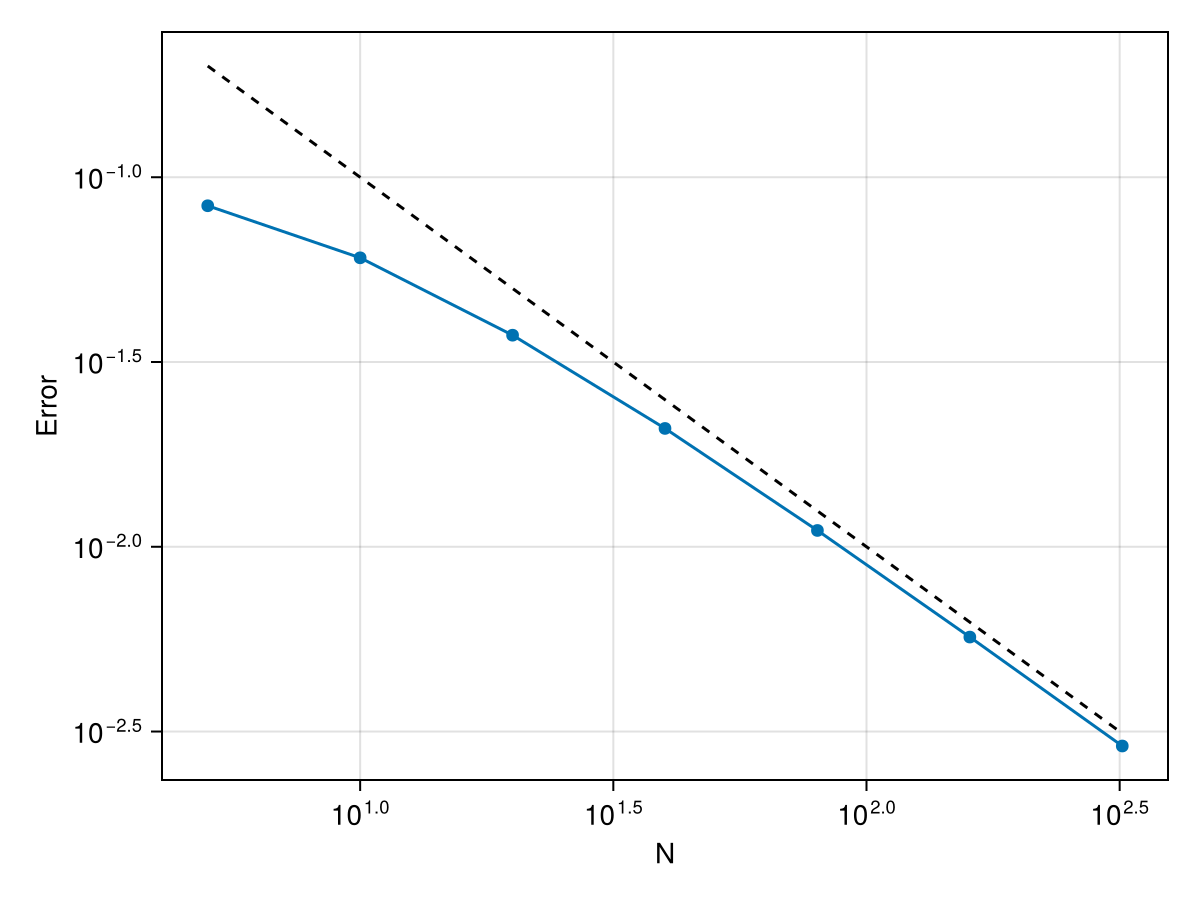}
    \caption{\revised{Total $L^1$ -error depending on number of layers ($N$) and cells ($10N$), dashed line is first order reference.}}
    \label{fig:2d_Euler_conv}
\end{figure}

To examine numerical diffusion, we introduce a column of tracer on the left side of the bump in the bottom topography in the initial condition and monitor the $L^2$-norm of the tracer over time. The results are presented in Figure~\ref{fig:L2_tracer_Euler}. This investigation reveals that, in this testcase, the original scheme without splitting exhibits lower diffusion compared to the barotropic-baroclinic splitting.

\begin{figure}[htbp]
    \begin{subfigure}[t]{0.48\textwidth}
        \includegraphics[width=\textwidth]{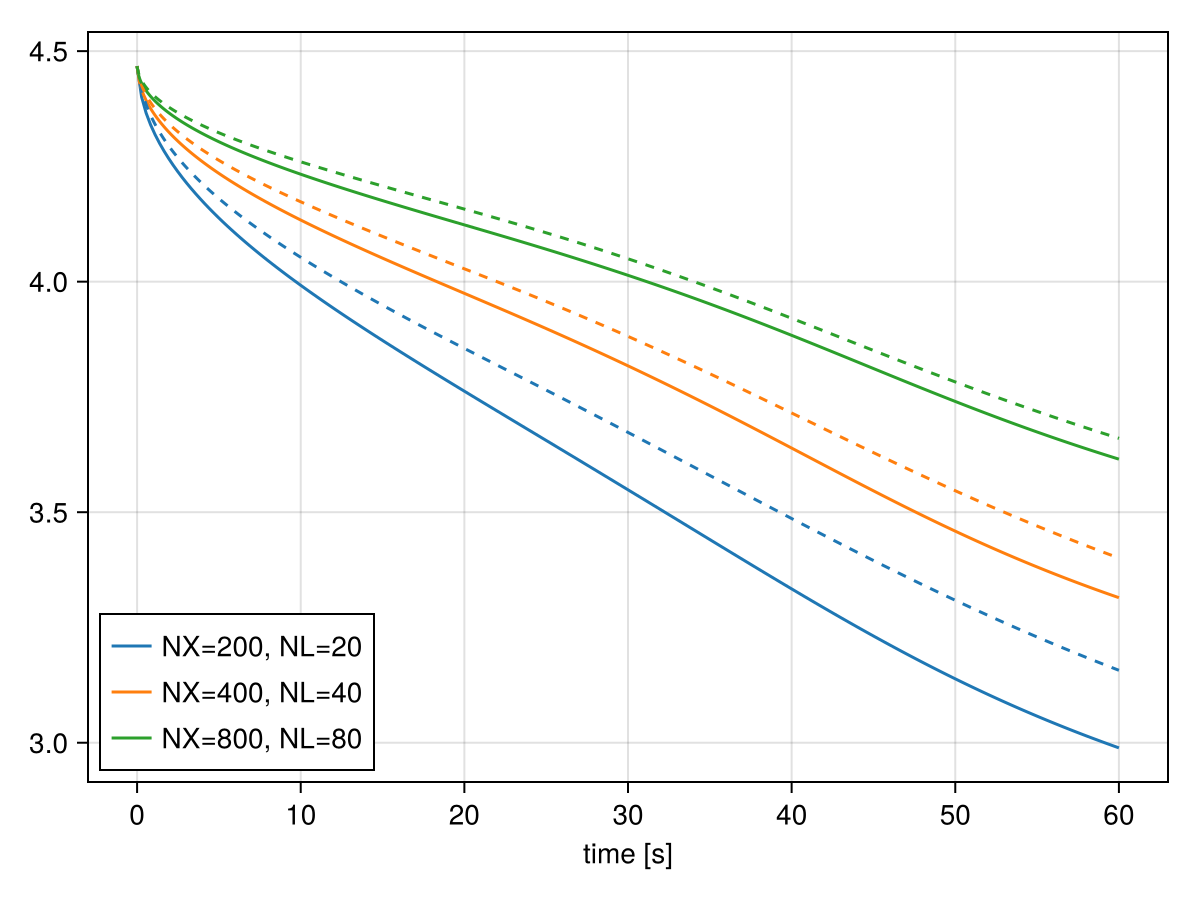}
        \caption{$L^2$-norm of the tracer in testcase~\ref{sec:analytical_sol_2d_euler}.}
        \label{fig:L2_tracer_Euler}
    \end{subfigure}
    \hfill
    \begin{subfigure}[t]{0.48\textwidth}
        \includegraphics[width=\textwidth]{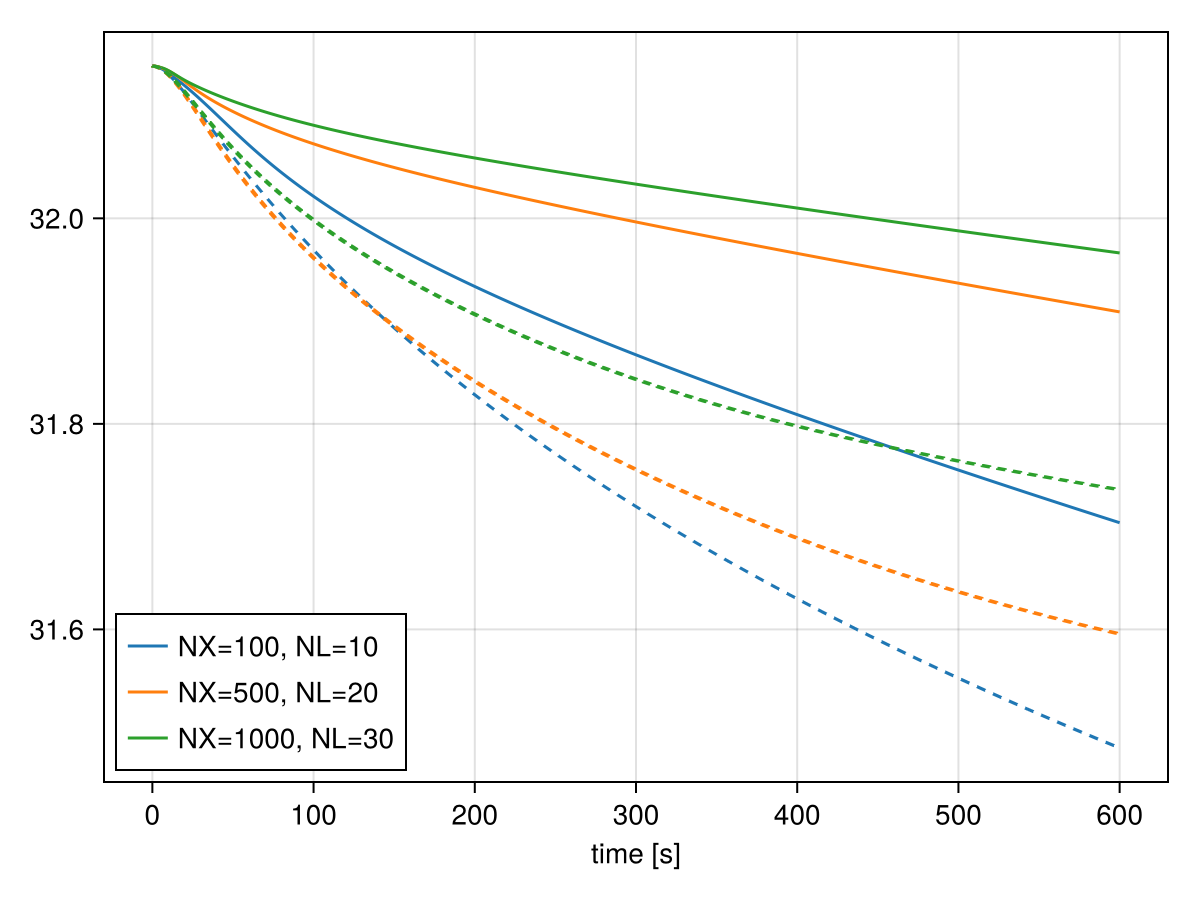}
        \caption{$L^2$-norm of the tracer in testcase~\ref{sec:wind_driven_cavity}.}
        \label{fig:L2_tracer_wind-driven-cavity}
    \end{subfigure}
    \caption{$L^2$-norm of the tracer to examine numerical diffusion. The solid line represents the splitting method, while the dashed line indicates the scheme without splitting.}
\end{figure}

\subsection{Computational cost} \label{S:CC}
For the comparison of the different schemes we perform several runs of the respective scheme and measure the averaged runtime and the total error. We compare the barotropic-baroclinic splitting with or without subcycling strategy, with implicit or explicit treatment of exchange terms, and the unsplit scheme~\cite{ABPSM11a}.

As a testcase we use the analytical solution for the 2d hydrostatic Euler equations given in section~\ref{sec:analytical_sol_2d_euler} until final time $t=1 \, \mathrm{s}$. We perform the simulation three times with different Froude numbers given by different values for $\alpha$. The averaged (over 10 runs) measured runtime for the different schemes is plotted against the total error $\| (h_\text{err}, hu_\text{err})^\top \|_2$ in Figure~\ref{fig:cost-error}. Note that the gain in computational cost is especially large in the low Froude setting. In addition, in low Froude situations the subcycling strategy has a bigger impact on the computational cost than the explicit or implicit treatment of the exchange terms. 
For a 1D problem computed with a first-order scheme one would expect the computational cost to grow by a factor of around 4 when halving the grid size (and thus also the error) because of the double amount of cells and the double amount of time steps that have to be computed. But the simulation shows that especially for low Froude simulations and the splitting scheme the factor is lower (see Figure~\ref{fig:cost-error}).

\begin{figure}[htbp]
    \centering
    \includegraphics[width=0.9\textwidth]{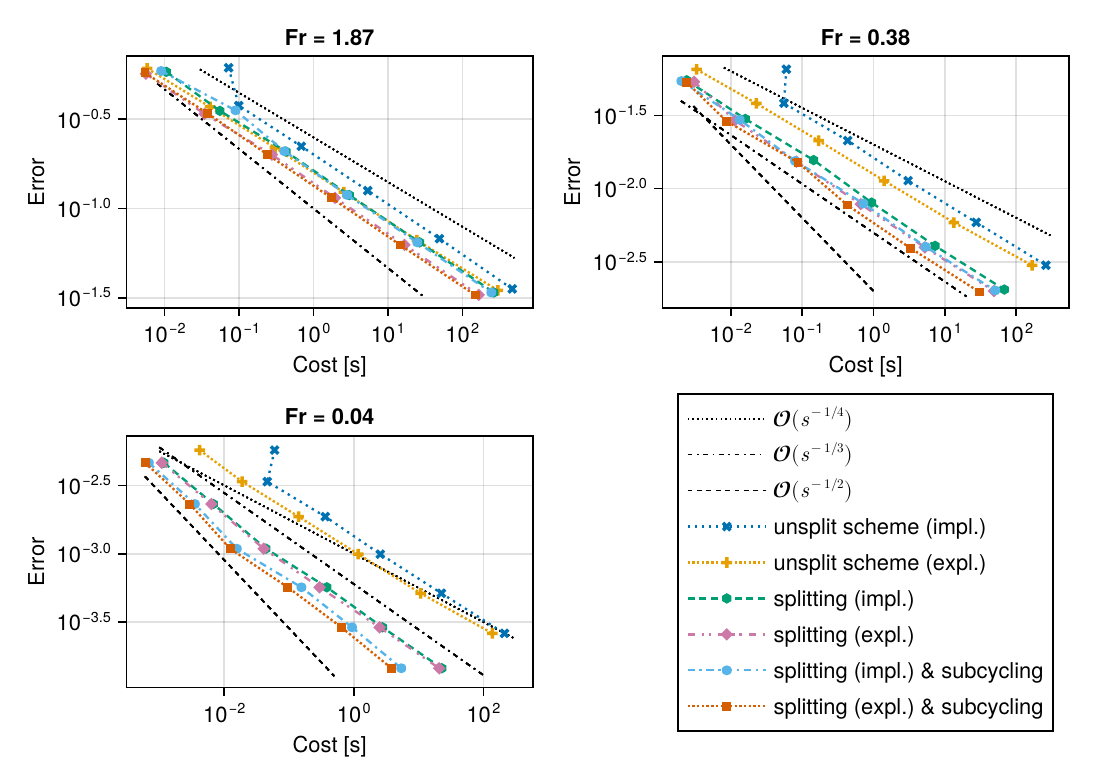}
    \caption{Comparison of the computational cost and the total error $\| (h_\text{err}, hu_\text{err})^\top \|_2$ for testcase~\ref{sec:analytical_sol_2d_euler} with different Froude numbers. Top left: $\mathrm{Fr}=1.87$ (using $\alpha = 5$), top right: $\mathrm{Fr}=0.38$ (using $\alpha = 1$), bottom left: $\mathrm{Fr}=0.04$ (using $\alpha = 0.1$).}
    \label{fig:cost-error}
\end{figure}

Furthermore we investigated the ratio of the computational cost of the barotropic-baroclinic splitting with subcycling and the scheme without splitting for different Froude numbers. The results are plotted in Figure~\ref{fig:cost_ratio} and show that the gain in computational cost is proportional to the Froude number.

\begin{figure}[htbp]
    \centering
        \includegraphics[width=0.6\textwidth]{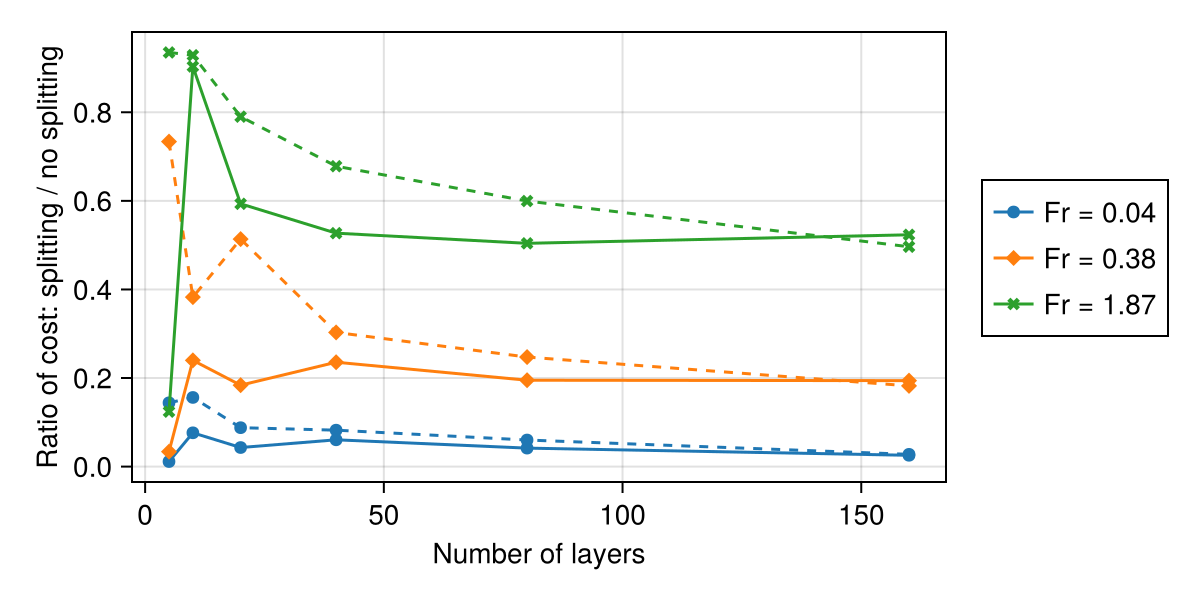}
    \caption{Comparison of the computational cost ratio (cost of barotropic-baroclinic splitting with subcycling over cost of the scheme without splitting) for various Froude numbers. The solid line represents the implicit version of the scheme, while the dashed line indicates the explicit version.}
    \label{fig:cost_ratio}
\end{figure}

\subsection{Wind driven cavity} \label{sec:wind_driven_cavity}

In this section, we perform the test case of a two-layer fluid flow, initially at rest, subject to wind stress as it was performed in \cite{ABPSM11b}, but with constant density. 
\revised{For this testcase, we include the effects of viscosity, bottom friction and wind forcing. We follow~\cite{ABPSM11a} and neglect the effects of horizontal viscosity. Then, the layer-averaged Navier-Stokes equations with bottom friction and wind forcing read~\cite{ABPSM11a}
\begin{equation*}
    \begin{cases}
        \ddt{h} + \sumaN \dx (\ha \ua) = 0,  \\  
        \dt (\ha \ua) + \dx \left(h \ua^2 \right) + g \ha \ddx{h} = -g \ha \ddx{\zb} + \uap\Gap - \uam\Gam - \kappa_\alpha \ua + W_\alpha \tau_\text{w} \\
        \hspace{38ex} + \frac{2\nu_\alpha}{h_{\alpha+1} + \ha} (u_{\alpha+1} - \ua) - \frac{2\nu_{\alpha-1}}{\ha +h_{\alpha-1}} (\ua - u_{\alpha-1}),
    \end{cases}
\end{equation*}
where the mass exchanges $\Gap$ and the interface velocities $\uap$ are the defined as before in~\eqref{eq:MSW1D-closed} and~\eqref{eq:interu}, respectively. The parameters $\kappa_\alpha$ for the bottom friction, $W_\alpha$ for the wind stress, and $\nu_\alpha$ for the vertical viscosity are defined by
\begin{equation*} 
    \kappa_\alpha := \begin{cases} \kappa, & \text{if } \alpha = 1, \\ 0, & \text{if } \alpha > 1, \end{cases} 
    \quad \text{and} \quad
    W_\alpha := \begin{cases} 1, & \text{if } \alpha = N, \\ 0, & \text{if } \alpha < N, \end{cases}
    \quad \text{and} \quad
    \nu_\alpha := \begin{cases} \nu, & \text{if } \alpha = 1, \dots, N-1, \\ 0, & \text{if } \alpha = 0 \text{ or } \alpha = N, \end{cases}
\end{equation*}
where $\kappa$ is a Navier friction coefficient. The surface wind stress $\tau_\text{w}$ is computed via the expression 
\begin{equation} \label{eq:wind-law}
    \tau_\text{w} = C_D \frac{\rho_\text{air}}{\rho_0} |u_\text{wind}| u_\text{wind}
\end{equation}
where $u_\text{wind}$ is the wind velocity, $\rho_\text{air}$ the air density, $\rho_0$ a reference water density and $C_D = 1.3 \cdot 10^{-3}$ the wind drag coefficient taken from the literature.
}

\revised{Numerically, the viscosity and friction terms are applied implicitly in the correction step after the treatment of the exchange terms. Applying the exchange terms in the correction step~\eqref{eq:correctionstep} gives the intermediate values $\hat{u}_{\alpha,j}^{n+\frac12}$. To compute the new velocities we solve the tridiagonal linear system that reads
\begin{equation*}
    \left(\operatorname{diag}(h_1^{n+\frac12}, \dots, h_N^{n+\frac12}) - \Dtn M_j \right) \begin{pmatrix} u_{1,j}^{n+\frac12} \\ \vdots \\ u_{N,j}^{n+\frac12} \end{pmatrix}
    = \begin{pmatrix} h_{1,j}^{n+\frac12} \hat{u}_{1,j}^{n+\frac12} \\ \vdots \\ h_{N,j}^{n+\frac12} \hat{u}_{N,j}^{n+\frac12} \end{pmatrix}
\end{equation*}
with the matrix 
\begin{equation*}
    (M_j)_{i,k} = 
    \begin{cases}
        \frac{2 \nu}{\hajd + h_{\alpha-1,j}^{n+\frac12}} & \text{if } k = i-1, \\
        -\frac{2 \nu}{h_{\alpha+1,j}^{n+\frac12} + \hajd} - \frac{2 \nu}{\hajd + h_{\alpha-1,j}^{n+\frac12}} -\kappa_\alpha & \text{if } k = i, \\
        \frac{2 \nu}{h_{\alpha+1,j}^{n+\frac12} + \hajd} & \text{if } k = i+1, \\
        0 & \text{otherwise}.
    \end{cases}
\end{equation*}
}

\revised{Now,} we consider a fluid in a rectangular basin: the computational domain is given by the intervall $\Omega = [0,3]$ with reflecting boundary conditions. Initially the fluid is at rest with a total water height of $h = 1$\,m. The initial temperature (tracer) distribution is given by 
\begin{equation*}
    T_0(x,z) =\begin{cases}
    25 , & \text{if } z-b \geq \frac{h}{2}\\
    8, & \text{otherwise}.
    \end{cases}
\end{equation*}
We impose a constant uniform wind stress in $x$-direction. As done in \cite{ABPSM11b} we use a vertical viscosity of $\nu = 0.003 \,\mathrm{m}^2\mathrm{s}^{-1}$, the friction coefficient $\kappa =~0.1 \,\mathrm{m}\mathrm{s}^{-1}$ and a wind velocity of $6 \, \mathrm{m} \mathrm{s}^{-1}$. We let the scheme run until final time $t=600\,\mathrm{s}$. 

The numerical results for the tracers are shown in Figure~\ref{fig:wind_driven_cavity_tracer}. The barotropic-baroclinic splitting converges much faster than the scheme without splitting. Note that the splitting seems to be less diffusive than the scheme without splitting. To investigate this further we again tracked the $L^2$-norm of the tracer over time. The results are plotted in Figure~\ref{fig:L2_tracer_wind-driven-cavity}.  

\begin{figure}[htbp]
    \centering
    \begin{subfigure}[t]{0.32\textwidth}
        \includegraphics[width=\textwidth]{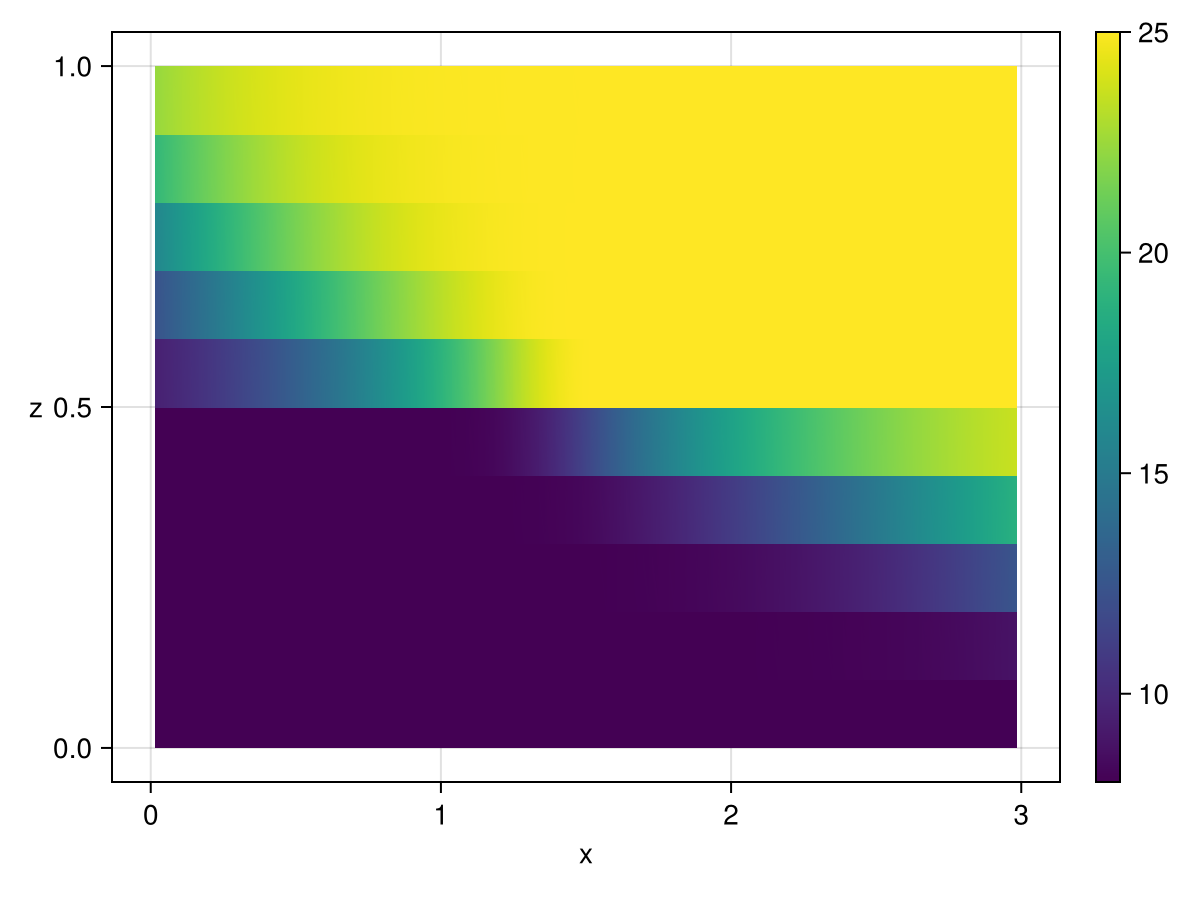}
        \caption{unsplit scheme, 100 cells, 10 layers}
    \end{subfigure}
    \begin{subfigure}[t]{0.32\textwidth}
        \includegraphics[width=\textwidth]{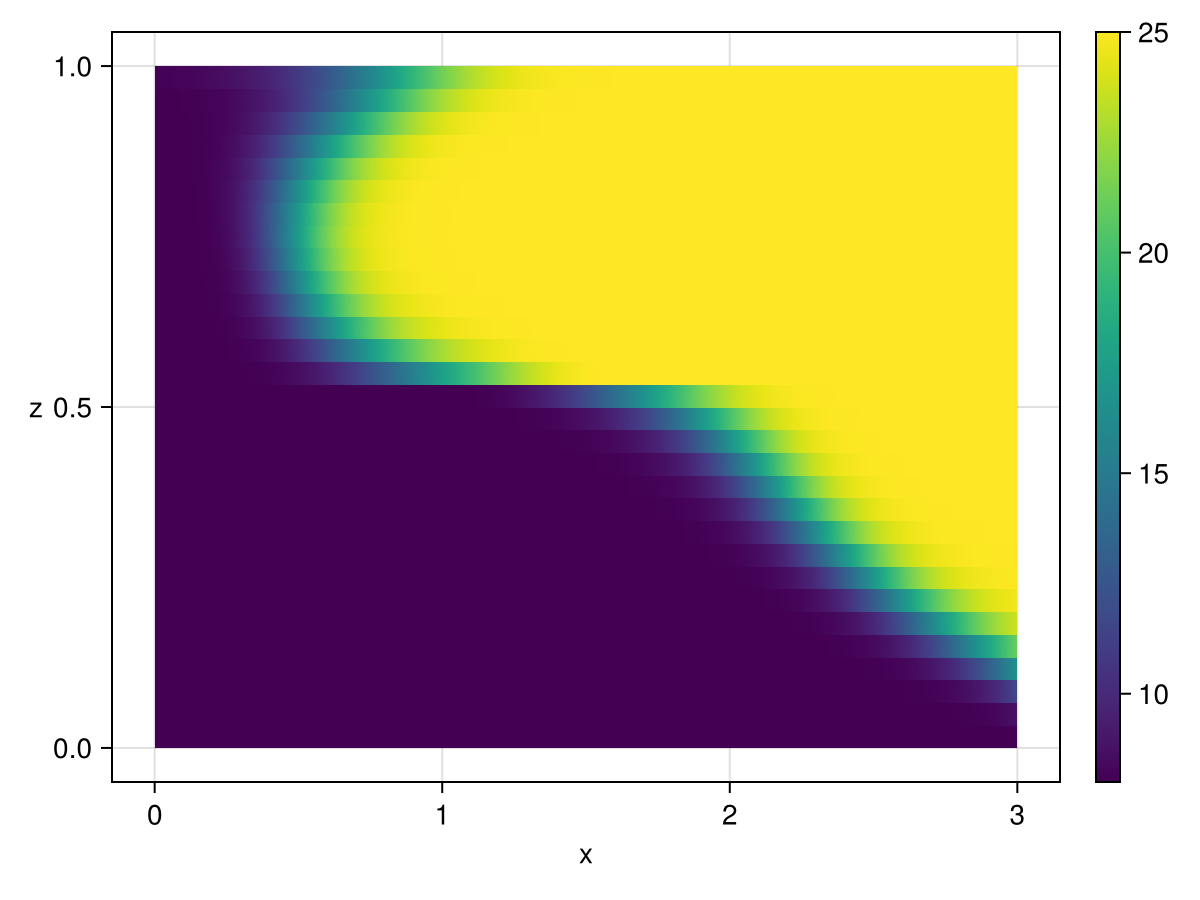}
        \caption{unsplit scheme, 1000 cells, 30 layers}
    \end{subfigure}
    % \hfill
    \begin{subfigure}[t]{0.32\textwidth}
        \includegraphics[width=\textwidth]{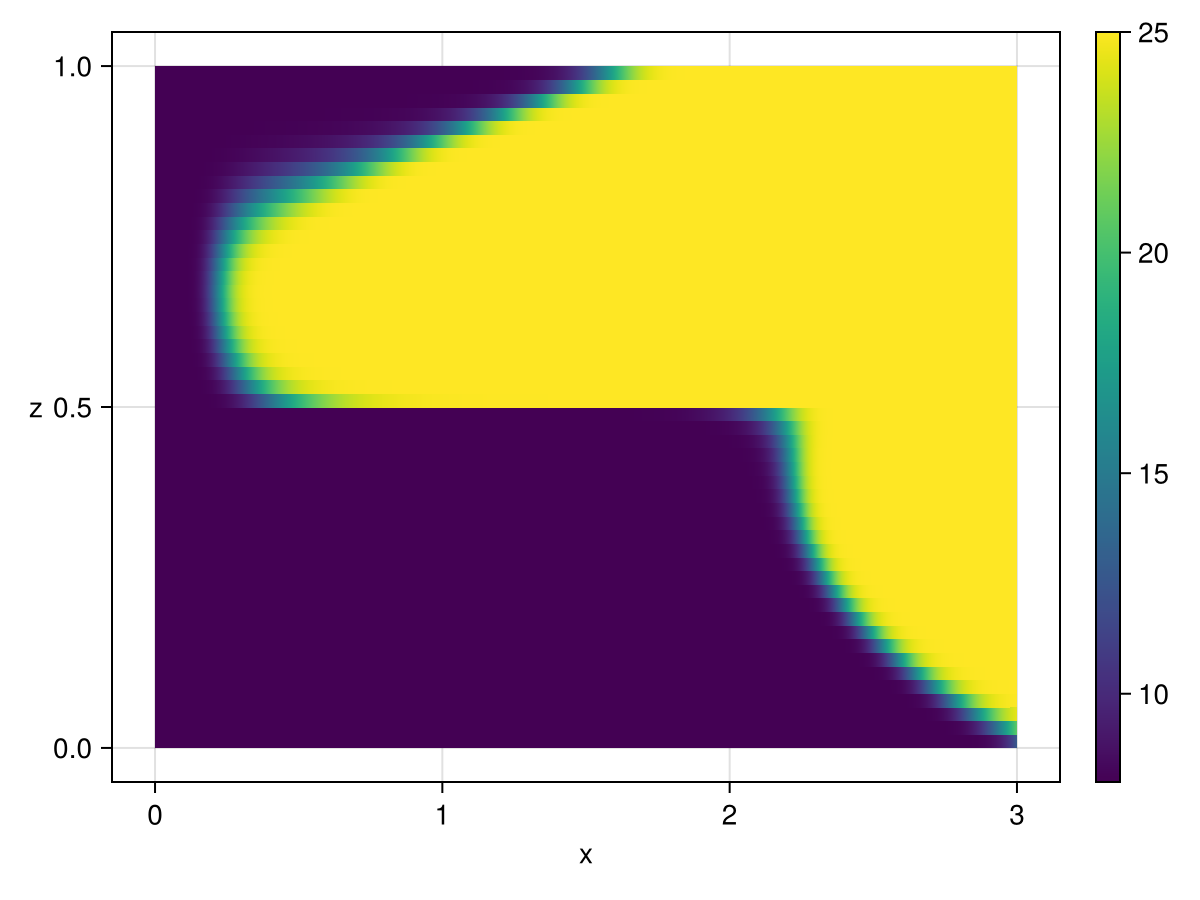}
        \caption{unsplit scheme, 8000 cells, 50 layers}
    \end{subfigure}
    % \hfill
    
    \begin{subfigure}[t]{0.32\textwidth}
        \includegraphics[width=\textwidth]{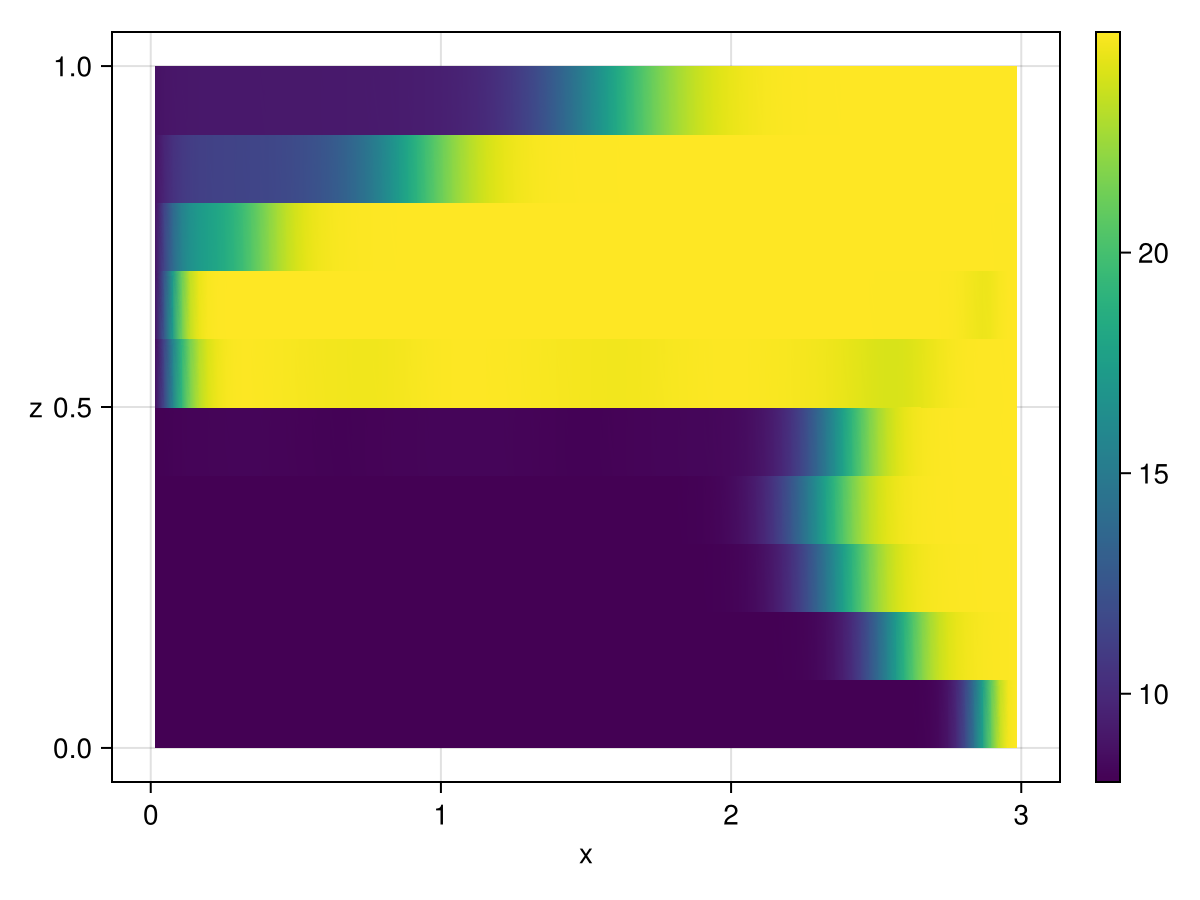}
        \caption{splitting, 100 cells, 10 layers}
    \end{subfigure}
    \begin{subfigure}[t]{0.32\textwidth}
        \includegraphics[width=\textwidth]{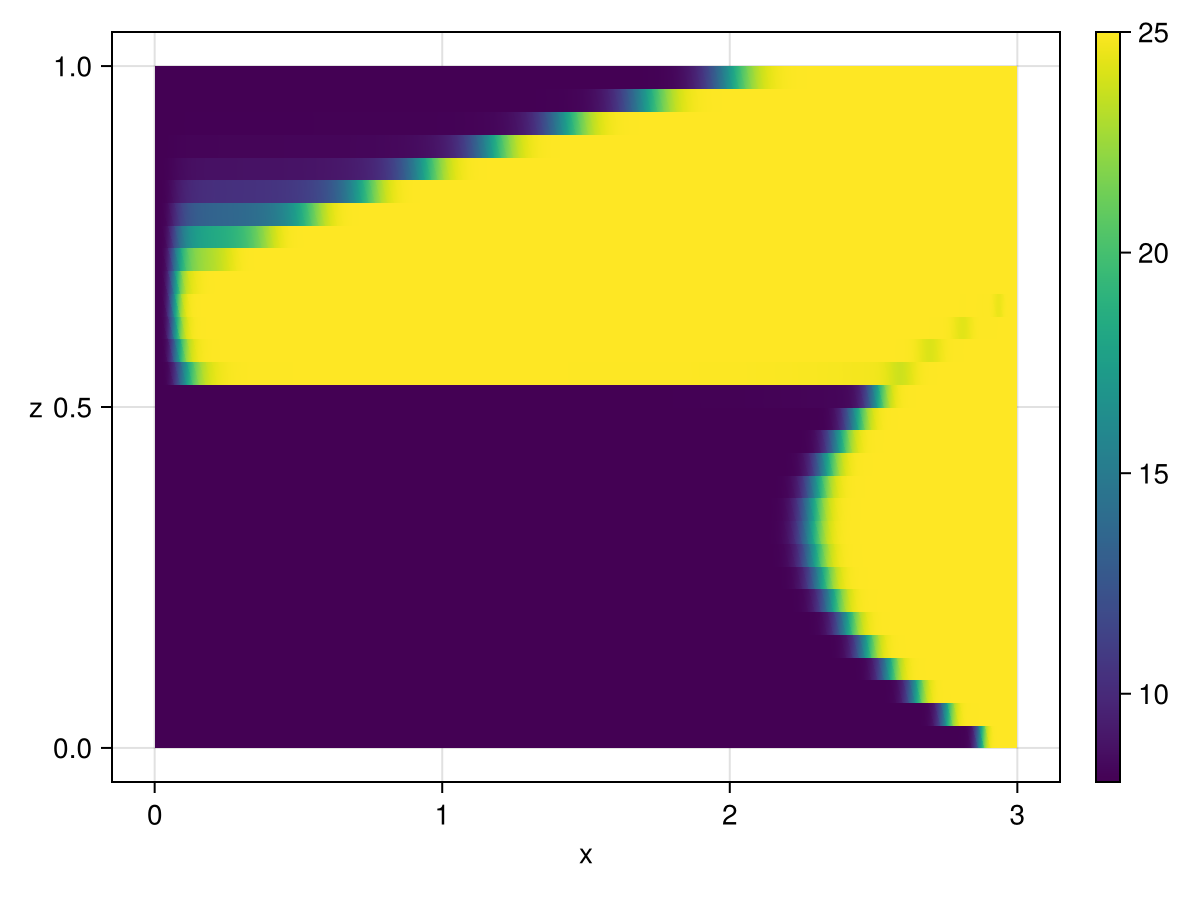}
        \caption{splitting, 1000 cells, 30 layers}
    \end{subfigure}
    % \hfill
    \begin{subfigure}[t]{0.32\textwidth}
        \includegraphics[width=\textwidth]{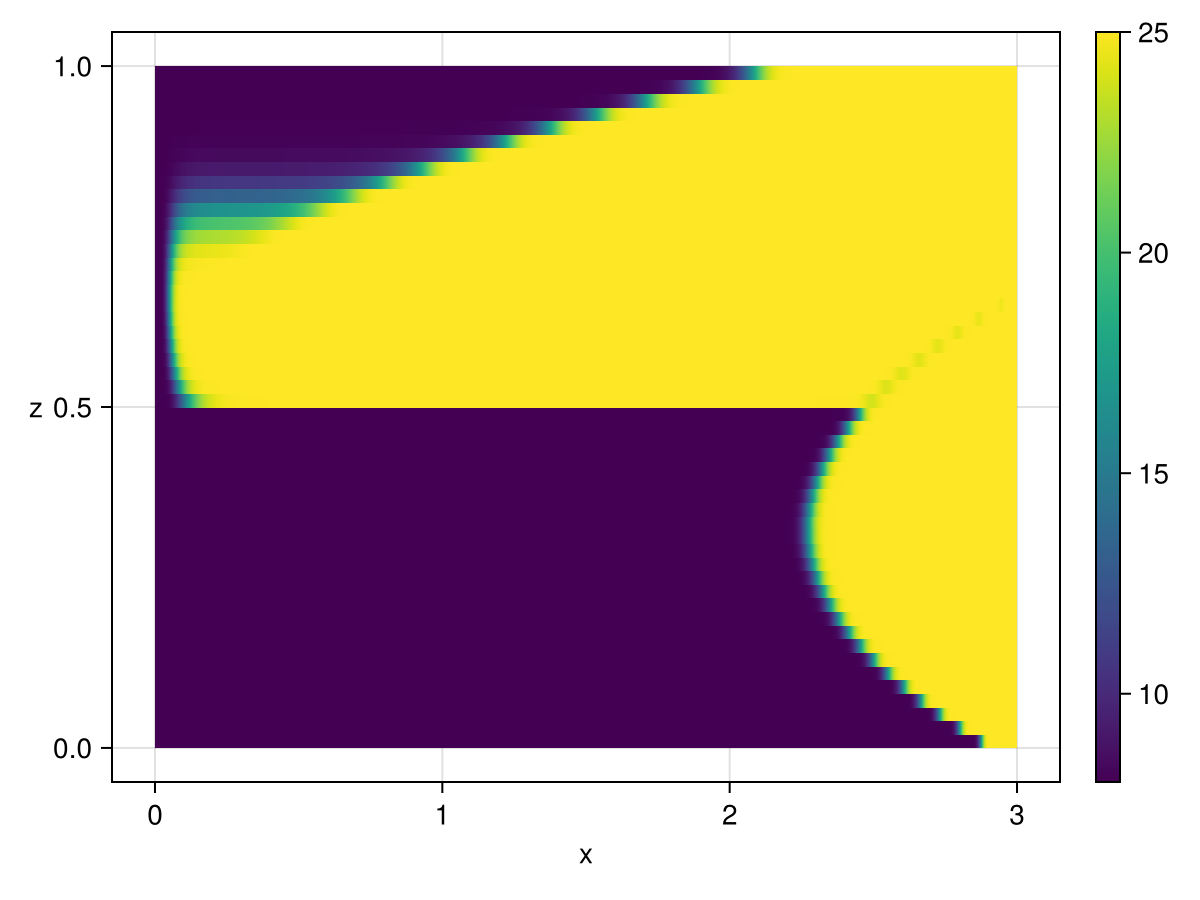}
        \caption{splitting, 4000 cells, 50 layers}
    \end{subfigure}
    % \hfill
    \caption{Tracer distribution of the wind driven cavity (testcase~\ref{sec:wind_driven_cavity}) at time $t = 600 \, \mathrm{s}$.}
    \label{fig:wind_driven_cavity_tracer}
\end{figure}

\revised{Visually the horizontal direction numerical diffusion in particular seems dimished in the splitting scheme. To sustain this interpretation, we run the same testcase in the presence of both horizontal and vertical viscosity with the same diffusion coefficient $\nu_\mathrm{hor} = \nu = 0.003 \, \,\mathrm{m}^2\mathrm{s}^{-1}$. The horizontal viscosity adds the diffusive source term $\partial_x (\ha \partial_x \ua) $ in the equation on $\ha \ua$. We discretize it in cell $j$ with the usual second order explicit centered approximation 
% $\frac{1}{\Delta x}\left(\frac{\hajpd+\hajd}{2}\frac{\uajdp-\uajd}{\Delta x}-\frac{\hajd+\hajmd}{2}\frac{\uajd-\uajdm}{\Delta x}\right)$
$\frac{1}{\Delta x}\left(\frac{\hajp+\haj}{2}\frac{\uajm-\uaj}{\Delta x}-\frac{\haj+\hajm}{2}\frac{\uaj-\uajm}{\Delta x}\right)$, where we omitted the superscript $n+\frac12$ for simplicity.
This discretization introduces a severe parabolic CFL condition $\Delta t \leq \beta \Delta x^2$ but we stick with an explicit scheme for the diffusive term. Indeed, we target large scale simulations on coarse grids for which it is documented that the restriction of timestep is controlled by advection or internal gravity waves, and not by horizontal viscosity, see~\cite{GA08} and~\cite{LDMDMH15}. This consideration explains why the grid is less refined on the results presented on Figure~\ref{fig:wind_driven_cavity_tracer_hor_visc=0.003}.}
\revised{
It appears that the horizontal viscosity plays an important role in this testcase because it smoothens the velocity fields near the vertical wall where the vertical velocities are the largest. With the unsplit scheme~\cite{ABPSM11b}, the horizontal numerical diffusion on a coarse grid is large and the results look like what is obtain with a $\nu = 0.003 \, \,\mathrm{m}^2\mathrm{s}^{-1}$ as a physical horizontal diffusion; while in the split scheme the difference is observed even on the coarsest grid because the numerical diffusion is very low.
}

% \sout{So far, horizontal viscosity was neglegted in the systems. 
% We add the horizontal viscosity ($\nu_\mathrm{hor} = \nu = 0.003 \, \,\mathrm{m}^2\mathrm{s}^{-1}$) both to the splitting and the scheme without splitting. The results are plotted in Figure~\ref{fig:wind_driven_cavity_tracer_hor_visc=0.003}.
% Since the horizontal viscosity is described by second order terms it gives a severe time step restriction depending on $\Delta x^2$. Thus the grids in the simulations are not as refined as before. But again, the splitting converges faster than the scheme without splitting.}

% with horizontal viscosity = 0.003
\begin{figure}[htbp]
    \centering
    \begin{subfigure}[t]{0.32\textwidth}
        \includegraphics[width=\textwidth]{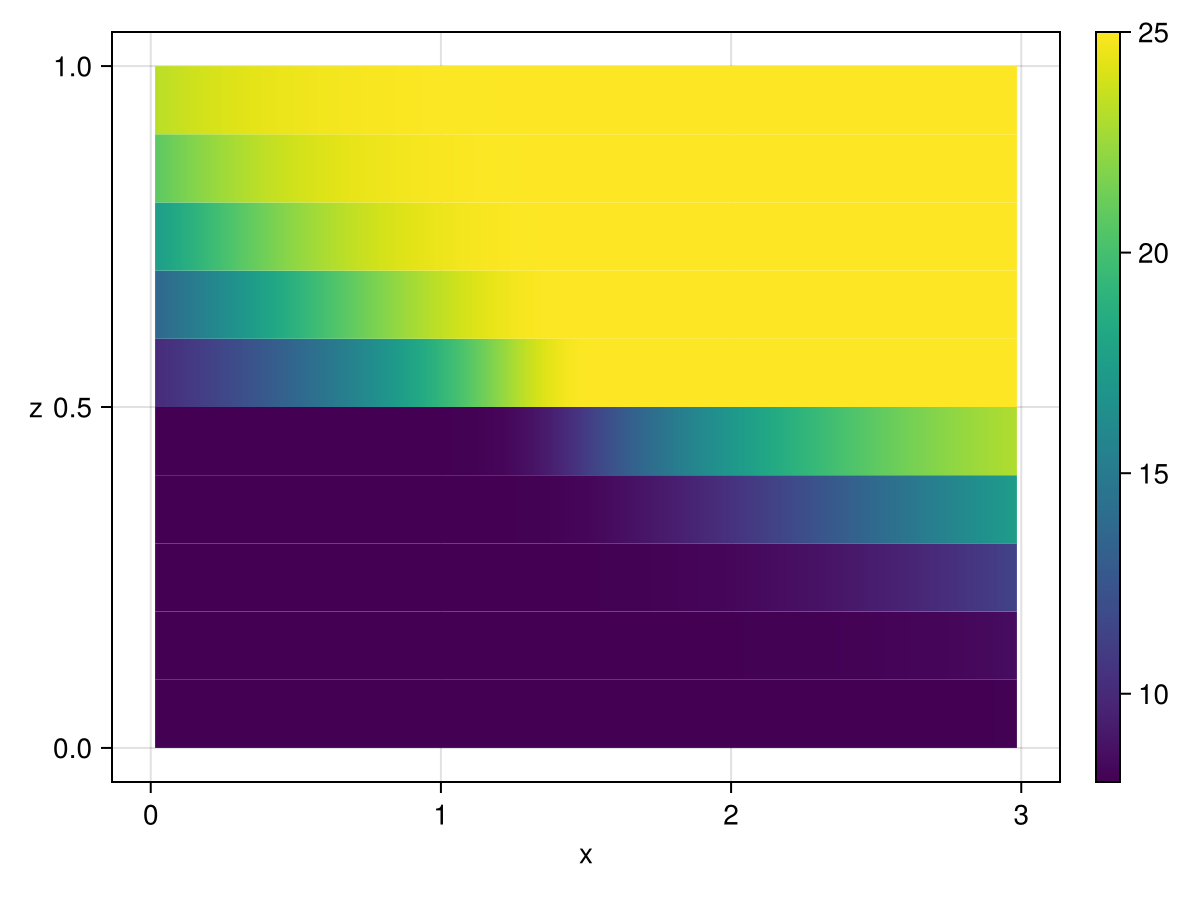}
        \caption{\revised{unsplit scheme, 100 cells, 10 layers}}
    \end{subfigure}
    \begin{subfigure}[t]{0.32\textwidth}
        \includegraphics[width=\textwidth]{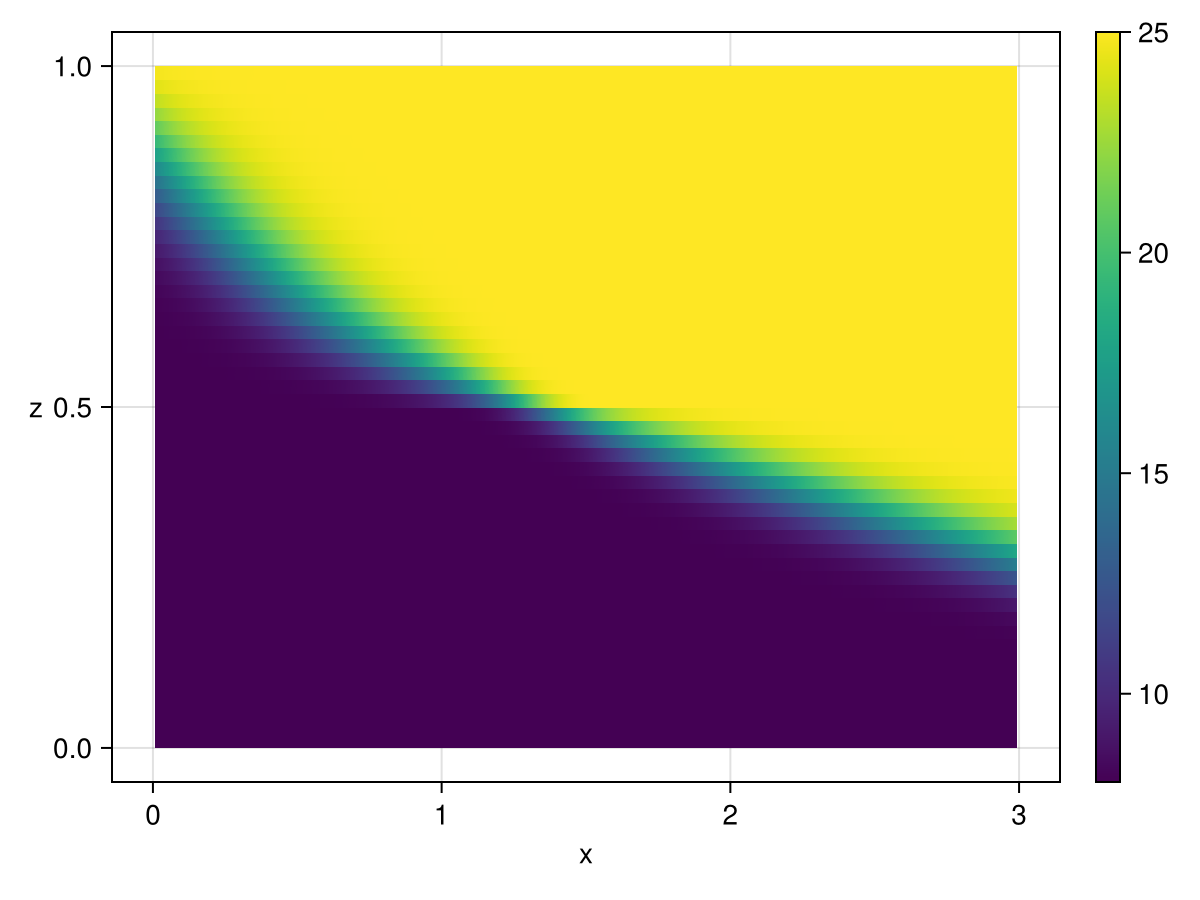}
        \caption{unsplit scheme, 200 cells, 50 layers}
    \end{subfigure}
    % \hfill
    \begin{subfigure}[t]{0.32\textwidth}
        \includegraphics[width=\textwidth]{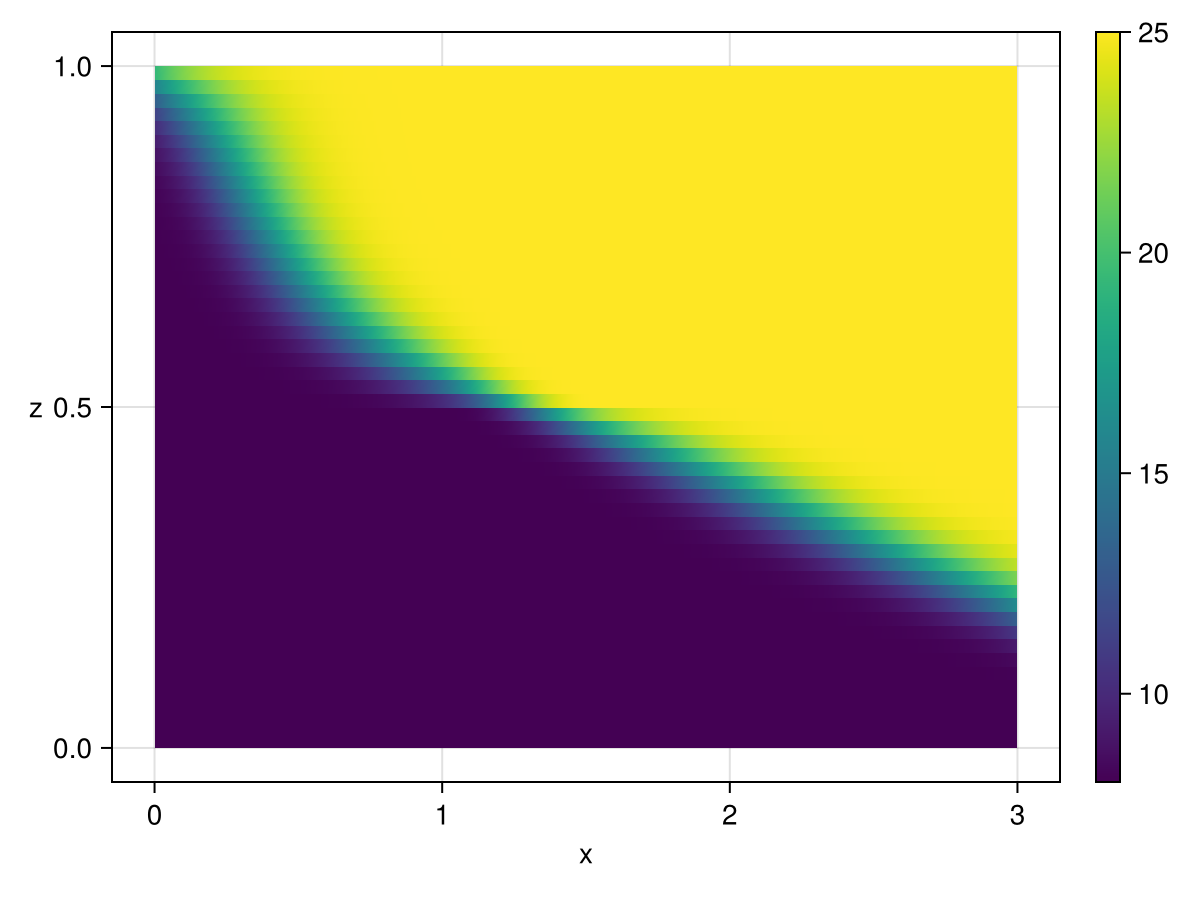}
        \caption{unsplit scheme, 800 cells, 50 layers}
    \end{subfigure}
    % \hfill
    
    \begin{subfigure}[t]{0.32\textwidth}
        \includegraphics[width=\textwidth]{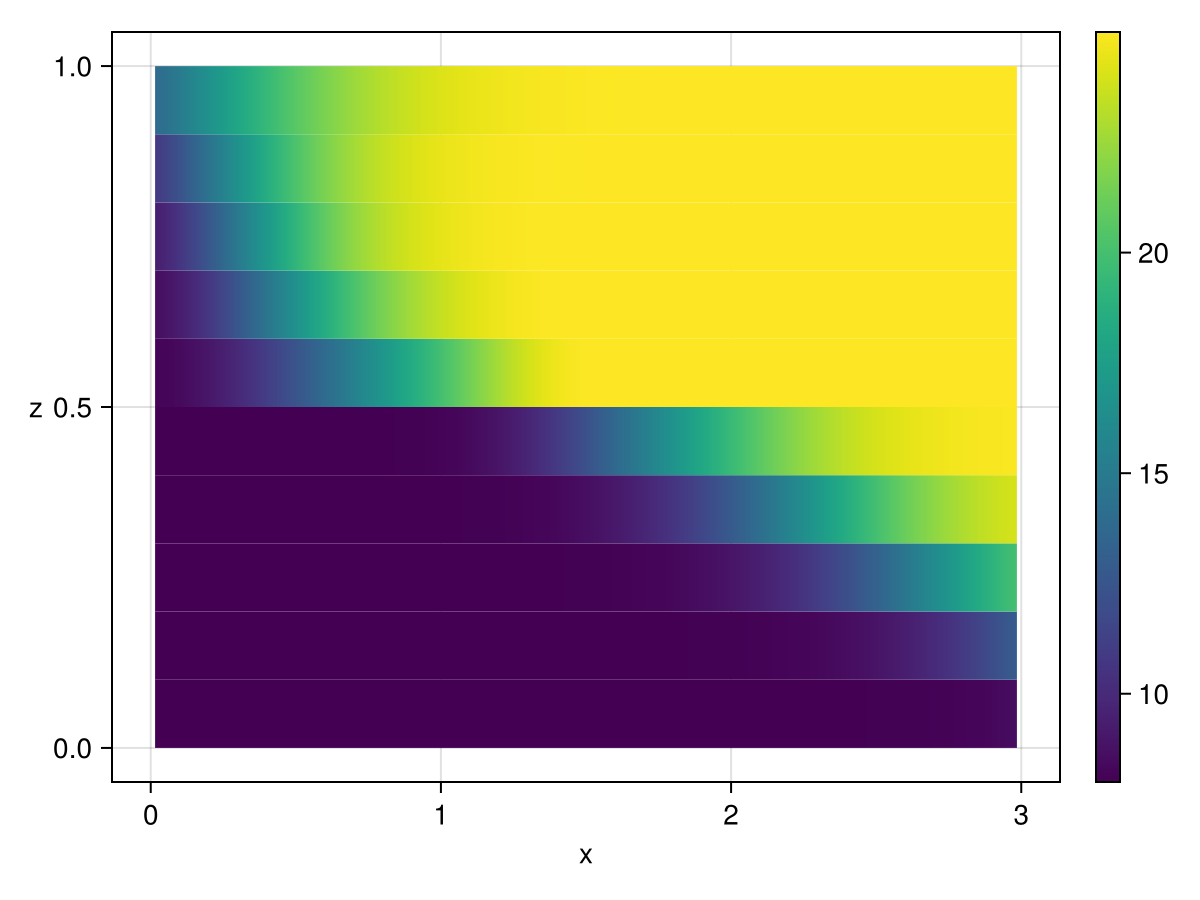}
        \caption{\revised{splitting, 100 cells, 10 layers}}
    \end{subfigure}
    \begin{subfigure}[t]{0.32\textwidth}
        \includegraphics[width=\textwidth]{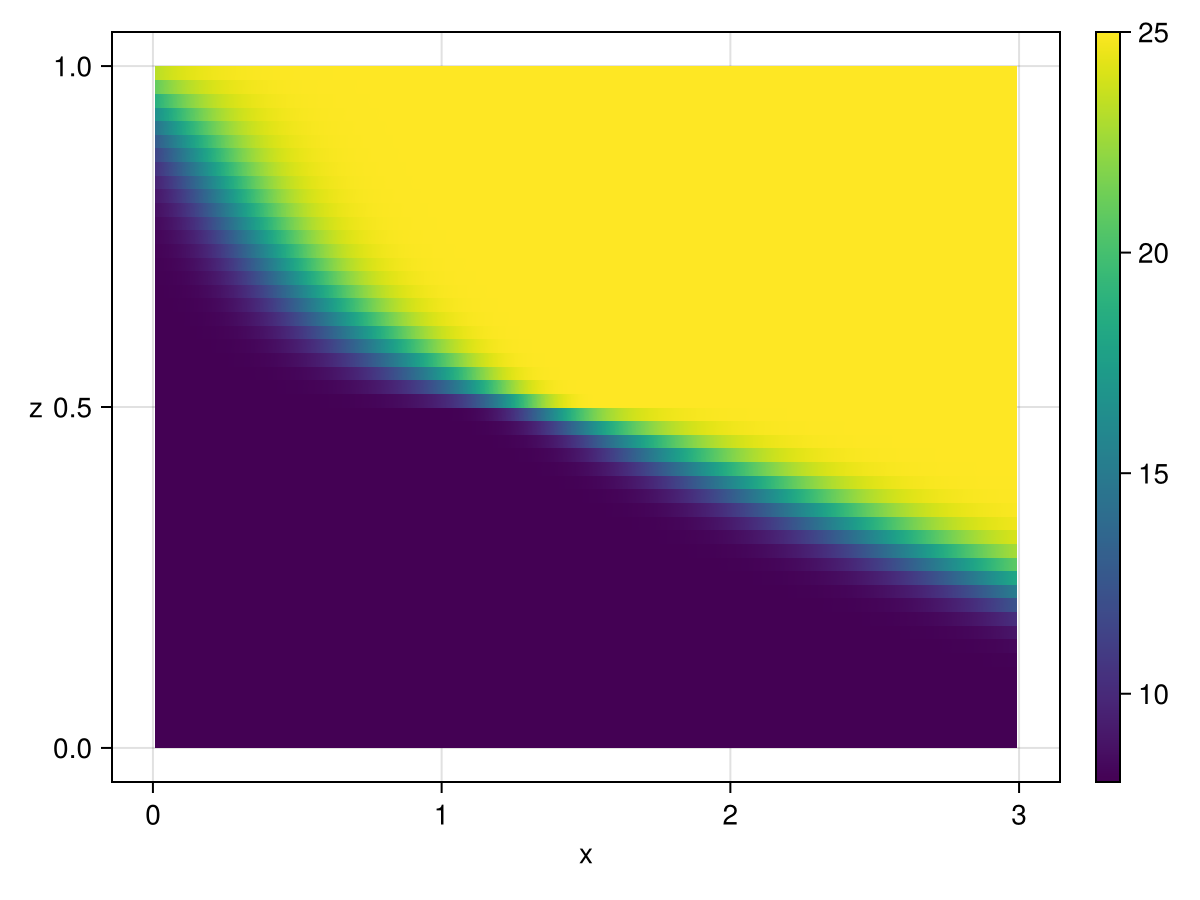}
        \caption{splitting, 200 cells, 50 layers}
    \end{subfigure}
    % \hfill
    \begin{subfigure}[t]{0.32\textwidth}
        \includegraphics[width=\textwidth]{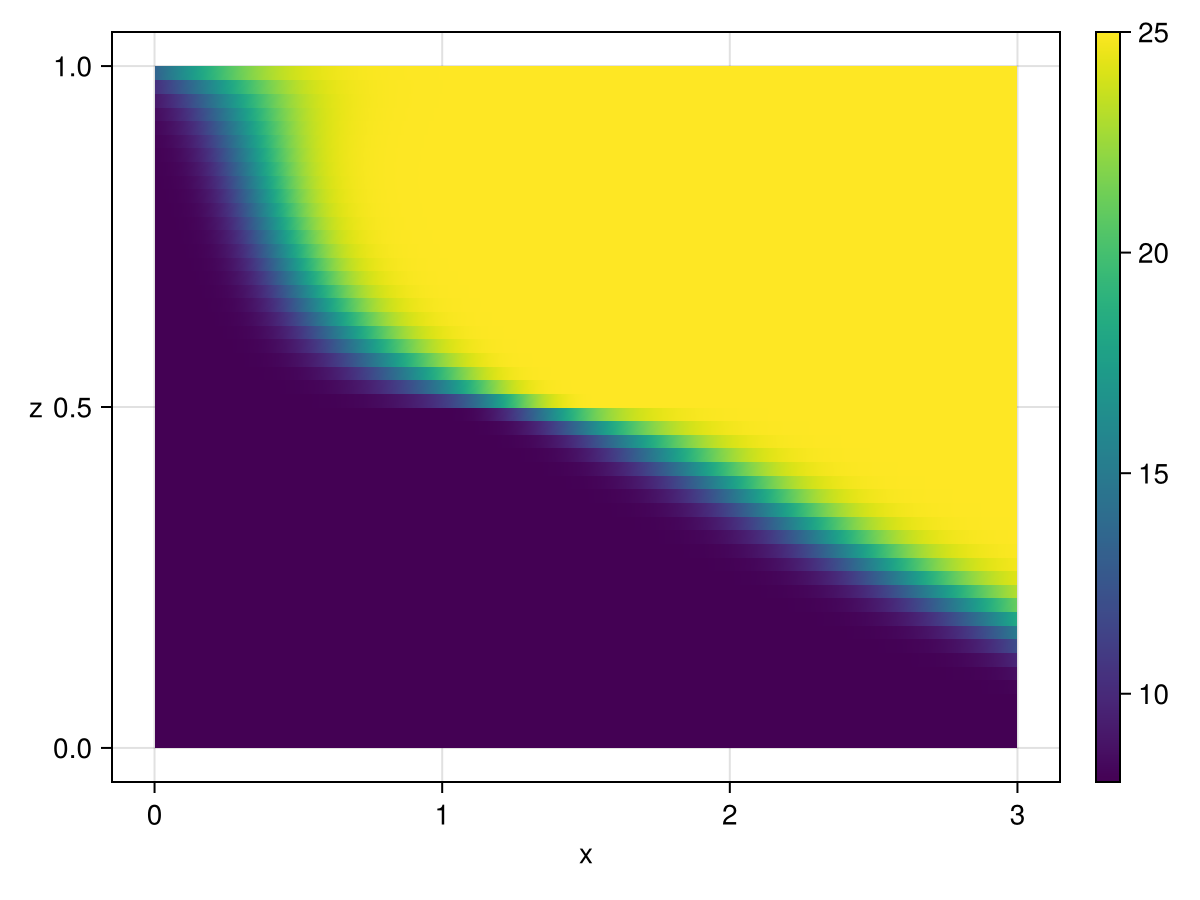}
        \caption{splitting, 800 cells, 50 layers}
    \end{subfigure}
    \hfill
    \caption{Tracer distribution of the wind driven cavity (testcase~\ref{sec:wind_driven_cavity}) at time $t = 600 \, \mathrm{s}$ with horizontal viscosity $\nu_\mathrm{hor} = \nu = 0.003 \, \,\mathrm{m}^2\mathrm{s}^{-1}$.}
    \label{fig:wind_driven_cavity_tracer_hor_visc=0.003}
\end{figure}

\subsection{Well-balancing for the lake at rest} \label{sec:wb_lar}

The barotropic-baroclinic splitting is well-balanced for the lake at rest if the shallow water solver~\eqref{eq:FV-SW} used in the barotropic step is well-balanced. Since we used the hydrostatic reconstruction of Audusse et al. \cite{ABBKP04,CN17}, which is well-balanced for the lake at rest, our barotropic-baroclinic splitting is also well-balanced. We verify this property with the following numerical test case taken from \cite{CN22}: On the computational domain $[0,4] \times [0,2]$ we consider a volcano shaped bottom topography with a lake in the crater and surrounding water. The initial data is given by
\begin{align*}
    r(x,y) &= 2(x-2)^2 + 4(y-1)^2, \\
    z_b(x,y) &= \begin{cases}
    1-0.8 \exp(-r(x,y)), & \text{ if } r(x,y) \leq \ln \frac85 \\
    0.8 \exp(-r(x,y)), & \text{ otherwise}
    \end{cases} \\
    h(x,y) &= \begin{cases}
    \max(0.45 - z_b(x,y),0), & r(x,y) \leq  \text{ if } \ln \frac85 \\
    \max(0.3 - z_b(x,y),0), & \text{ otherwise }
    \end{cases}\\
    u &\equiv 0 , \quad v \equiv 0
\end{align*}
and is plotted in Figure~\ref{fig:wb_lake-at-rest_initial}. We perform the simulation on a grid of $200 \times 100$ cells with 10 layers until final time $t=2$. The results are shown in Figure~\ref{fig:wb_lake-at-rest_final}. Since the deviation of the numerical solution from the lake at rest is in order of the machine accuracy the scheme is well-balanced.

\begin{figure}[htbp]
    \centering
    \begin{subfigure}[t]{0.32\textwidth}
        \includegraphics[width=\textwidth]{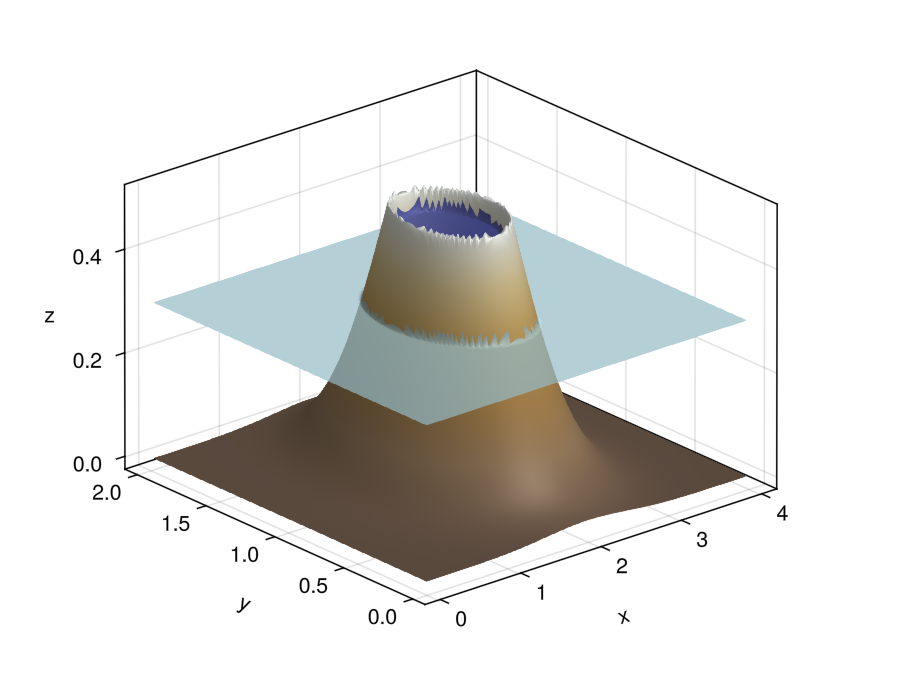}
        \caption{Bottom $z_b$ and water surface $\eta$.}
    \end{subfigure}
    \qquad
    \begin{subfigure}[t]{0.32\textwidth}
        \includegraphics[width=\textwidth]{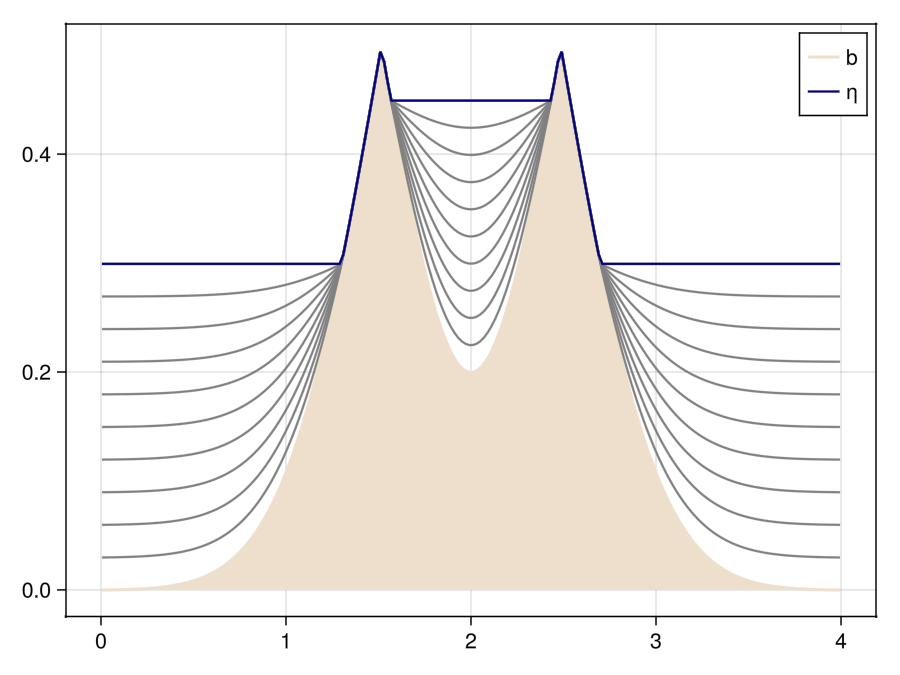}
        \caption{Cut at position $y=1$.}
    \end{subfigure}
            
    \caption{Initial data of lake at rest testcase~\ref{sec:wb_lar}.}
    \label{fig:wb_lake-at-rest_initial}
\end{figure}
    
\begin{figure}[htbp]
    \centering
    \begin{subfigure}[t]{0.32\textwidth}
        \includegraphics[width=\textwidth]{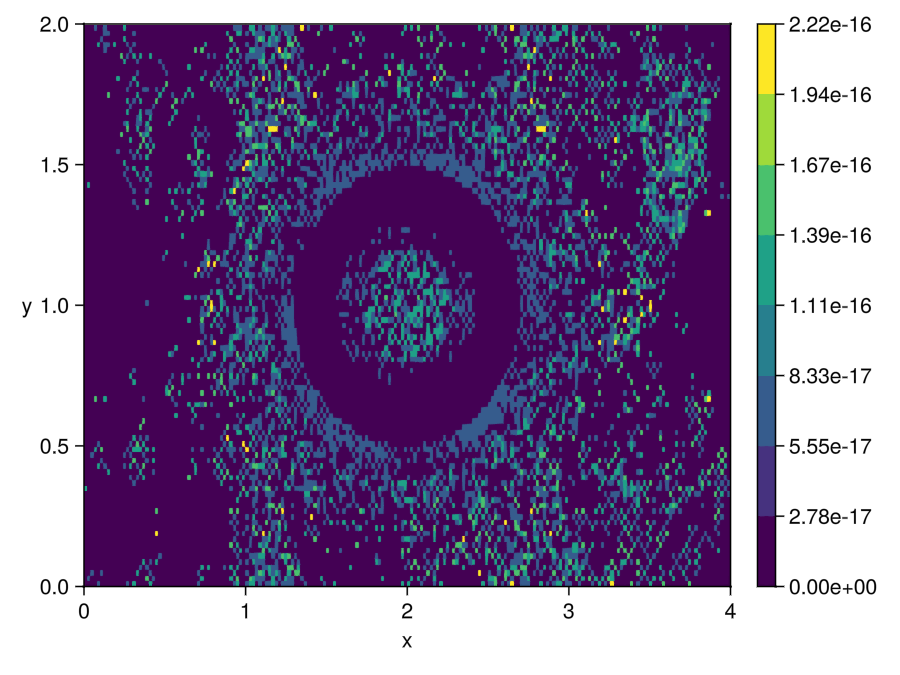}
        \caption{Error of water surface $\eta$.}
    \end{subfigure}
    \hfill
    \begin{subfigure}[t]{0.32\textwidth}
        \includegraphics[width=\textwidth]{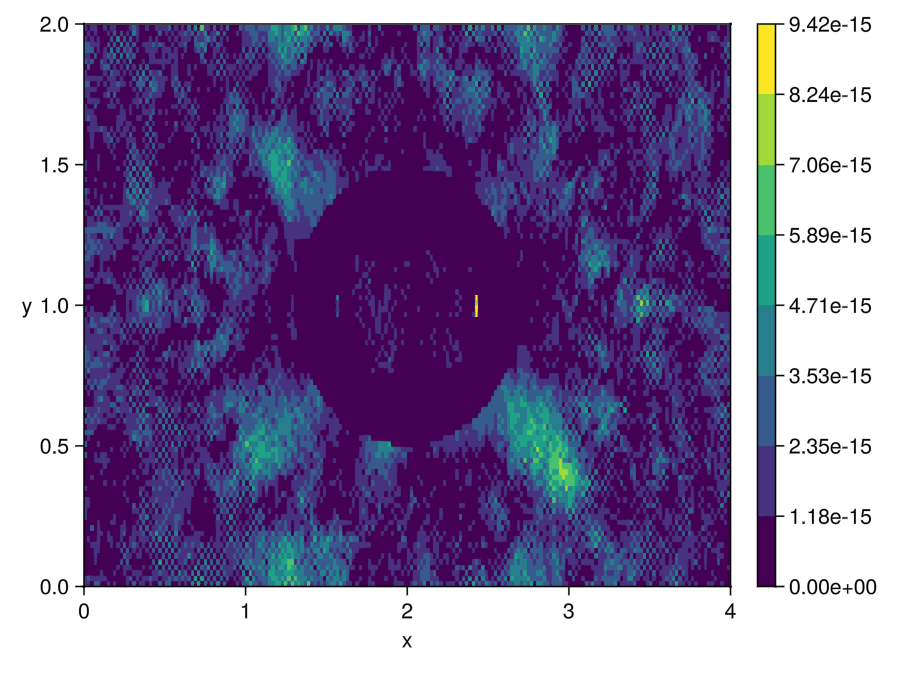}
        \caption{$\sum_{\alpha=1}^N |u_\alpha|$}
    \end{subfigure}
    \quad
    \begin{subfigure}[t]{0.32\textwidth}
        \includegraphics[width=\textwidth]{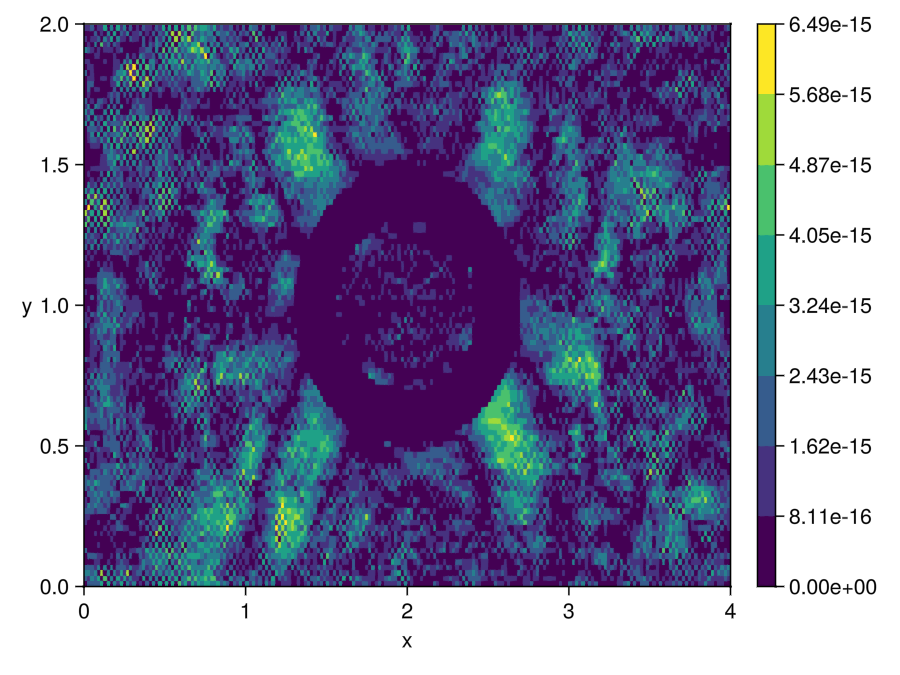}
        \caption{$\sum_{\alpha=1}^N |v_\alpha|$}
    \end{subfigure}
    \caption{Deviation from the lake at rest steady state in testcase~\ref{sec:wb_lar} at time $t=2$.}
    \label{fig:wb_lake-at-rest_final}
\end{figure}

\subsection{Western boundary current and "linear well-balancing" for the geo\-strophic equilibrium} \label{S:wbc}

In this section, we \revised{extend the barotropic-baroclinic splitting developed in Section~\ref{S:scheme} and take the earth's rotation into account to perform large scale geophysical simulations. The layer-averaged Euler equations with free surface and Coriolis force read
\begin{equation*}\label{eq:ML-Coriolis}
    \begin{cases}
        % conservation of mass
        \ddt{h} + \sumaN \dx (\ha \ua) + \sumaN \dy (\ha v_\alpha)  = 0,  \\  
         % hu, middle layers
        \dt (\ha \ua) + \dx \left(h \ua^2 \right) + g \ha \ddx{h} + \dy (\ha \ua \va) = -g \ha \ddx{\zb} + \uap\Gap - \uam\Gam + f\ha \va,   \\
        % hv, middle layers
        \dt (\ha \va) + \dx (h \ua \va) + \dx \left(\ha \va^2 \right) + g \ha \ddy{h} = -g \ha \ddy{\zb} + \vap\Gap - \vam\Gam - f \ha \ua, 
    \end{cases}
\end{equation*}
where $f$ is the Coriolis parameter and the mass exchanges $\Gap$ and the interface velocities $\uap$ are the defined as before in~\eqref{eq:MSW1D-closed} and~\eqref{eq:interu}, respectively. For simplicity, we ignored the viscosity, friction and wind forcing terms in this formulation. Their numerical treatment within the baroclinic part has already been described in Section~\ref{sec:wind_driven_cavity}. Now, we perform the barotropic-baroclinic splitting for this system. Analogously to~\eqref{eq:BT}, the barotropic step now reads
% \begin{equation*}
%     \begin{cases}
%         \ddt{h} + \dx (h \baru) + \dy (h \barv)  = 0,  \\ 
%         \dt (h \ua) + \dx \left(h \ua \baru \right) + g h \ddx{h} + \dy (h \ua \barv) = -gh \ddx{\zb} + fh \va,   \\ 
%         \dt (h \va) + \dx (h \baru \va) + \dx \left(h \va \barv \right) + g h \ddy{h} = -gh \ddy{\zb} - fh \ua, 
%     \end{cases}
% \end{equation*}
% which is equivalent to the system
\begin{equation} \label{eq:BT_Coriolis}
    \begin{cases}
        \ddt{h} + \dx (h \baru) + \dy (h \barv)  = 0,  \\ 
        \dt (h \baru) + \dx \left(h \baru^2 + \frac{g}{2} h^2 \right) + \dy (h \baru \barv) = -gh \ddx{\zb} + fh\barv,   \\ 
        \dt (h \barv) + \dx (h \baru \barv) + \dx \left(h \barv^2  + \frac{g}{2} h^2 \right) = -gh \ddy{\zb} - fh\baru, \\
        \dt (h\sa) + \dx (h\baru\sa) + \dy (h\barv\sa) = fh\ta, \\
        \dt (h\ta) + \dx (h\baru\ta) + \dy (h\barv\ta) = -fh\sa. \\
    \end{cases}
\end{equation}
Thus, we obtain the rotating shallow water equations and two transport equations with a Coriolis force source term. The baroclinic step stays the same as before, see~\eqref{eq:BC}. Accordingly, the Coriolis force is treated entirely in the barotropic part and does not appear in the baroclinic part.
}

\revised{In the following, we investigate two different numerical treatments of the rotating shallow water equations, i.e., the first three equations of~\eqref{eq:BT_Coriolis}. The first strategy is a semi-implicit treatment, the second strategy is the ``linearly well-balancing'' strategy~\cite{GCGP24}. The transport equations with Coriolis force source terms, i.e., the last two equations of~\eqref{eq:BT_Coriolis} are always treated semi-implicitly in the subcycling step: We first update $(h \sa)_{ij}^{n+\frac12}$ and $(h \ta)_{ij}^{n+\frac12}$ with~\eqref{eq:subcycling} to obtain the intermediate values $\hat{\sigma}_{\alpha,ij}^{n+1}$ and $\hat{\theta}_{\alpha,ij}^{n+1}$. Then, we apply the Coriolis force source term semi-implicitly by
\begin{equation} \label{eq:semi-implicit_Coriolis}
    \begin{cases}
        \sigma_{\alpha,ij}^{n+1} = \hat{\sigma}_{\alpha,ij}^{n+1} + \Dtn f \hat{\theta}_{\alpha,ij}^{n+1} \\
        \theta_{\alpha,ij}^{n+1} = \hat{\theta}_{\alpha,ij}^{n+1} - \Dtn f \sigma_{\alpha,ij}^{n+1}.
    \end{cases}
\end{equation}}

\revised{Next, we present the two strategies for the treatment of the rotating shallow water equations. The first strategy is a semi-implicit treatment similar to~\eqref{eq:semi-implicit_Coriolis}: In each step of the barotropic loop, we first apply the shallow water solver~\eqref{eq:FV-SW} to obtain the intermediate values $\hat{\baru}_{ij}^{n+\frac12,k+1}$ and $\hat{\barv}_{ij}^{n+\frac12,k+1}$. Then, we apply~\eqref{eq:semi-implicit_Coriolis} with $\baru$, $\barv$ and $\dtk$ playing the role of $\sa$, $\ta$ and $\Dtn$, respectively. }

\revised{The second strategy investigated in the barotropic step is the "linearly well-balancing" strategy of \cite{GCGP24} that aims to better preserve the geostrophic equilibrium than the non-well-balanced strategy discussed above. The solver presented in the following replaces the shallow water solver~\eqref{eq:FV-SW} in the barotropic loop.}
The goal of this strategy is to preserve the geostrophic equilibrium $g \nabla (h+\revised{\zb}) = f (v,-u)^\top$ with a linear velocity field. To obtain this we first perform a linear bottom reconstruction $\revised{z_{b,ij}^R}(x,y) = \revised{z_{b,ij}} + \mathcal{M}_{ij}(x-x_i) + \mathcal{N}_{ij}(y-y_j)$ to have access to $\nabla \revised{z_{b,ij}^R}(x,y)$. Using the cell values $u_{ij}$ and $v_{ij}$ we construct a piecewise linear divergence free velocity field
\begin{align*}
    u_{ij}^R(x,y) &= u_{ij} + \mathcal{A}_{ij}(x-x_i) + \mathcal{B}_{ij}(y-y_j),\\
    v_{ij}^R(x,y) &= v_{ij} + \mathcal{C}_{ij}(x-x_i) - \mathcal{A}_{ij}(y-y_j)
\end{align*}
solving a least squares minimization problem using the neighbouring cells. Note that with this choice of the coefficients the velocity field $(u^R,v^R)$ is already divergence free. Then we construct 
\begin{equation*}
    h_{ij}^R(x,y) = h_{ij} + \alpha_{ij}(x-x_i) + \beta_{ij}(y-y_j) + \gamma_{ij}(x-x_i)^2 + \delta_{ij}(y-y_j)^2 + \lambda_{ij}(x-x_i)(y-y_j)
\end{equation*}
such that 
\begin{equation*}
    \nabla h_{ij}^R = \frac{f}{g} (v_{ij}^R,-u_{ij}^R)^\top - \nabla \revised{z_{b,ij}^R}.
\end{equation*}
This condition is used to determine the coefficients
\begin{align*}
    \alpha_{ij} = \frac{f}{g} v_{ij} - \mathcal{M}_{ij}, \quad 
    \beta_{ij} = -\frac{f}{g} u_{ij} - \mathcal{N}_{ij}, \quad
    \gamma_{ij} = \frac{f}{2g} \mathcal{C}_{ij}, \quad 
    \delta_{ij} = -\frac{f}{2g} \mathcal{B}_{ij}, \quad 
    \lambda_{ij} = -\frac{f}{g} \mathcal{A}_{ij}.
\end{align*}
After having reconstructed $h^R_{ij}$, $u^R_{ij}$ and $v^R_{ij}$, we compute the numerical fluxes using the evaluations of the reconstructions at the midpoint of the cell interfaces. Finally we compute the source term via 
\begin{align*}
    S_{ij} 
    &= \int_{C_{ij}} fh_{ij}^R(v_{ij}^R,-u_{ij}^R)^\top
    = \int_{C_{ij}} g h_{ij}^R (\nabla h_{ij}^R + \nabla \revised{z_{b,ij}^R}) \\
    &= \int_{C_{ij}} \frac12 g(\nabla h_{ij}^R)^2 + \int_{C_{ij}} gh_{ij}^R \nabla \revised{z_{b,ij}^R} 
    = \int_{\Gamma_{ij}} \frac12 g (h_{ij}^R)^2 \cdot \vec{n} + \int_{C_{ij}} gh_{ij}^R \nabla \revised{z_{b,ij}^R}.
\end{align*}

% \subsubsection{Western boundary current}

To test \revised{the semi-implicit and the "linearly well-balancing"} strategy, we perform a large scale simulation and perform the the testcase Henry Stommel investigated in his paper \cite{Stommel48}. Stommel considered a rectangular ocean $\Omega = [0,\lambda] \times [0,b]$ with $\lambda = 10^7 \,\mathrm{m}$ and $b = 2\pi\cdot 10^6 \, \mathrm{m}$. The ocean is considered as a homogeneous layer of constant depth of $200\,\text{m}$. At initial time the ocean is at rest. The imposed wind stress is given by $\tau^x = -F\cos \left( \frac{\pi y}{b} \right)$ and $\tau^y = 0$ with $F=0.1 \, \mathrm{N}\mathrm{m}^{-2}$. \revised{Note that in comparison to testcase~\ref{sec:wind_driven_cavity}, here the wind stress is two-dimensional and imposed directly to the surface layer and is not computed using the empirical law~\eqref{eq:wind-law}.} The bottom friction coefficient is given by $\kappa = 0.02$. As boundary conditions we use walls. Stommel has shown that using the beta-plane approximation for the Coriolis parameter $f = f_0 + \beta_0 y$ results in a stationary western boundary current that is in geostrophic equilibrium. In contrast to Stommel we choose a midlatitude setting for our simulation, so we set $f_0 = 2.5 \cdot 10^{-5}$ and $\beta_0 = 10^{-11}$.  The exact linearized solution of Stommel is plotted in Figure~\ref{fig:Stommel}. The maximum speed of the western boundary current is $2\,\mathrm{ms}^{-1}$. The width of the western boundary current is in the order of $100\,\mathrm{km}$.

\begin{figure}[htbp]
    \centering
    \begin{subfigure}[t]{0.275\textwidth}
        \includegraphics[width=\textwidth]{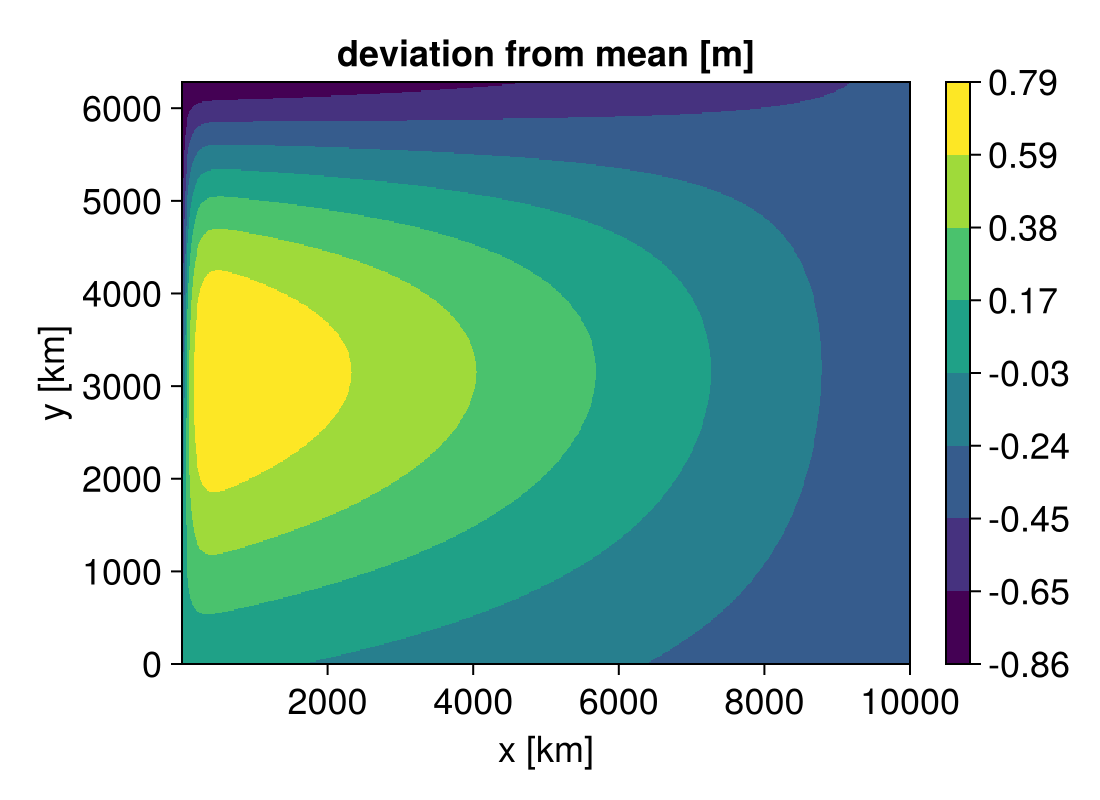}
        \caption{Deviation from mean water height.}
    \end{subfigure}
    \hfill
    \begin{subfigure}[t]{0.275\textwidth}
        \includegraphics[width=\textwidth]{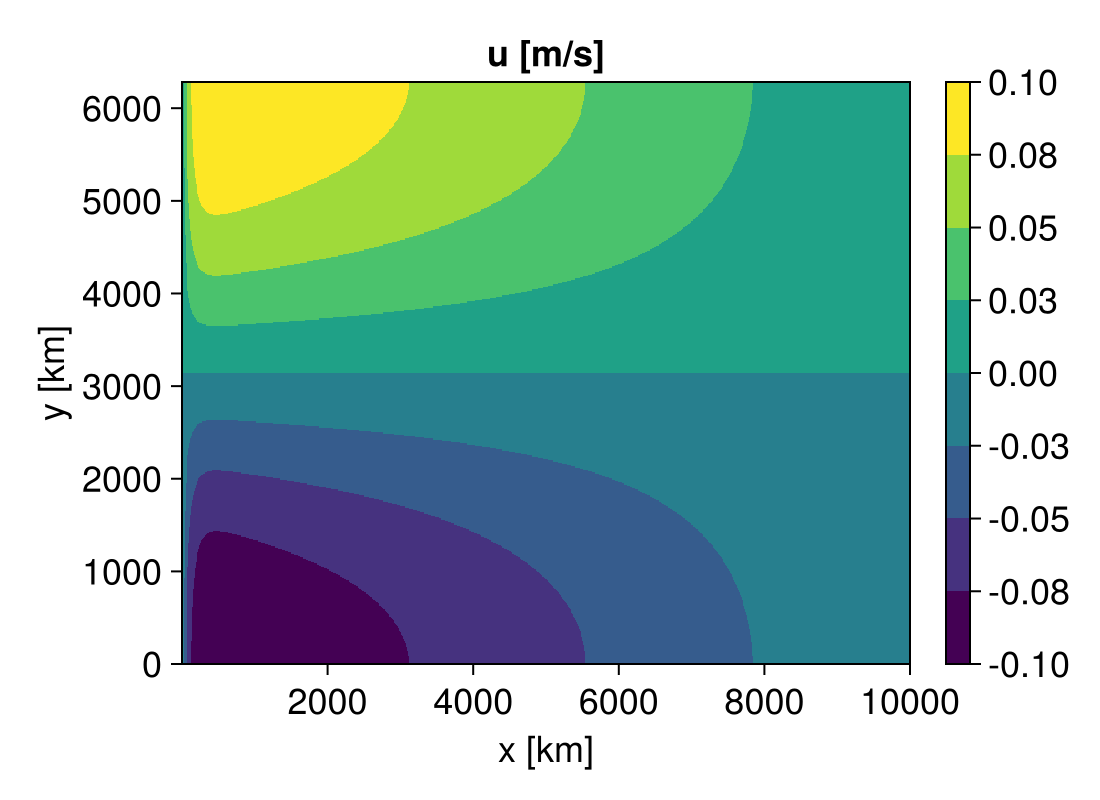}
        \caption{Depth averaged zonal velocity $\bar{u}$.}
    \end{subfigure}
    \hfill
    \begin{subfigure}[t]{0.275\textwidth}
        \includegraphics[width=\textwidth]{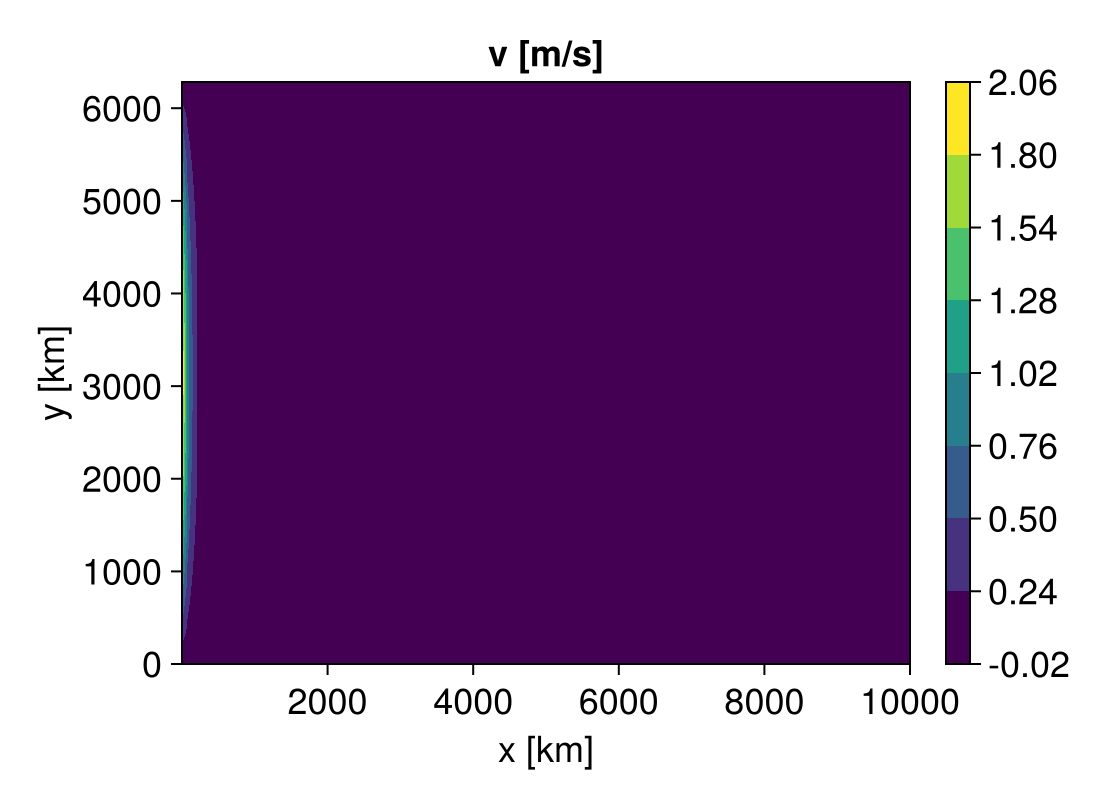}
        \caption{Depth averaged meridional velocity $\bar{v}$.}
    \end{subfigure}
    \hfill
    \begin{subfigure}[t]{0.15\textwidth}
        \includegraphics[width=\textwidth]{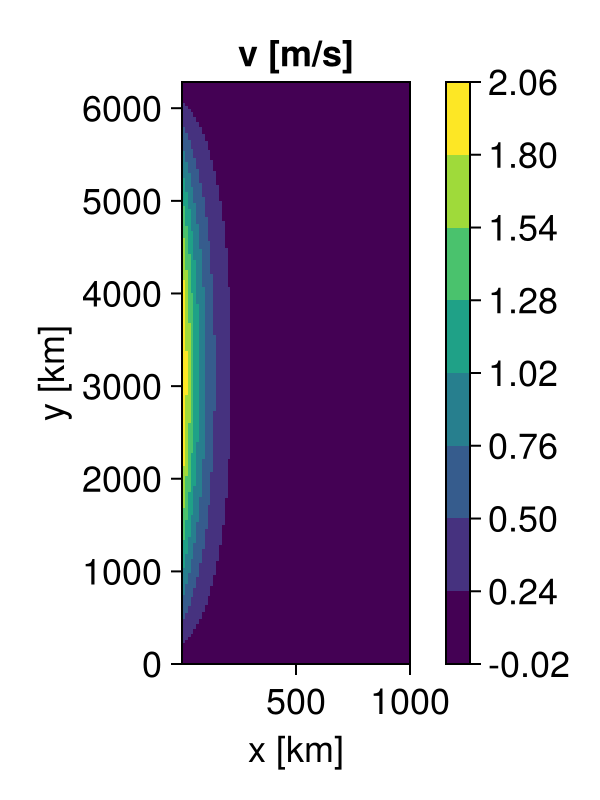} 
        \caption{\revised{Zoom on $\bar{v}$.}}
    \end{subfigure}
    \caption{Exact linearized solution for a western boundary current of Stommel.}
    \label{fig:Stommel}
\end{figure}

For our simulation we use the fastest version of the barotropic-baroclinic splitting (with explicit treatment of the exchange terms and the subcycling strategy). We run the scheme on a grid of \revised{$800 \times 480$} cells and $50$ layers until final time $t=10 \, \mathrm{years}$ and compare the numerical results obtained with and without the use of the well-balancing strategy detailed above.
Since the western boundary current is only created when the beta-plane approximation $f = f_0 + \beta_0 y$ is done, we use a variable Coriolis coefficient $f=f_{j}$ in each cell. The results of the simulations are plotted in Figure~\ref{fig:wbc_no-wb_vs_wb} and Figure~\ref{fig:wbc_streamplot}. One can observe that the modification of the shallow water solver by using the linear well-balancing strategy has a big impact on the solution, even with a first-order scheme on a coarse mesh. 
\revised{Using the ``linearly well-balancing'' strategy presented above,} the western boundary current appears more narrow \revised{(approximately $250\,\mathrm{km}$ compared to $600 \,\mathrm{km}$)} and with higher speed (approximately $\revised{2.37}\,\mathrm{ms}^{-1}$ compared to $\revised{0.74} \,\mathrm{ms}^{-1}$). \revised{In addition, we observe an elongation in the deviation from the mean water height and the formation of a blob in $\bar{u}$ in the north west corner of the ocean.} 

It is quite difficult to compare the theoretical results of Stommel's illustrated on Figure~\ref{fig:Stommel} with our numerical simulations on Figure~\ref{fig:wbc_no-wb_vs_wb} for at least two reasons:
\begin{itemize}
    \item Stommel model is a one layer linearized model in primitive variables, while we solve a multilayer nonlinear model with conservative variables. The vertical diffusion coefficient has been chosen in accordance with the choice in OGCMs.
    \item The boundary layer is strongly influenced by the dissipating mechanisms taken in consideration in the equation and the presence of nonlinear effects~\cite{B63},~\cite{B86}. In our simulation the boundary current width and the grid size approximately are comparable. As a result the numerical dissipation is quite strong near the western boundary, resulting in a non-controlled dissipating mechanism not explicitly appearing in the equation but of the same order of magnitude, if not higher, than the frictionnal term. 
\end{itemize}
Our interpretation of the numerical results is the following. Since we are using a Rusanov scheme we can argue that the numerical diffusion writes as a (space and time varying) laplacian term and that the solution of the model is closer to the solution of (nonlinear) Munk model. Then the shape of the western boundary layer depends on the Reynolds number $\operatorname{Re}=\frac{L_s L}{A}$ where $L_s$ is the interior velocity flow and $A$ is the diffusion coefficient. 
With the original Rusanov scheme, the numerical diffusion is quite high and the Reynolds number is low. On the other hand the well-balancing strategy results in a much lower numerical diffusion and a higher Reynolds number. In those circumstances the elongation of the western boundary current in $h$ and the eddy in velocity is documented~\cite[Figure 3]{B63}.

\begin{figure}[htbp]
    \centering
    \begin{subfigure}[t]{0.275\textwidth}
        \includegraphics[width=\textwidth]{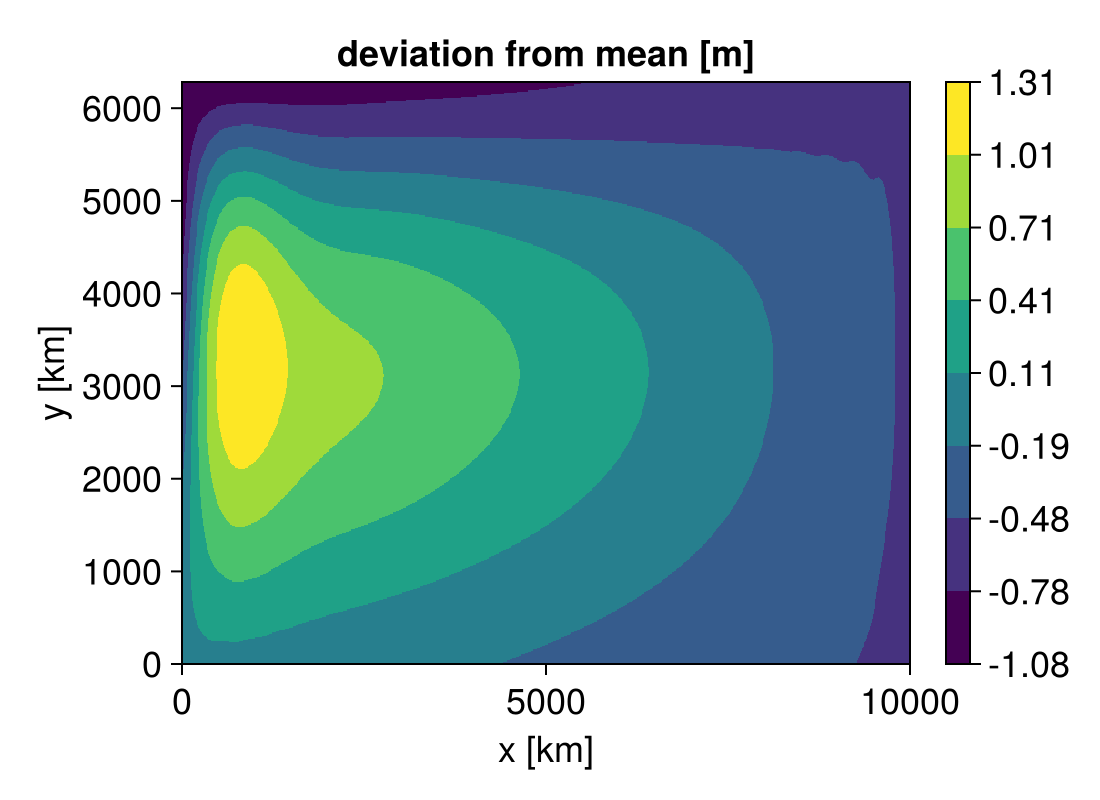}
        \caption{Deviation from mean height.}
    \end{subfigure}
    \hfill
    \begin{subfigure}[t]{0.275\textwidth}
        \includegraphics[width=\textwidth]{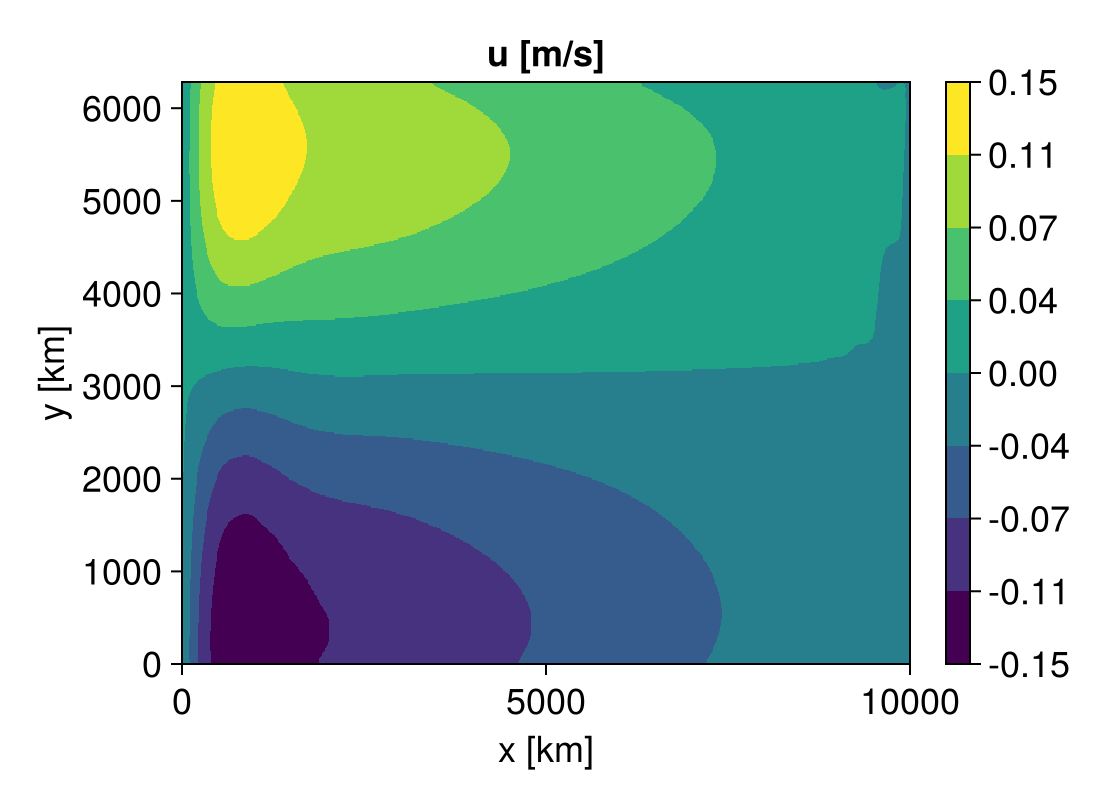}
        \caption{Depth averaged zonal velocity $\bar{u}$.}
    \end{subfigure}
    \hfill
    \begin{subfigure}[t]{0.275\textwidth}
        \includegraphics[width=\textwidth]{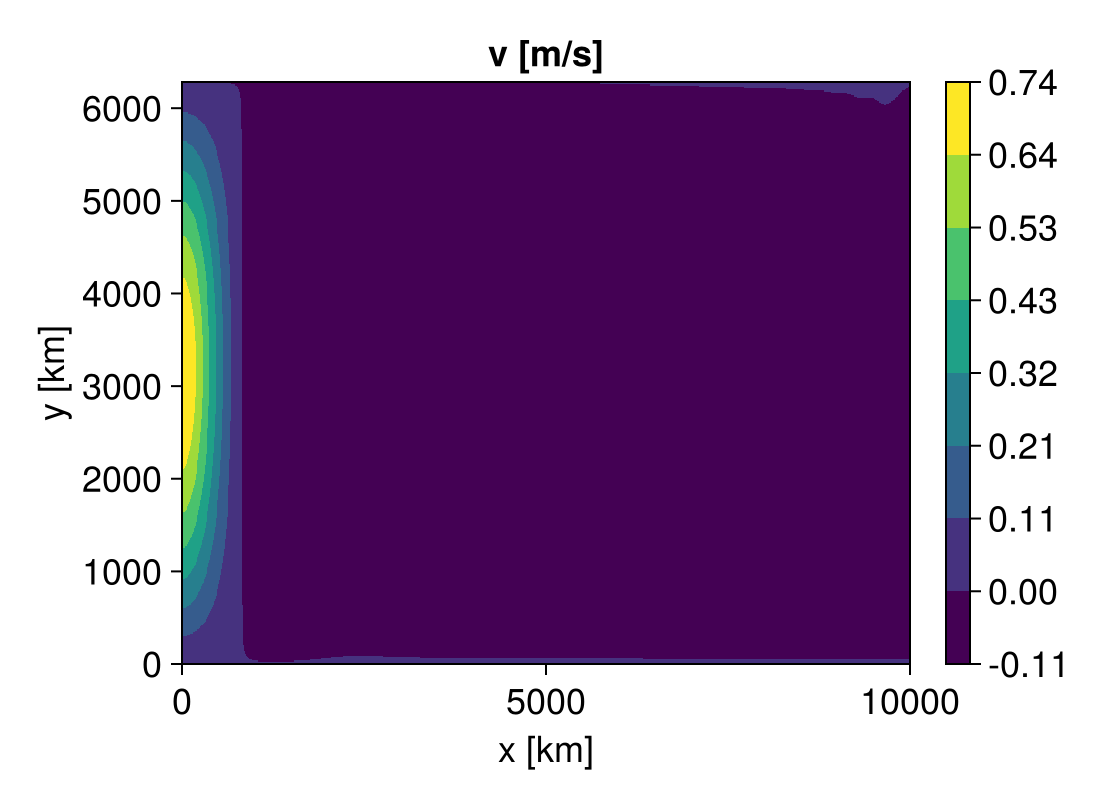} 
        \caption{Depth averaged meridional velocity $\bar{v}$.}
    \end{subfigure}
    \hfill
    \begin{subfigure}[t]{0.15\textwidth}
        \includegraphics[width=\textwidth]{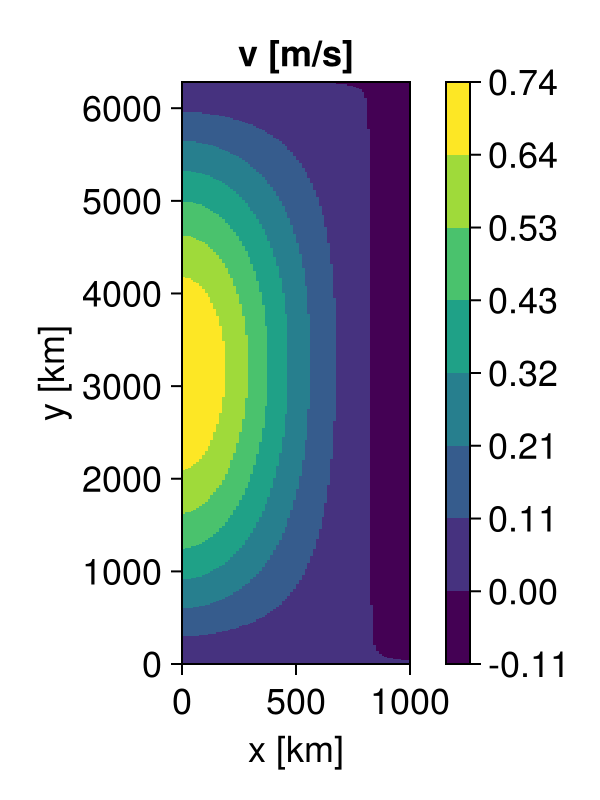} 
        \caption{\revised{Zoom on $\bar{v}$.}}
    \end{subfigure}

    \begin{subfigure}[t]{0.275\textwidth}
        \includegraphics[width=\textwidth]{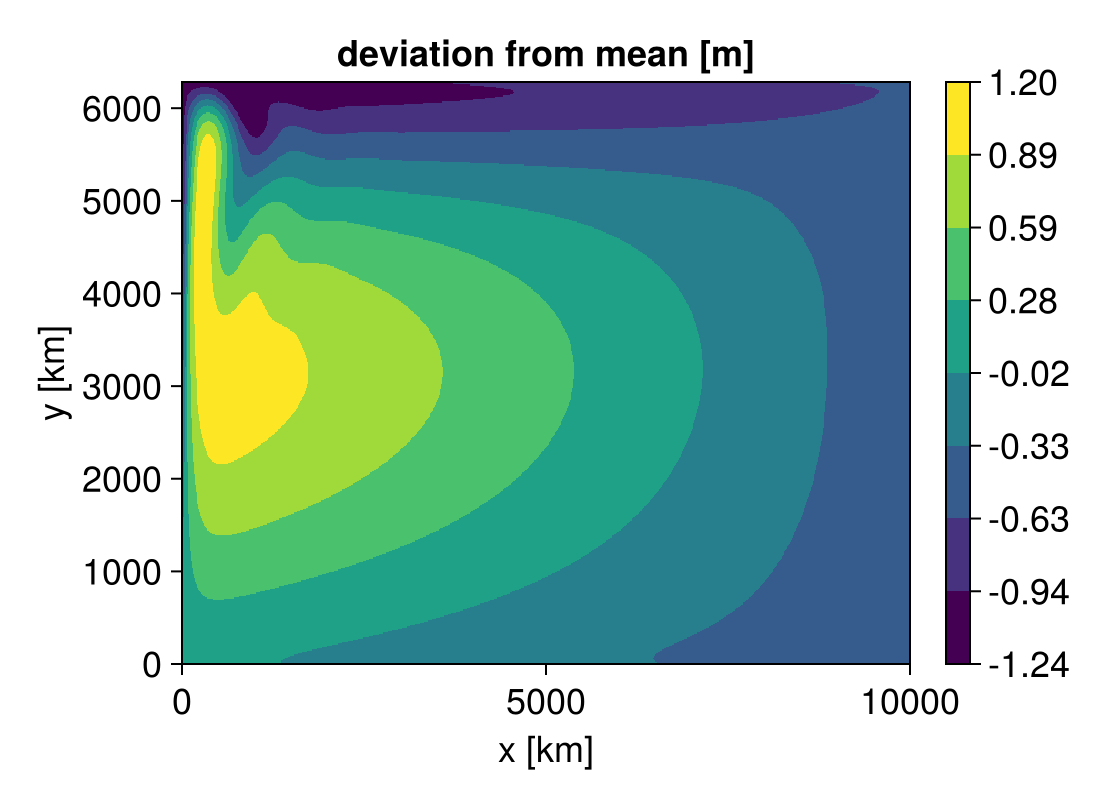}
        \caption{Deviation from mean height.}
    \end{subfigure}
    \hfill
    \begin{subfigure}[t]{0.275\textwidth}
        \includegraphics[width=\textwidth]{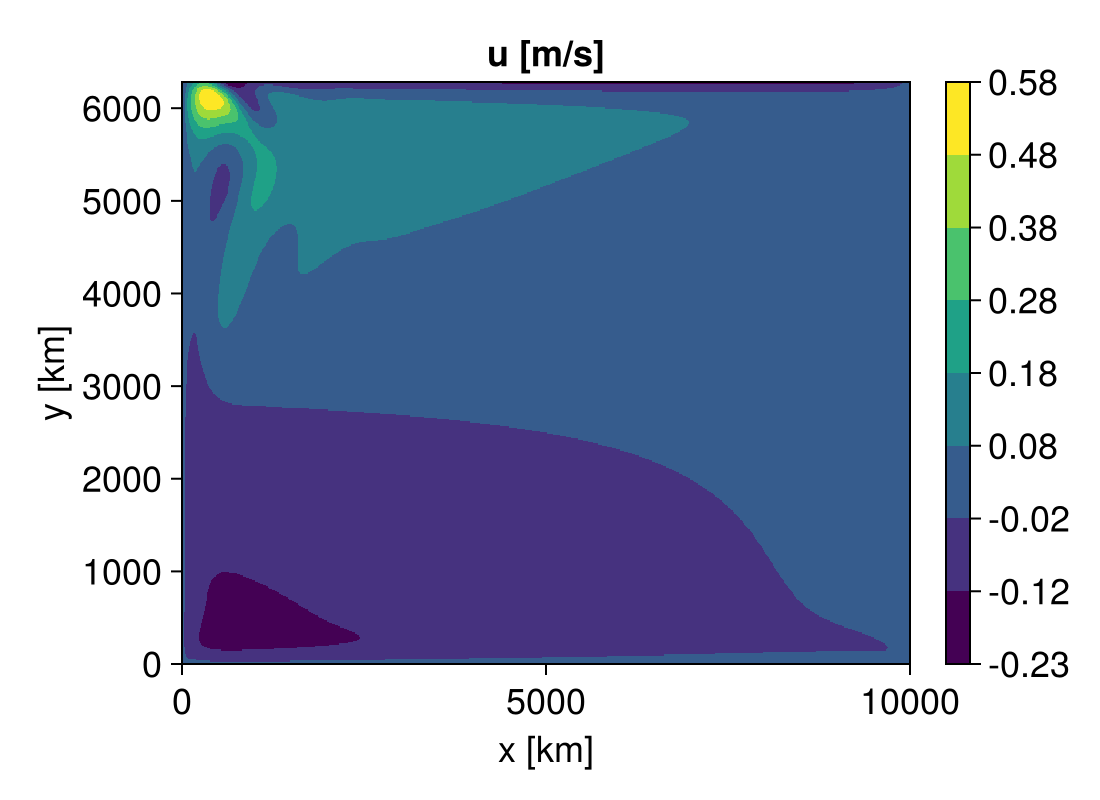}
        \caption{Depth averaged zonal velocity $\bar{u}$.}
    \end{subfigure}
    \hfill
    \begin{subfigure}[t]{0.275\textwidth}
        \includegraphics[width=\textwidth]{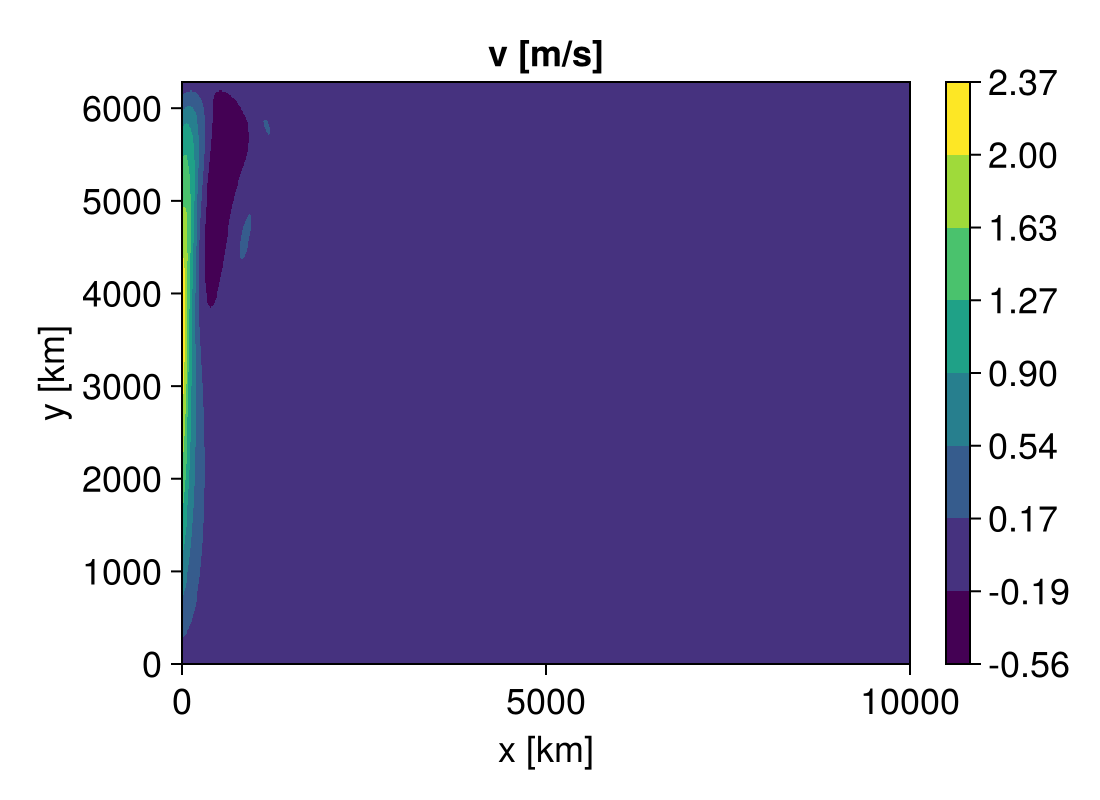}
        \caption{Depth averaged meridional velocity $\bar{v}$.}
    \end{subfigure}
    \hfill
    \begin{subfigure}[t]{0.15\textwidth}
        \includegraphics[width=\textwidth]{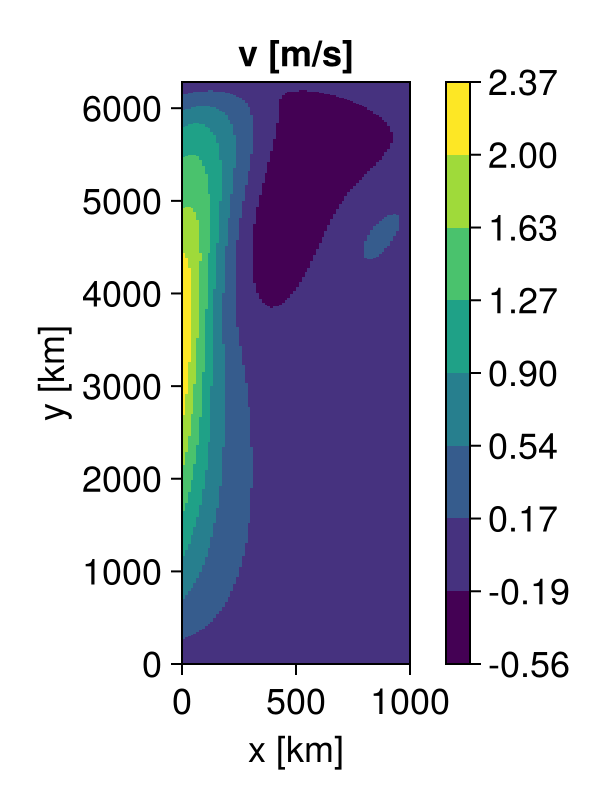} 
        \caption{\revised{Zoom on $\bar{v}$.}}
    \end{subfigure}
    
    \caption{Numerical solution on a grid with \revised{$800 \times 480$} cells ($\Delta x = 12.5 \, \mathrm{km} \ \revised{\approx \Delta y}$ without well-balancing strategy for the geostrophic equilibrium (top) and with well-balancing strategy (bottom)).}
    \label{fig:wbc_no-wb_vs_wb}
\end{figure}

\begin{figure}
    \centering
    \hfill
    \begin{subfigure}[t]{0.4\textwidth}
        \includegraphics[width=\textwidth]{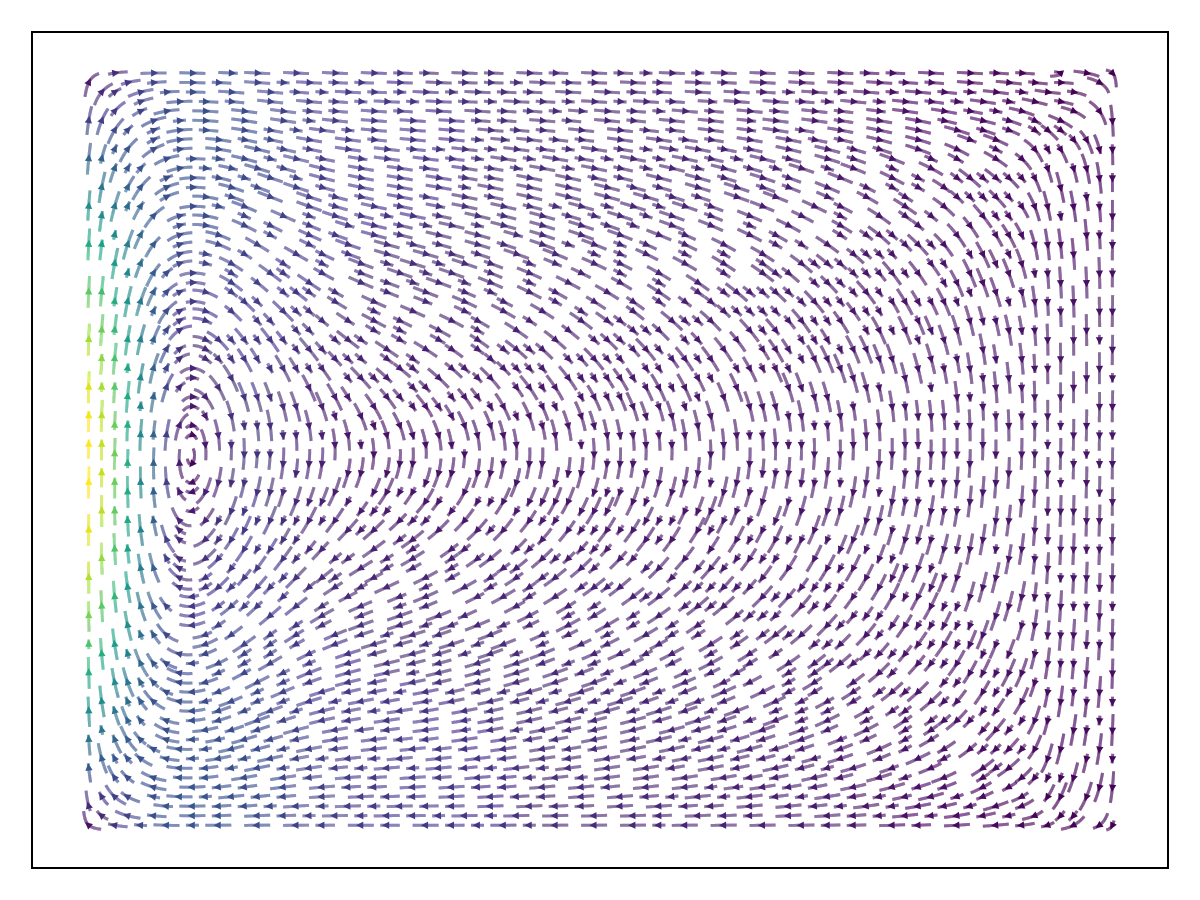}
        \caption{No well-balancing.}
    \end{subfigure}
    \hfill
    \begin{subfigure}[t]{0.4\textwidth}
        \includegraphics[width=\textwidth]{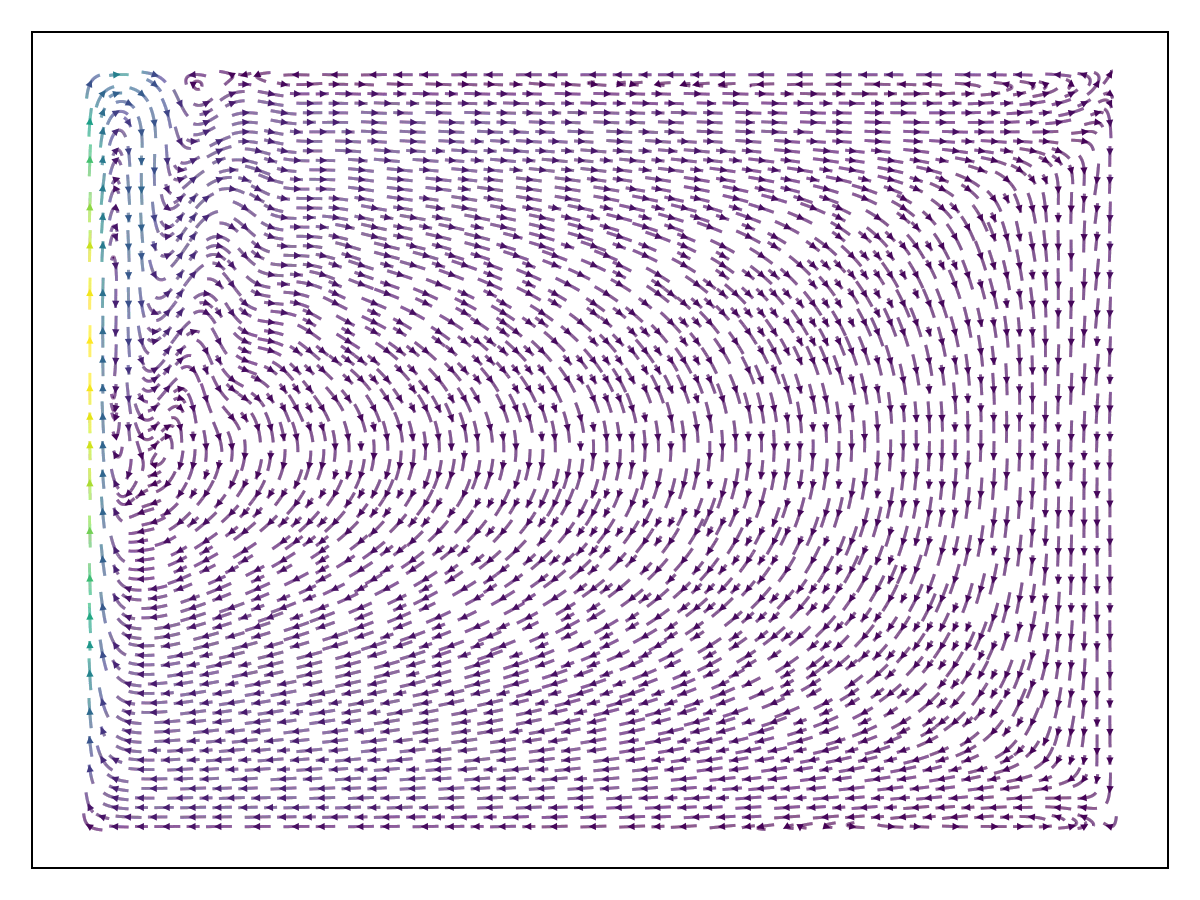}
        \caption{With well-balancing.}
    \end{subfigure}
    \hfill \phantom{.}
    \caption{Streamplot of the depth averaged velocities obtained without (a) and with (b) linear well-balancing strategy for the geostrophic equilibrium.}
    \label{fig:wbc_streamplot}
\end{figure}

\revised{In addition, we compare the results obtained with the splitting with the results obtained with the Freshkiss3D code~\cite{ABF19}. Freshkiss3D is a second-order multilayer solver running on unstructured meshes. Due to the high computational cost (the Freshkiss solver runs approximately 100 times longer than the barotropic-baroclinic splitting) we reduced the grid resolution to approximately 60000 cells, the number of layers to $N=20$ and also the final time of the simulation, since we observed in previous simulations that the equilibrium is already reached after 0.5 years of simulation time. The results are plotted in Figure~\ref{fig:wbc_freshkiss}. The speed of the western boundary current in the Freshkiss3D simulation lays inbetween the speed of the simulations performed with and without well-balancing strategy. For the zonal velocity $\baru$ we can slowly observe the formation of the blob in the north-west corner. The deviation from the mean water height shows a slight distortion, but the extent of this is barely noticeable. We observe on this test case the well-established conclusion that on coarse meshes, first order well-balanced schemes are competitive with second order non well-balanced schemes.}

\begin{figure}
    \centering
    \begin{subfigure}[t]{0.275\textwidth}
        \includegraphics[width=\textwidth]{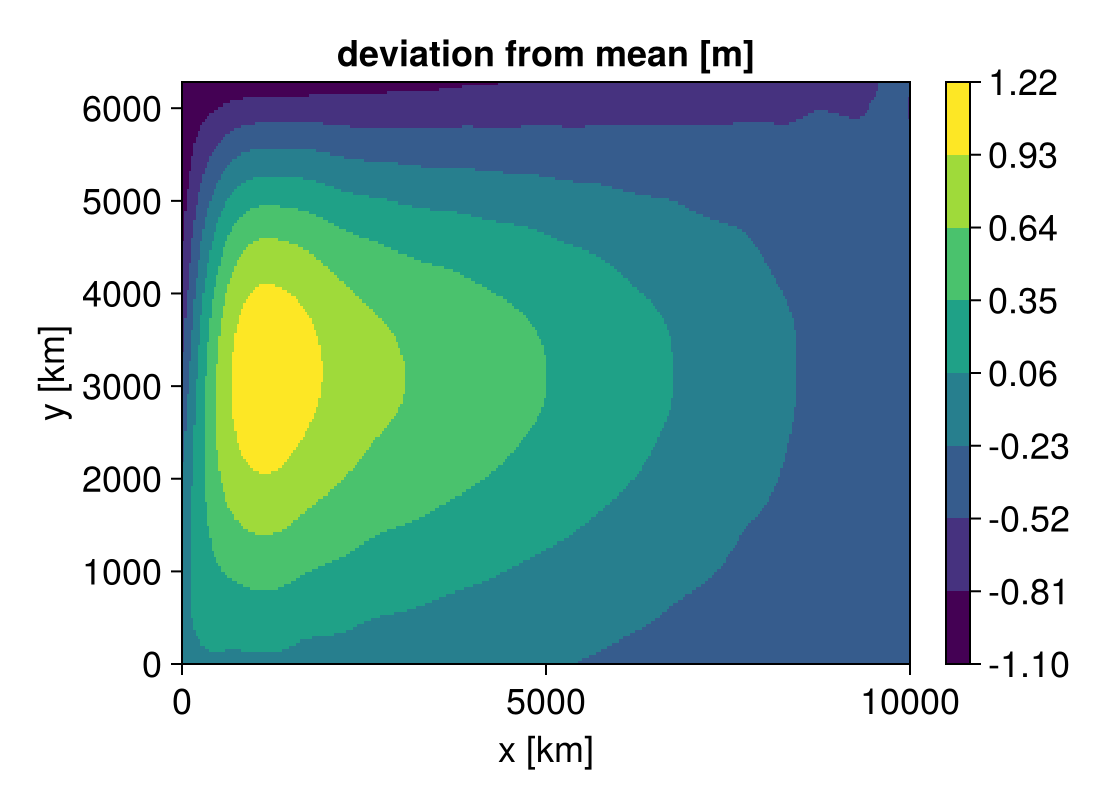}
        \caption{Deviation from mean height.}
    \end{subfigure}
    \hfill
    \begin{subfigure}[t]{0.275\textwidth}
        \includegraphics[width=\textwidth]{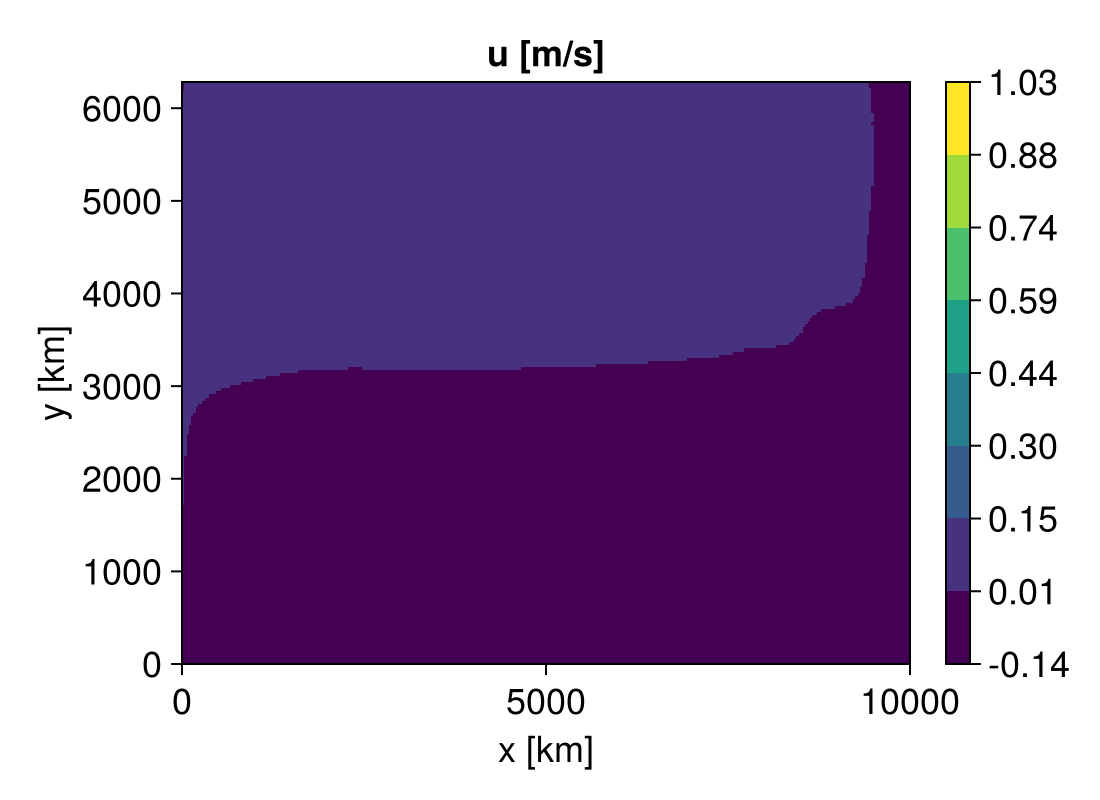}
        \caption{Depth-averaged zonal velocity $\bar{u}$.}
    \end{subfigure}
    \hfill
    \begin{subfigure}[t]{0.275\textwidth}
        \includegraphics[width=\textwidth]{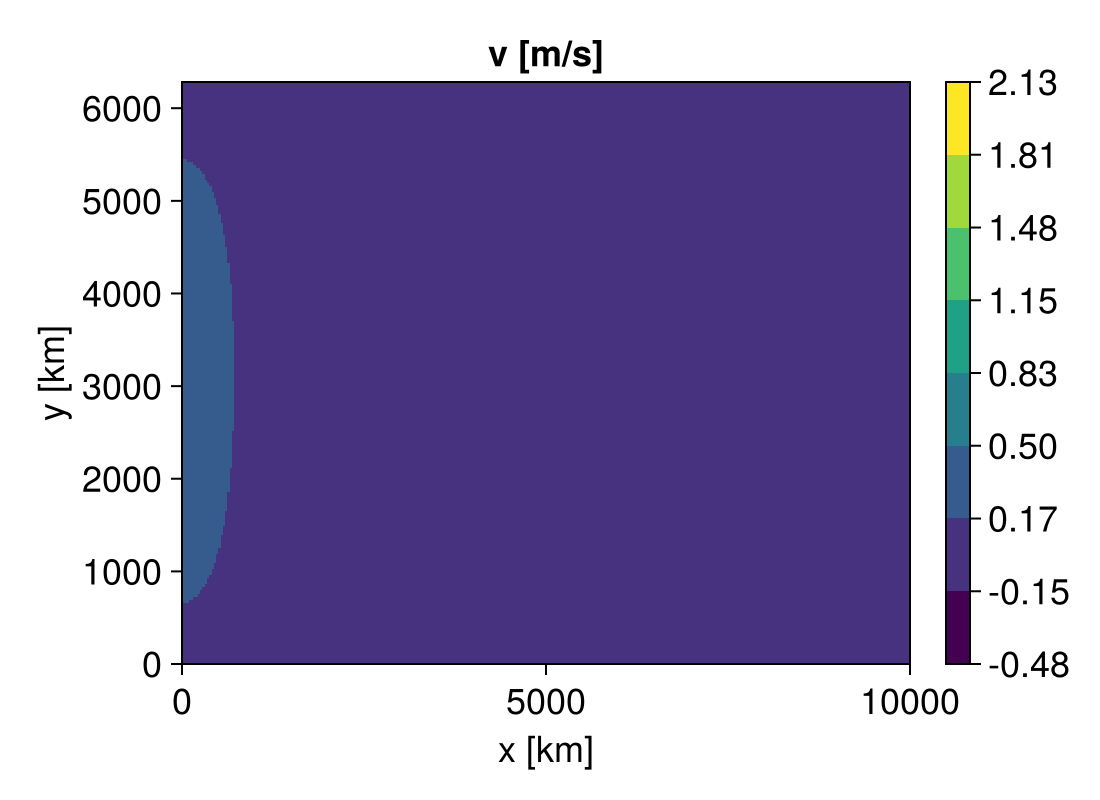}
        \caption{Depth-averaged meridional velocity $\bar{v}$.}
    \end{subfigure}
    \hfill
    \begin{subfigure}[t]{0.15\textwidth}
        \includegraphics[width=\textwidth]{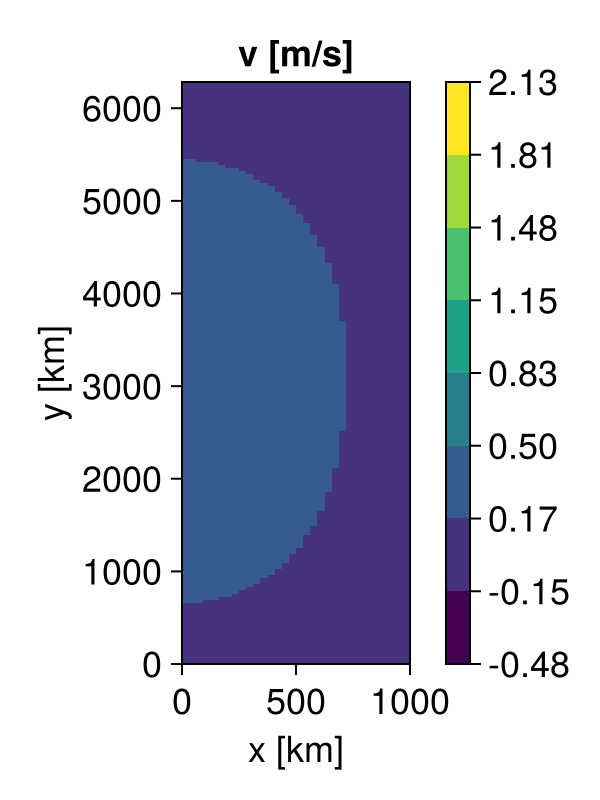} 
        \caption{\revised{Zoom on $\bar{v}$.}}
    \end{subfigure}

    \begin{subfigure}[t]{0.275\textwidth}
        \includegraphics[width=\textwidth]{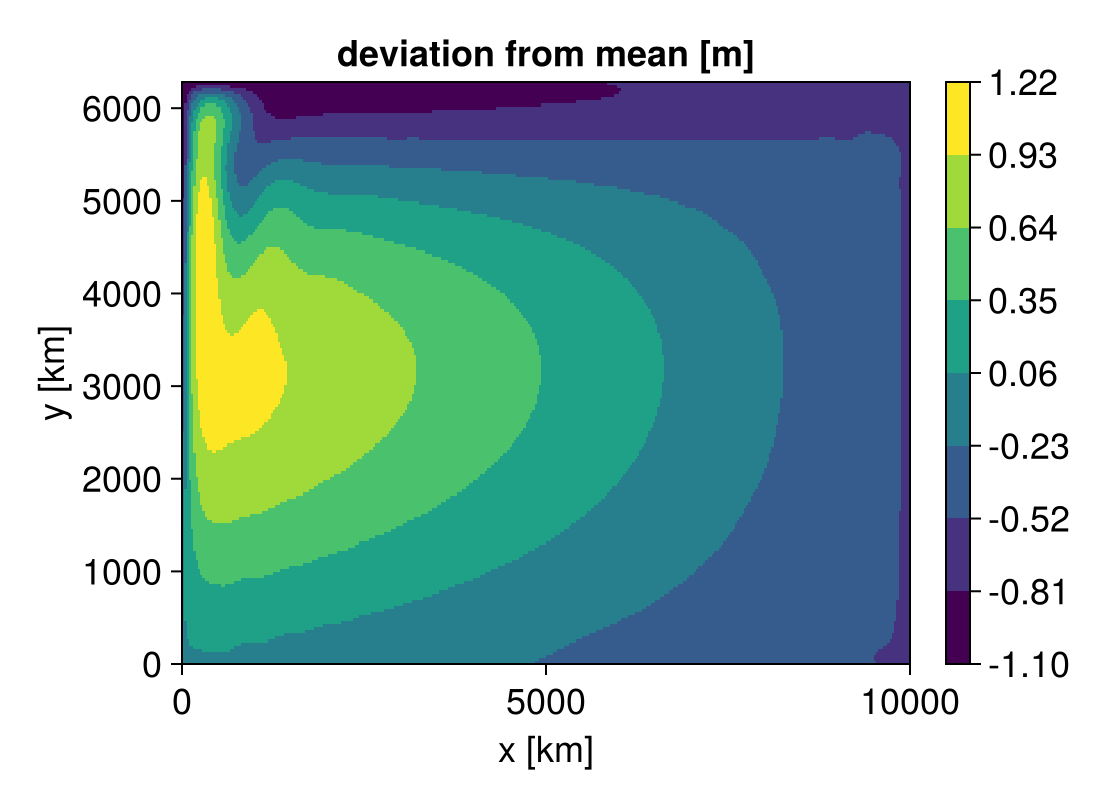}
        \caption{Deviation from mean height.}
    \end{subfigure}
    \hfill
    \begin{subfigure}[t]{0.275\textwidth}
        \includegraphics[width=\textwidth]{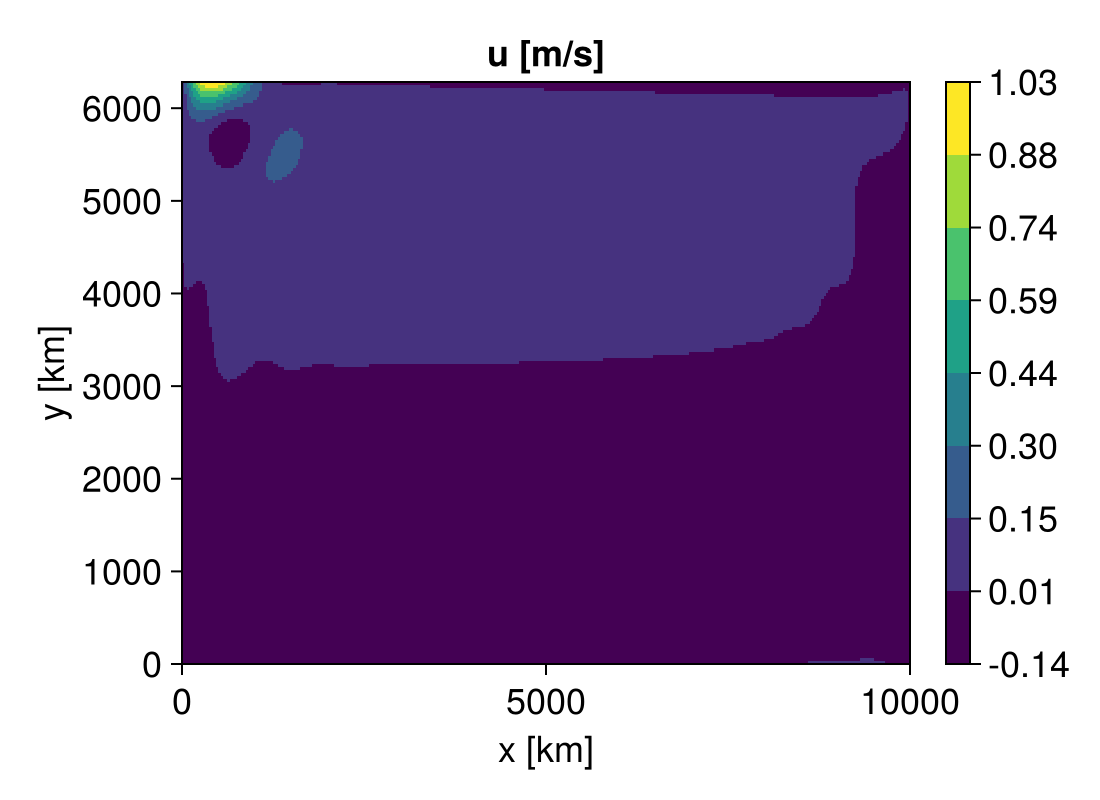}
        \caption{Depth-averaged zonal velocity $\bar{u}$.}
    \end{subfigure}
    \hfill
    \begin{subfigure}[t]{0.275\textwidth}
        \includegraphics[width=\textwidth]{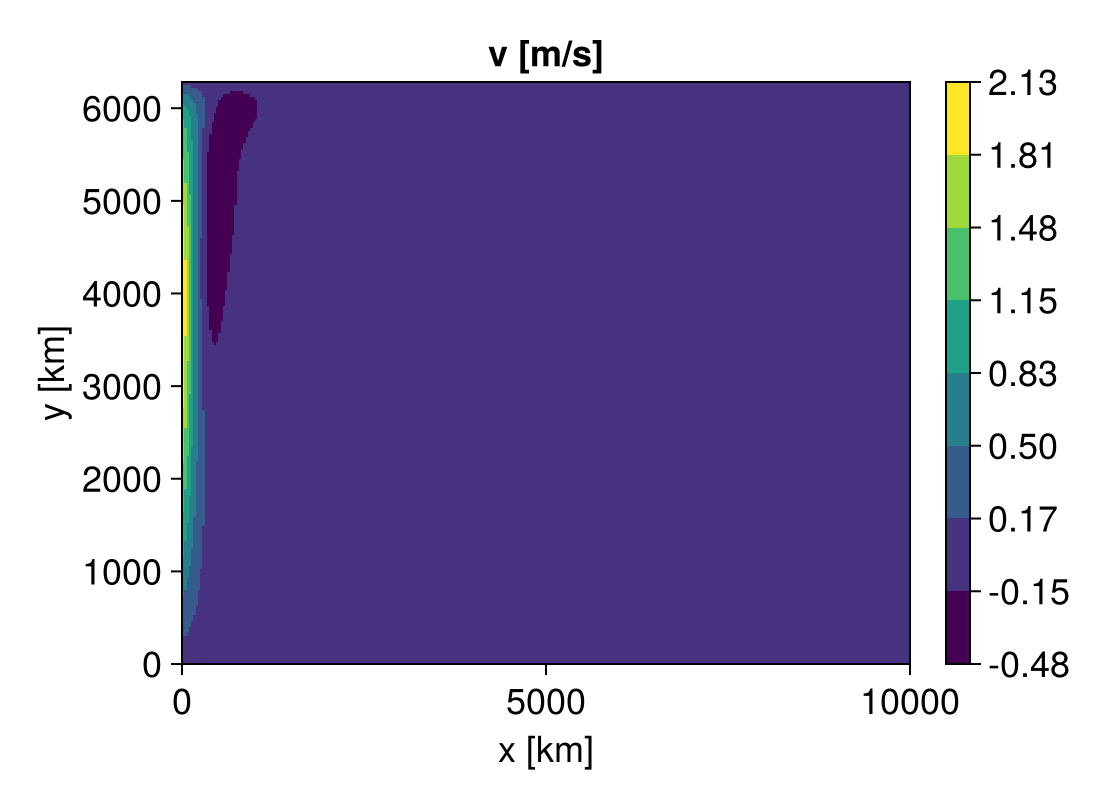}
        \caption{Depth-averaged meridional velocity $\bar{v}$.}
    \end{subfigure}
    \hfill
    \begin{subfigure}[t]{0.15\textwidth}
        \includegraphics[width=\textwidth]{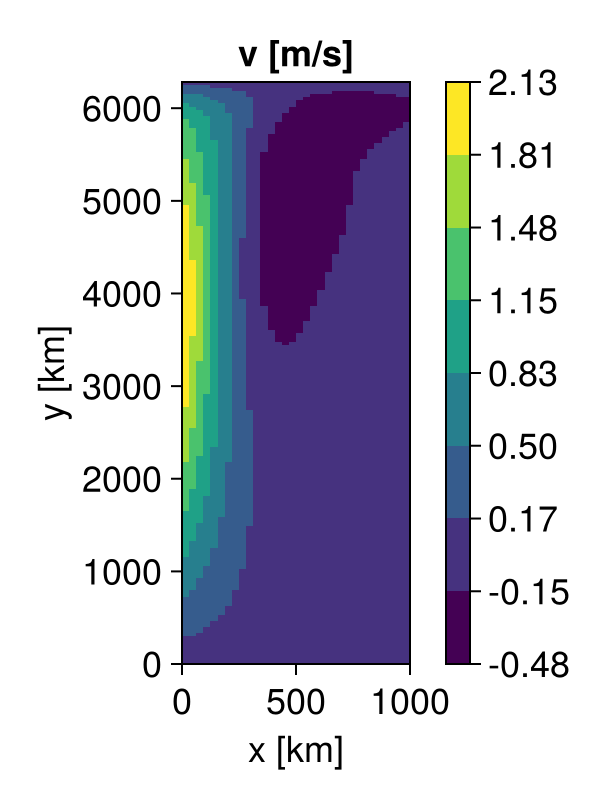} 
        \caption{\revised{Zoom on $\bar{v}$.}}
    \end{subfigure}

    \begin{subfigure}[t]{0.275\textwidth}
        \includegraphics[width=\textwidth]{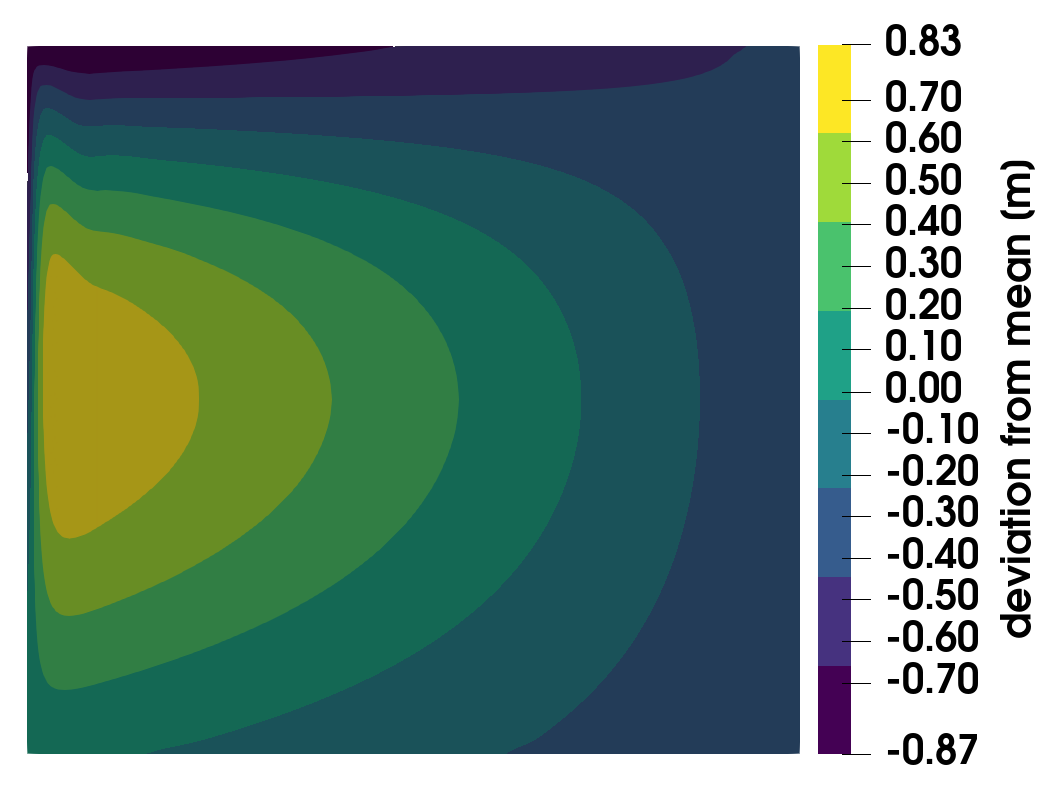}
        \caption{\revised{Deviation from mean height.}}
    \end{subfigure}
    \hfill
    \begin{subfigure}[t]{0.275\textwidth}
        \includegraphics[width=\textwidth]{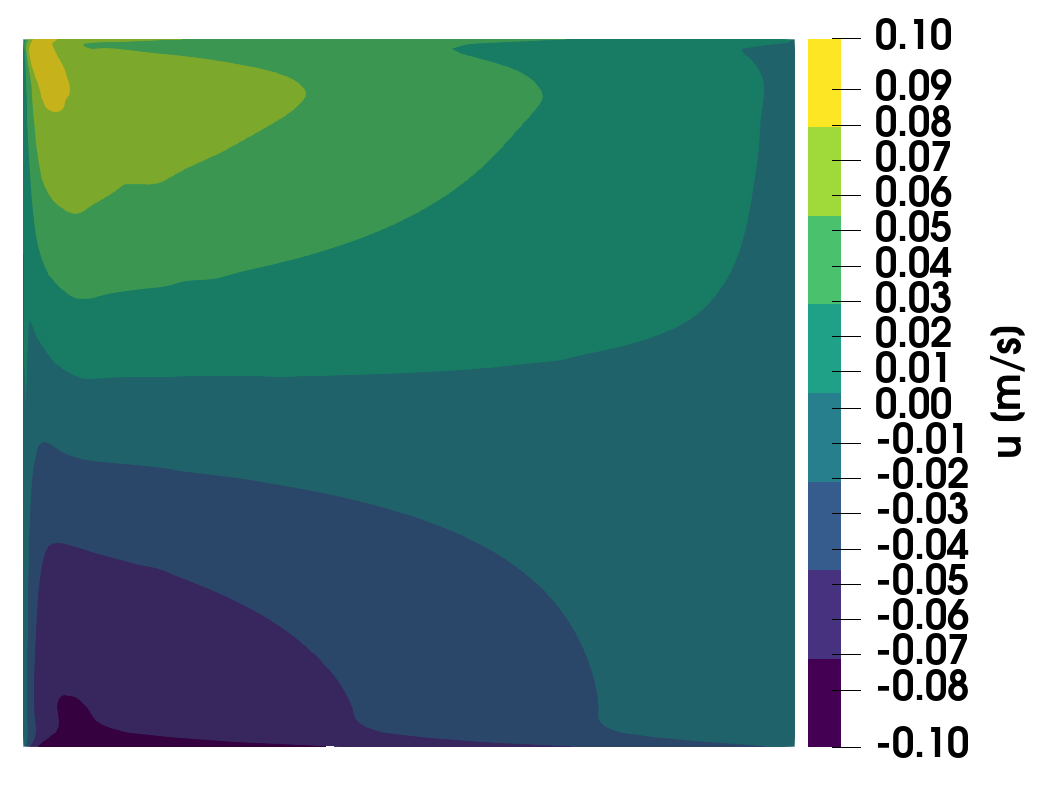}
        \caption{\revised{Depth-averaged zonal velocity $\bar{u}$.}}
    \end{subfigure}
    \hfill
    \begin{subfigure}[t]{0.275\textwidth}
        \includegraphics[width=\textwidth]{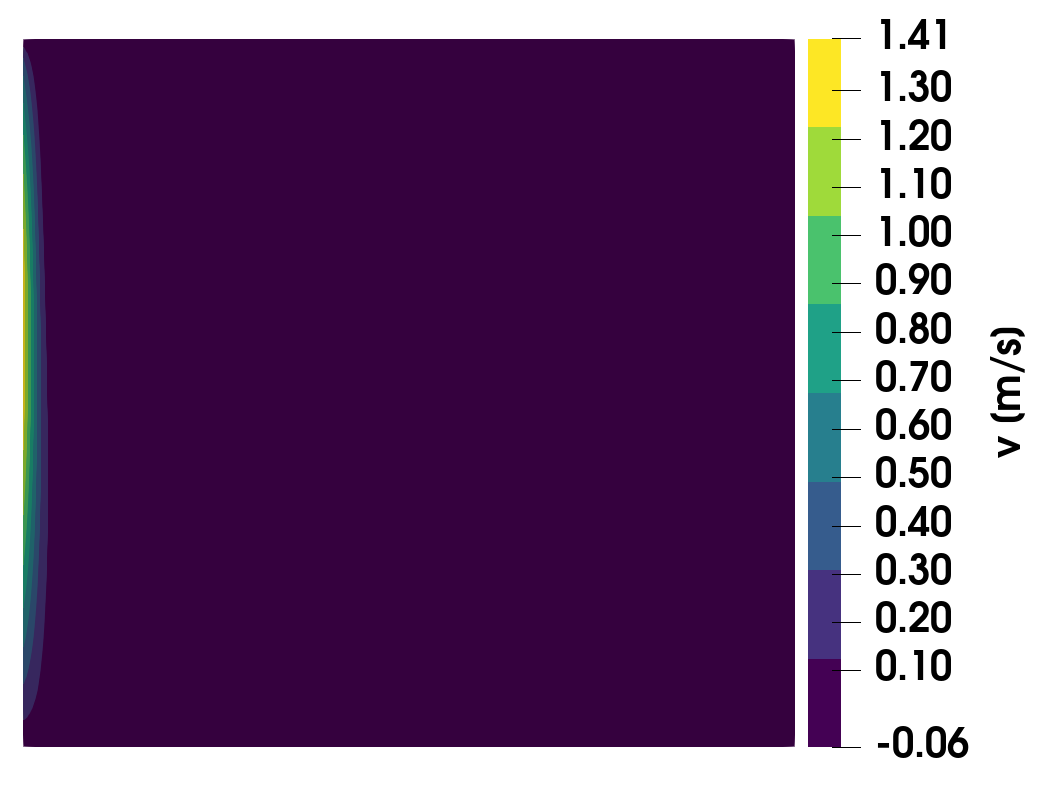}
        \caption{\revised{Depth-averaged meridional velocity $\bar{v}$.}}
    \end{subfigure}
    \hfill
    \begin{subfigure}[t]{0.135\textwidth}
        \includegraphics[width=\textwidth]{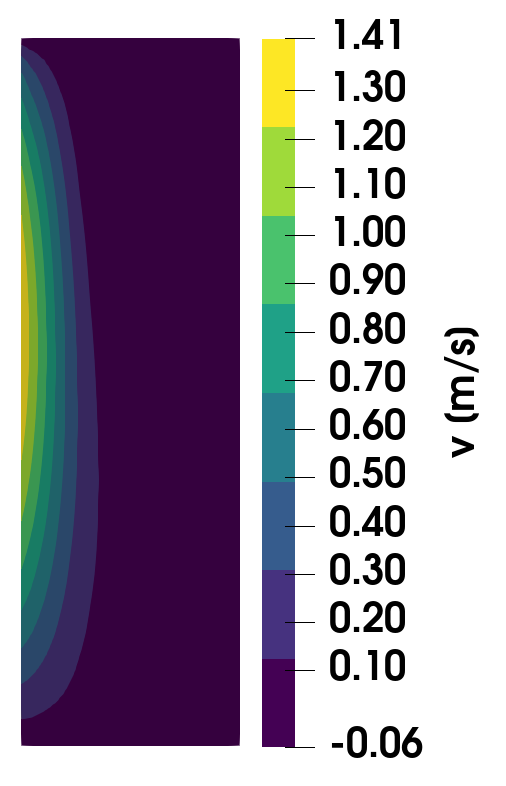} 
        \caption{\revised{Zoom on $\bar{v}$.}}
    \end{subfigure}
    
    \caption{\revised{Top: Numerical solution of the barotropic-baroclinic splitting on a grid with $320 \times 190$ cells ($\Delta x \approx \Delta y \approx 32\,\mathrm{km}$) and $20$ layers at final time $t=0.5\,\mathrm{years}$ with the non-well-balanced semi-implicit strategy.
    Middle: Same simulation as above but using the well-balancing strategy presented in this section.
    Bottom: Numerical solution of the second order Freshkiss3D solver on a grid with approximately $60000$ cells and $20$ layers at final time $t=0.5\,\mathrm{years}$.}}
    \label{fig:wbc_freshkiss}
\end{figure}

\addcontentsline{toc}{section}{Conclusion}
\revised{\section*{Conclusion}}
\revised{In this paper, the authors have described a barotropic-baroclinic splitting strategy for nonlinear multilayer shallow water flows in terrain-following coordinates. The proposed method is an operator splitting into a fast travelling barotropic part and a slow travelling baroclinic part which can be treated with different time steps. The work presents the numerical anlysis of the splitting: on the continuous level hyperbolicity of the barotropic step and the baroclinic step without mass exchanges is proven. In addition, both steps fulfill a entropy inequality. For the proposed numerical scheme, the authors showed that each of the steps and thus the overall scheme fulfills a discrete maximum principle and proved the main theoretical result -- a discrete entropy inequality, which is the key stability property for the class of finite volume schemes under consideration. Furthermore, the overall scheme inherits the well-balancing properties of the one-layer shallow water solver used in the barotropic step. Validations of the numerical scheme are presented, such as comparisons with analytical solutions, the large gain in computational cost compared to the multilayer shallow water solver without use of the barotropic-baroclinic splitting and the investigation of the preservation of the lake at rest and the geostrophic equilibrium.}

\revised{In the near future, the present work will be the building block to develop a fully second order scheme for variable density and salinity flows. Reaching second order involves the use of a Strang splitting between the barotropic and the baroclinic steps, a second order march in time in both steps and second order fluxes everywhere. All of these are well documented but require a significant effort in development. To include variable density (and salinity) the easiest method is to have a constant background density in the barotropic step. The variation of active tracers is entirely treated in the baroclinic step. A major change is that the baroclinic step will have a pressure term, but only with reduced gravity, so the computational cost will remain reasonable.}

\revised{The influence of the shallow water solver could also be investigated in more details, either in this simple setting or  more interestingly in the second order version. Indeed our preliminary results in Section~\ref{S:wbc} indicate that using a well-balanced strategy for the geostrophic equilibrium has a strong impact on the simulation. Exploring whether other strategies, such as implicit/explicit low Froude number schemes, have the same effects would be interesting, since neither these strategies nor the well-balanced fluxes have been tested at climatic scales in multilayer models. 
}

\appendix
\section{Proofs of energy (in)equalities} 

In the following we are presenting the proofs concerning the continuous and discrete energy (in)equalities presented in the sections before.

\subsection{Continuous energy equalities for smooth solutions}\label{sec:continuous_Entropy_Proofs}

We start with the continuous energy equalities.

\begin{proof}[Proof of Proposition~\ref{P:energyequality}]
    The energy equality~\eqref{eq:entropyMSW} for the original multilayer shallow water model can be found in \cite[Proposition 4.2]{ABPSM11a}. 

 We now turn to the energy equality~\eqref{eq:EIBT} fulfilled by the smooth solutions of the barotropic system~\eqref{eq:BT-closed}. 
% \textcolor{purple}{
%       Combining the first two equations of~\eqref{eq:BT-closed} we obtain the evolution equation on $\ua$
%       $$  
%             \ddt{\ua} + \baru \ddx{\ua} + g \ddx{h} = -g\ddx{\zb}.
%       $$
% }
   We have
    \begin{align*}
         & \quad - \dt \left( \frac{g h^2}{2} \right. + \left. gh \zb + \sumaN \frac{\ha \ua^2}{2} \right)  =  - \left( gh + g \zb + \sumaN \frac{\la \ua^2}{2} \right) \ddt{h} - \sumaN \ha \ua \ddt{\ua} \\
    	= & \left( gh + g \zb+ \sumaN \frac{ \la \ua^2}{2} \right) \dx (h \baru) + \sumaN \ha \ua \left( \baru \ddx{\ua} + g \ddx{h} + g \ddx{\zb} \right) \\
    	= & \left( 2gh \baru + g \zb \baru + \sumaN \frac{\la \baru \ua^2}{2} \right) \ddx{h} + \left( gh^2 + gh \zb + \sumaN \frac{\ha \ua^2}{2} \right) \ddx{\baru} + \sumaN \ha \ua \baru \ddx{\ua} + gh \baru \ddx{\zb}\\
        = & \dx \left( g h^{2} \baru+ g h \zb \baru + \sumaN \frac{\ha \ua^2 \baru}{2} \right).
    \end{align*}
 We now turn to the baroclinic System~\eqref{eq:BC} and Inequality~\eqref{eq:EIBC}. With $\ha=\la h$ and using $\dt \ha=0$ and Identity~\eqref{eq:G} we obtain
 $$ 
        \dt \left( \frac{\ha \ua^2}{2} \right) + \dx \left( \frac{\ha \ua^2 (\ua-\baru)}{2} \right) = \left(\ua \uap - \frac{\ua^2}{2} \right) \Gap - \left( \ua \uam - \frac{\ua^2}{2}\right) \Gam.
    $$
% \textcolor{purple}{
%       Indeed
%       \begin{align*}
%             \uap \Gap - \uam \Gam &= \dt (\ha\ua) + \dx (\ha\ua(\ua-\baru)) \\
%             &= \ha \ddt{\ua} + \ha (\ua - \baru) \ddx{\ua} + \ua \dx \left( \ha (\ua -\baru) \right) \\
%             &\overset{\eqref{eq:G}}{=} \ha \ddt{\ua} + \ha (\ua - \baru) \ddx{\ua} + \ua  (\Gap - \Gam)
%         \end{align*}
%      And
%      \begin{align*}
%             \dt \left( \frac{\ha \ua^2}{2} \right) + \dx \left( \frac{\ha \ua^2 (\ua-\baru)}{2} \right) 
%             &= \ha \dt \ua + \ha (\ua-\baru) \ua \dx \ua + \frac{\ua^{2}}{2} \dx ( \ha (\ua-\baru)) \\
%             &= (\uap-\ua) \Gap - (\uam-\ua) \Gam + \frac{\ua^{2}}{2}  (\Gap - \Gam)  \\
%             &=  \left(\ua \uap - \frac{\ua^2}{2} \right) \Gap - \left( \ua \uam - \frac{\ua^2}{2}\right) \Gam \\
%      \end{align*}
%  }
    
The summation over the layers and the upwinding choice~\eqref{eq:interu} yield the result
    $$ \begin{aligned}
        \dt \left( \sumaN \frac{\ha \ua^2}{2} \right) + \dx \left( \sumaN \frac{\ha \ua^2 (\ua-\baru)}{2} \right) & = \sumaN (\ua - u_{\alpha+1}) \left( \uap- \frac{u_{\alpha+1}+\ua}{2} \right) \Gap \\
        & =  - \frac12 \sumaN (\ua-u_{\alpha+1})^2|\Gap| \leq 0.
        \end{aligned}
    $$
\end{proof}

\subsection{Discrete entropy inequality} \label{S:DEIproofs}
We now prove Theorem~\ref{thm:DEI}.

\subsubsection*{Prediction step of the baroclinic step}
We start with the discrete counterpart of~\eqref{eq:EIBC} through the baroclinic part. With $\Eca := \frac{\ha \ua^2}{2}$ the kinetic energy in layer $\alpha$ and $\Ecasa := \frac{\ha \ua^{2} (\ua-\baru)}{2}$  the baroclinic  energy flux this inequality simply writes
$$ \dfrac{\partial}{\partial t} \left( E^{p} + \sumaN \Eca \right) + \dfrac{\partial}{\partial x} \left( \sumaN \Ecasa \right) \leq 0.$$
\begin{lemma} \label{L:pred}
%Smooth solutions of the prediction step~\eqref{eq:predictionbaroclinic} verify the entropy equality
%\begin{equation} \label{EIBC}
%\dfrac{\partial}{\partial t} \left(  \dfrac{g h^{2}}{2} + g h \zb + \sum_{\alpha} \dfrac{\ha \ua^{2}}{2} \right) + \dfrac{\partial}{\partial x} \left(  \sum_{\alpha} \dfrac{\ha \ua^{2} (\ua-\baru)}{2} \right) = 0
%\end{equation}
%Suppose that the mass flux used in the prediction of the total water height~\eqref{eq:hjd} is such that $\sumaN \Fhajpm= 0$ for all $j$ and that the numerical flux for the momentum $\Fhauajp$ is given by the decentered choice~\eqref{eq:fluxdec}. Define similarly the numerical kinetic energy flux
Suppose that~\eqref{eq:Ep} holds, i.e. that the potential energy $E^{p}$ decreases through the prediction step~\eqref{eq:predictionstep}. Define 
$$
\FEcasa_{j+\frac12} = \frac{1}{2}(\uajn)^{2} \left(\Fhajp\right)^{+} -  \frac{1}{2}(\uajpn)^{2}\left(\Fhajp\right)^{-}  \revised{.}
%\qquad \text{ and } \qquad \FEcjp= \sumaN  F_{j+\frac12}^{\eta}.
$$
Then $\FEcasa_{j+\frac12}+ F^{E^{p}}_{j+\frac12}$ is consistent with $\Ecasa$ and the following discrete entropy inequality holds
$$
%\left(  \frac{g h^2}{2} + g h \zb+  \sum_{\alpha=1}^N \frac{\ha \ua^{2}}{2} \right)_{j}^{\star} \leq \left(  \frac{g h^{2}}{2} + g h \zb+  \sum_{\alpha=1}^N \frac{\ha \ua^{2}}{2} \right)_{j}^{n} - \frac{\Dtn}{\Delta x} \sum_{\alpha=1}^N \left(\FEcasajp-\FEcasajm\right).
E_{j}^{\star} \leq E_{j}^{n} - \frac{\Dtn}{\Delta x} \left( F^{E^{p}}_{j+\frac12} \revised{+} \sumaN \FEcasajp  - F^{E^{p}}_{j-\frac12}  \revised{-} \sumaN  \FEcasajm  \right).
$$
\end{lemma}
\begin{proof}  
We prove that
$$
\left(  \Eca \right)_{j}^{\star} \leq \left( \Eca\right)_{j}^{n} - \dfrac{\Dtn}{\Delta x} \left(\FEcasa_{j+\frac12}-\FEcasa_{j-\frac12} \right)$$
and the result follows by summation over the layers and with~\eqref{eq:Ep}. With $\lambda = \frac{\Dtn}{\Delta x}$ we have
$$
\begin{aligned}
\hajs \Delta &
    :=\hajs \left[\left(  \frac{\ha \ua^{2}}{2} \right)_{j}^{\star} -\left(  \frac{\ha \ua^{2}}{2} \right)_{j}^{n} + \lambda \left(\FEcasajp-\FEcasajm \right) \right]  \\
	&= -  \lambda \hajn \uajn \left(\Fhauajp- \Fhauajm\right) + \frac{\lambda^{2}}{2}\left(\Fhauajp- \Fhauajm\right) ^{2}  + \frac{\lambda}{2} \hajn \left(\uajn \right)^{2} \left( \Fhajp-\Fhajm \right) \\
	&  \qquad + \lambda \hajn \left(\FEcasajp-\FEcasajm\right) - \lambda^{2} \left(\Fhajp-\Fhajm\right)\left(\FEcasajp-\FEcasajm\right)
 \end{aligned}
$$
% \textcolor{purple}{
%     $$
%     \begin{aligned}
%     \hajs \Delta & =\hajs \left[\left(  \frac{\ha \ua^{2}}{2} \right)_{j}^{\star} -\left(  \frac{\ha \ua^{2}}{2} \right)_{j}^{n} + \lambda \left(\FEcasajp-\FEcasajm \right) \right]  \\
%      	&= \frac12 \left[ \hajn \uajn- \lambda \left(\Fhauajp- \Fhauajm\right) \right]^{2} - \frac12\left[  \hajn - \lambda \left( \Fhajp-\Fhajm \right)\right]   \hajn \left(\uajn \right)^{2}\\
%     	& \qquad \qquad +\left[  \hajn - \lambda \left( \Fhajp-\Fhajm \right)\right]  \lambda \left(\FEcasajp-\FEcasajm \right) \\
%     	&= -  \lambda \hajn \uajn \left(\Fhauajp- \Fhauajm\right) + \frac{\lambda^{2}}{2}\left(\Fhauajp- \Fhauajm\right) ^{2}  + \frac{\lambda}{2} \hajn \left(\uajn \right)^{2} \left( \Fhajp-\Fhajm \right) \\
%     	&  \qquad + \lambda \hajn \left(\FEcasajp-\FEcasajm\right) - \lambda^{2} \left(\Fhajp-\Fhajm\right)\left(\FEcasajp-\FEcasajm\right)
%      \end{aligned}
%     $$
% }
We write $ \Fhauajp=  \uasjp \Fhajp$ and $  \FEcasajp = \frac{(\uasjp)^{2}}{2} \Fhajp$ where $$ \uasjp =
\begin{cases}
\uajn, & \text{ if }  \Fhajp \geq 0, \\
\uajpn, & \text{ otherwise. }
\end{cases}
$$
Straightforward computations yield
% \textcolor{purple}{
%     $$
%     \begin{aligned}
%       2 \hajs & \left[ \left(  \Eca \right)_{j}^{\star}  -  \left( \Eca\right)_{j}^{n} +  \lambda \left(\FEcasajp-\FEcasajm \right) \right]   \\
%      	&= - 2 \lambda \hajn \uajn \left( \uasjp \Fhajp- \uasjm \Fhajm\right) + \lambda^{2}\left( \uasjp \Fhajp- \uasjm \Fhajm\right) ^{2}  \\
%     	 & \qquad \qquad + \lambda \hajn \left(\uajn \right)^{2} \left( \Fhajp-\Fhajm \right) + \lambda \hajn \left( (\uasjp)^{2}\Fhajp-(\uasjm)^{2}\Fhajm\right) \\
%     	 & \qquad \qquad- \lambda^{2} \left(\Fhajp-\Fhajm\right)\left((\uasjp)^{2}\Fhajp-(\uasjm)^{2}\Fhajm\right) \\
%     	 &= \lambda \Fhajp \left( -2 \hajn \uajn \uasjp + \lambda (\uasjp)^{2}  \Fhajp + \hajn (\uajn)^{2} + \hajn (\uasjp)^{2} - \lambda (\uasjp)^{2} \Fhajp \right)\\
%     	 & \qquad \qquad +\lambda \Fhajm \left( 2 \hajn \uajn \uasjm + \lambda (\uasjm)^{2} \Fhajm - \hajn (\uajn)^{2} - \hajn (\uasjm)^{2} - \lambda (\uasjm)^{2} \Fhajm  \right) \\
%     	 & \qquad \qquad +\lambda^{2} \Fhajp \Fhajm  \left( - 2 \uasjp \uasjm + (\uasjp)^{2}  +  (\uasjm)^{2} \right) \\
%     \end{aligned}
%     $$
% }
%and we obtain finally
\begin{equation*} \label{eq:simplification}
% \small
\begin{aligned}	 
  \hajs \Delta 
 	  = \frac{\lambda}{2} \hajn \left[ (\uajn -\uasjp)^{2} \Fhajp - (\uajn -\uasjm)^{2} \Fhajm\right] + \frac{\lambda^{2}}{2} \Fhajp\Fhajm(\uasjp-\uasjm )^{2}.
\end{aligned}
\end{equation*}
When $\Fhajp \Fha_{j-\frac12} \leq 0$ the righthandside is clearly nonpositive. Otherwise
$$
\begin{cases}
\Delta =   -\dfrac{\lambda}{2 \hajs}(\uajn-u_{\alpha,j-1}^{n})^{2} \Fha_{j-\frac12} \left( \hajn- \lambda \pos{\Fhajp}\right) & \text{ if $\Fhajp>0$ and $\Fha_{j-\frac12}>0$} \\
\Delta = \dfrac{\lambda}{2\hajs}(\uajn-u_{\alpha,j+1}^{n})^{2} \Fhajp \left( \hajn+\lambda \nega{\Fha_{j-\frac12}}\right) & \text{ if $\Fhajp<0$ and $\Fha_{j-\frac12}<0$} \\
\end{cases}
$$
which in nonpositive under the CFL condition~\eqref{eq:positivity_prediction}.
%\textcolor{purple}{
%\begin{itemize}
%\item If $\Fhajp>0$ and $\Fha_{j-\frac12}>0$, we obtain
%$$  \left(  \Eca \right)_{j}^{\star}  -  \left( \Eca\right)_{j}^{n} +  \lambda \left(\FEcasajp-\FEcasajm \right) 
%= -\dfrac{\lambda}{2 \hajs}(\uajn-u_{\alpha,j-1}^{n})^{2} \Fha_{j-\frac12} ( \hajn- \lambda \Fhajp) $$
%	which is non-positive under the CFL condition~\eqref{eq:positivity_prediction}.
%\item If $\Fhajp<0$ and $\Fha_{j-\frac12}<0$, we obtain
%$$  \left(  \Eca \right)_{j}^{\star}  -  \left( \Eca\right)_{j}^{n} +  \lambda \left(\FEcasajp-\FEcasajm \right) 
%= \dfrac{\lambda}{2\hajs}(\uajn-u_{\alpha,j+1}^{n})^{2} \Fhajp ( \hajn+\lambda \Fha_{j-\frac12})$$
%	which is non-positive under the CFL condition~\eqref{eq:positivity_prediction}.
%\item If $\Fhajp>0$ and $\Fha_{j-\frac12}<0$, then all the terms at the righthandside of~\eqref{eq:simplification} vanish.
%\item  If $\Fhajp<0$ and $\Fha_{j-\frac12}>0$, then all the terms at the righthandside of~\eqref{eq:simplification} are non-positive.
%\end{itemize}}
\end{proof}

\subsubsection*{Correction step of the baroclinic step}
\begin{lemma}\label{L:cor}
    Consider Definition~\eqref{eq:def_exchange_terms} for the exchange terms and Definition~\eqref{eq:interu} for the interface velocities. 
    The total energy decreases  in the correction step~\eqref{eq:correctionstep} with implicit interface velocities $\sharp=n+\frac12$. It is also true in the explicit case $\sharp = \star$ under the \revised{additional} restriction on the timestep~\eqref{eq:CFLcor}. %$E_{j}^{n+\frac12} \leq E_{j}^{\star}$
   % $$
     %   \left(  \frac{g h^2}{2} + g h \zb+  \sumaN \frac{\ha \ua^{2}}{2} \right)_{j}^{n+\frac12} \leq \left(  \frac{g h^{2}}{2} + g h \zb+  \sumaN \frac{\ha \ua^{2}}{2} \right)_{j}^{\star}.
   % $$
%The inequality  also holds in the explicit case $\sharp = \star$ under Condition~\eqref{eq:CFLcor}.
\end{lemma}

\begin{proof}
The total water height $h$ does not vary during the correction step, so we only have to prove that the total kinetic energy decreases. 
Definitions~\eqref{eq:hajd} and~\eqref{eq:correctionstep} yield in the explicit case
$$
\begin{aligned}
\dfrac{2 \hajd }{\Dtn}  \left( \left(  \Eca \right)_{j}^{n+\frac12}  -\left(  \Eca \right)_{j}^{\star}  \right) &= -\hajs \Gapj^{\star}\left( \uajpds- \uajs \right)^{2}+ \hajs \Gapj^{\star}(\uajpds)^{2} \\
& \qquad  +\hajs \Gamj^{\star}\left(  \uajmds- \uajs \right)^{2}- \hajs \Gamj^{\star}(\uajmds)^{2} \\
 & \qquad \qquad + \Dtn \left( \uajpds \Gapj^{\star} -\uajmds \Gamj^{\star} \right)^{2} \\
\end{aligned}
$$
%$$
%\begin{aligned}
%\dfrac{1}{\Dtn} & \left( (\hajd \uajd)^{2} - \hajd \hajs (\uajs)^{2} \right) \\
%	&= \dfrac{\left[\hajs \uajs + \Dtn \left( \uajpds \Gapj^{\star} -  \uajmds  \Gamj^{\star}\right) \right]^{2}}{\Dtn} \\
%	& \qquad \qquad - \dfrac{\left[\hajs + \Dtn \left(  \Gapj^{\star} -  \Gamj^{\star}\right) \right] \hajs (\uajs)^{2}}{\Dtn} \ \\
%&= \hajs \Gapj^{\star} \left(  2  \uajpds \uajs- (\uajs)^{2} \right) - \hajs\Gamj^{\star} \left(  2  \uajmds \uajs- (\uajs)^{2}  \right)\\
%	& \qquad \qquad + \Dtn \left(  \uajpds \Gapj^{\star} -  \uajmds  \Gamj^{\star}\right)^{2} \\
%&= -\hajs \Gapj^{\star}\left( \uajpds- \uajs \right)^{2}+ \hajs \Gapj^{\star}(\uajpds)^{2} \\
%& \qquad \qquad +\hajs \Gamj^{\star}\left(  \uajmds- \uajs \right)^{2}- \hajs \Gamj^{\star}(\uajmds)^{2} \\
% & \qquad \qquad + \Dtn \left( \uajpds \Gapj^{\star} -\uajmds \Gamj^{\star} \right)^{2} \\
%\end{aligned}
%$$
We replace $\hajs$ by $\hajd-\Dtn (\Gapj^{\star}-\Gamj^{\star})$ in the second and forth terms
% \textcolor{purple}{
%     The terms in $\Dtn$ in total are
%     $$ \begin{aligned}
%     - (\Gapj^{\star}-\Gamj^{\star})& \Gapj^{\star}(\uajpds)^{2} +(\Gapj^{\star}-\Gamj^{\star}) \Gamj^{\star}(\uajmds)^{2} \\
%     &+\left( \uajpds \Gapj^{\star} -\uajmds\Gamj^{\star}\right)^{2}  \\
%     &= \Gapj^{\star}\Gamj^{\star} (\uajpds)^{2}+\Gapj^{\star} \Gamj^{\star} (\uajmds)^{2}-2 \uajpds \uajmds \Gapj^{\star} \Gamj^{\star} \\
%     &=\Gapj^{\star} \Gamj^{\star}(\uajpds-\uajmds)^{2}. \\
%     \end{aligned}
%      $$
% }
 and we arrive at
 $$
 \begin{aligned}
\dfrac{\hajd(\uajd)^{2}  - \hajs (\uajs)^{2} }{\Dtn}&=  - \dfrac{\hajs}{\hajd}\Gapj^{\star}\left(  \uajpds- \uajs \right)^{2}+ \dfrac{\hajs}{\hajd}\Gamj^{\star}\left(  \uajmds- \uajs \right)^{2}\\
& \qquad \qquad +  \Gapj^{\star}(\uajpds)^{2}- \Gamj^{\star}(\uajmds)^{2} \\
& \qquad \qquad + \dfrac{\Dtn}{\hajd} \Gapj^{\star} \Gamj^{\star}(\uajpds-\uajmds)^{2}
\end{aligned}
$$
We now write $\Gapj^{\star} \Gamj^{\star} = \left( (\Gapj^{\star})^{+} -  (\Gapj^{\star})^{-} \right) \left( (\Gamj^{\star})^{+} -  (\Gamj^{\star})^{-} \right)$ and use the decentered choice~\eqref{eq:interu} to get
% \textcolor{purple}{
%      $$
%      \begin{aligned}
%     \dfrac{\hajd(\uajd)^{2}  - \hajs (\uajs)^{2} }{\Dtn}&=  -\dfrac{\hajs}{\hajd} (\Gapj^{\star})^{+}\left(  \uajps- \uajs \right)^{2}- \dfrac{\hajs}{\hajd}(\Gamj^{\star})^{-} \left(  \uajms- \uajs\right)^{2} \\
%     & \qquad \qquad +  \Gapj^{\star}(\uajpds)^{2}- \Gamj^{\star}(\uajmds)^{2} \\
%     & \qquad \qquad + \dfrac{\Dtn}{\hajd} (\Gapj^{\star})^{+} (\Gamj^{\star})^{+}(\uajps-\uajs)^{2} \\
%     & \qquad \qquad + \dfrac{\Dtn}{\hajd} (\Gapj^{\star})^{-} (\Gamj^{\star})^{-}(\uajs-\uajms)^{2} \\
%     & \qquad \qquad - \dfrac{\Dtn}{\hajd} (\Gapj^{\star})^{+} (\Gamj^{\star})^{-}(\uajps-\uajms)^{2} 
%     \end{aligned}
%     $$
% }
 $$
 \begin{aligned}
2 \dfrac{ \left(  \Ec \right)_{j}^{n+\frac12}  -\left(  \Ec \right)_{j}^{\star}  }{\Dtn}&   \leq - \sumaN(\Gapj^{\star})^{+}\left(  \uajps- \uajs \right)^{2} \left[\dfrac{\hajs- \Dtn (\Gamj^{\star})^{+} }{\hajd} \right] \\
& \qquad \qquad -\sumaN(\Gamj^{\star})^{-}\left(  \uajms- \uajs \right)^{2} \left[\dfrac{\hajs- \Dtn (\Gapj^{\star})^{-} }{\hajd} \right].
\end{aligned}
$$
The bracket terms are non-negative under the CFL condition~\eqref{eq:CFLcor}. 
%The result follows by summing over all the layers.
%Summing over all the layer we obtain
%$$
% \begin{aligned}
% \sumaN \dfrac{\hajd(\uajd)^{2}}{2} & \leq    \sumaN \dfrac{ \hajs (\uajs)^{2} }{2} - \Dtn  \sumaN (\Gapj^{\star})^{+}\left(  \uajps- \uajs \right)^{2} \left[\dfrac{\hajs- \Dtn (\Gamj^{\star})^{+} }{\hajd} \right]  \\
% & \qquad \qquad  -(\Gapj^{\star})^{-}\left( \uajs -  \uajps\right)^{2} \left[\dfrac{h_{\alpha, j+1}^{\star}- \Dtn (\Gappj^{\star})^{-} }{\hajpd} \right]
% \end{aligned}
%$$

The implicit case $\sharp=n+\frac12$ which is simpler. We write
$$
\begin{aligned}
\dfrac{\hajs (\uajd)^{2}  -  \hajs (\uajs)^{2} }{\Dtn}& =  -   \Gapj^{\star} \left(   \uajpdnd - \uajd \right)^{2} +  \Gamj^{\star} \left( \uajmdnd - \uajd\right)^{2} \\
 & \quad +    \Gapj^{\star} (\uajpdnd)^{2} -  \Gamj^{\star} (\uajmdnd)^{2} \\
 & \quad \quad   - \Gapj^{\star} (\uajpdnd)^{2} + \Gamj^{\star} (\uajmdnd)^{2} 
\end{aligned}
$$
Choice~\eqref{eq:interu} yields the result
$$
\begin{aligned}
2 \dfrac{ \left(  \Ec \right)_{j}^{n+\frac12}  -\left(  \Ec \right)_{j}^{\star}  }{\Dtn}
 & \leq - \sumaN \left[ \pos{\Gapj^{\star}} \left(   u_{\alpha, j+1}^{n+\frac12} - \uajd \right)^{2} + \nega{\Gamj^{\star}} \left( u_{\alpha, j-1}^{n+\frac12} - \uajd\right)^{2} \right].
\end{aligned}
$$
\end{proof}

\subsubsection*{Barotropic step}
We now turn to the discrete entropy inequality in the barotropic loop. Let us recall that the physically relevant solutions of the shallow water equation (first two equations of~\eqref{eq:BT}) verify the entropy inequality
$$
\partial_{t} E^{sw}+ \partial_{x}  f^{sw} \leq 0
$$
where denote by $E^{sw}= \frac{g h^{2}}{2} + g h \zb + \frac{ h \baru^{2}}{2}$ and $f^{sw} = g h^{2} \baru + g h \zb \baru + \frac{ h \baru^{3}}{2} $
\begin{lemma} \label{L:EIBT}
Suppose that the shallow water scheme~\eqref{eq:FV-SW}  is entropy satisfying under Condition~\eqref{eq:CFL_SWE}: there exists some discrete numerical entropy fluxes $f^{sw, k}_{j\pm\frac12}$ consistent with the continuous entropy flux $f^{sw}$ such that
\begin{equation} \label{eq:DEI-SW}
 E^{sw, n+\frac12,k+1}_{j} \leq  E^{sw, n+\frac12,k}_{j}  + \dfrac{\dtk}{\Delta x} \left(  f^{sw, k}_{j+\frac12}-f^{sw, n, k}_{j-\frac12} \right).
\end{equation}
If  Condition~\eqref{eq:non-negativity} holds, there exists some numerical entropy fluxes $\mathcal{F}^{E}_{j \pm \frac12}$ consistent with 
the entropy flux of~\eqref{eq:EIBT} such that the following discrete entropy inequality holds
$$%\dfrac{g h^{2}}{2} + g h \zb+  \sum_{\alpha} \dfrac{\ha \ua^{2}}{2} 
E_{j}^{n+1} \leq E_{j}^{n+\frac12} + \dfrac{\Dtn}{\Delta x} (\mathcal{F}^{E}_{j+\frac12}-\mathcal{F}^{E}_{j-\frac12}).$$
\end{lemma}
\begin{proof}
Proposition~\ref{P:sumsa} implies that the identity $E=E^{sw}+ \sumaN \Eca$ holds at the discrete level.
% \textcolor{purple}{
%     We introduce $\baru$ in the kinetic energy
%     $$ \left( \sum_{\alpha} \dfrac{\la h \ua^{2}}{2}  \right)_{j}^{n+1} =  \left( \sum_{\alpha} \dfrac{\la h \baru^{2} + 2 \la h \baru(\ua-\baru)+ \la h (\ua-\baru)^{2}}{2} \right)_{j}^{n+1} $$
%     and Proposition~\ref{P:sumsigma} to write
%     $$ 
%     \left(  \dfrac{g h^{2}}{2} + g h \zb+  \sum_{\alpha} \dfrac{h \ua^{2}}{2} \right)_{j}^{n+1} = E_{j}^{sw, n+1} +  \frac{1}{2} \sum_{\alpha}  \la ( h \sa^{2} )^{n+1}_{j}
%     $$
% }
The summation of the inequalities~\eqref{eq:DEI-SW} on all the subtimestep gives
$$E_{j}^{sw, n+1}  :=E_{j}^{sw, n+1, K} \leq E_{j}^{sw, n+\frac12} + \dfrac{\Dtn}{\Delta x} \left( \mathcal{F}^{sw}_{j+\frac12} - \mathcal{F}^{sw}_{j-\frac12} \right) \ \text{ where } \ \Dtn \mathcal{F}^{sw}_{j+\frac12}= \sum_{k=0}^{K-1 } \dtk f_{j+\frac12}^{sw, k} .  $$
We now prove that
$$
 (h \sa^{2})_{j}^{n+1} \leq  (h \sa^{2})_{j}^{n+\frac12}  - \dfrac{\Dtn }{\Delta x}  \big( (\sigma_{\alpha, j+\frac12}^{n})^{2} \mathcal{F}_{j+\frac12}^{h}-  (\sigma_{\alpha, j-\frac12}^{n})^{2} \mathcal{F}_{j-\frac12}^{h} \big)
 $$
 with the familiar decentered choice
$$ \sigma_{\alpha, j+\frac12}^{n}= 
\begin{cases}  
\sigma_{\alpha,j}^n & \text{ if } \mathcal{F}_{j+\frac12}^{h} \geq 0 \\  
\sigma_{\alpha,j+1}^n &\text{ if } \mathcal{F}_{j+\frac12}^{h} < 0 
\end{cases} $$
Multiplying by  $\hjnp$ which is nonnegative if Condition~\eqref{eq:non-negativity} holds, the inequality is equivalent to nonnegativity of
 $$  -\mathcal{F}_{j+\frac12}^{h} \hjn \left[ \sigma_{\alpha, j+\frac12}^{n+\frac12} - \sigma_{\alpha,j}^{n+\frac12}\right]^{2}  + \mathcal{F}_{j-\frac12}^{h} \hjn \left[ \sigma_{\alpha, j-\frac12}^{n+\frac12} - \sigma_{\alpha,j}^{n+\frac12} \right]^{2} 
 - \frac{\Dtn}{\Delta x}  \mathcal{F}_{j+\frac12}^{h} \mathcal{F}_{j-\frac12}^{h} \left[\sigma_{\alpha, j-\frac12}^{n+\frac12}-\sigma_{\alpha, j+\frac12}^{n+\frac12} \right]^{2}.$$
 This quantity is equal to
 $$
 \begin{cases}
 \mathcal{F}_{j-\frac12}^{h} \left[ \sigma_{\alpha,j-1}^{n+\frac12} - \sigma_{\alpha,j}^{n+\frac12} \right]^{2}  \left[ \hjn  - \frac{\Dtn}{\Delta x} \pos{ \mathcal{F}_{j+\frac12}^{h}}\right] & \text{ if  $\mathcal{F}_{j+\frac12}^{h} \geq 0$ and $\mathcal{F}_{j-\frac12}^{h} \geq 0$}, \\
 -\mathcal{F}_{j+\frac12}^{h} \left[ \sigma_{\alpha, j+1}^{n+\frac12} - \sigma_{\alpha,j}^{n+\frac12} \right]^{2}  \left[ \hjn  + \frac{\Dtn}{\Delta x}  \mathcal{F}_{j-\frac12}^{h} \right]  & \text{ if  $\mathcal{F}_{j+\frac12}^{h} \leq 0$ and $\mathcal{F}_{j-\frac12}^{h} \leq 0$}, \\
 \text{clearly nonnegative otherwise.}
 \end{cases}
 $$
Condition~\eqref{eq:non-negativity}  yields the nonnegativity in the first two cases.

\end{proof}
%\begin{remark}
%The condition   $\frac{\Dtn}{\Delta x} \la \mathcal{F}_{j+\frac12}^{h} \leq \hajn$ in the first case is fulfilled when using the Rusanov scheme. Indeed this case
%$$ \mathcal{F}_{j+\frac12}^{h} = \dfrac{\hj^{n+\frac12} \bar\uj^{n}+ \hjpn \baru_{j+1}^{n}}{2} - A_{j+\frac12} \dfrac{ \hjpn -\hj^{n}}{2} $$
%with $$A_{j+\frac12} \geq \max\left( |\bar\uj^{n}|+ \sqrt{g \hj^{n}},  |\baru_{j+1}^{n}|+ \sqrt{g \hjpn}\right) \geq  |\baru_{j+1}^{n}| $$
%Thus
%$$
%\frac{\Dtn}{\Delta x} \mathcal{F}_{j+\frac12}^{h} \leq \dfrac{\Dtn}{\Delta x} \left( \dfrac{\hj^{n} \bar\uj^{n}}{2} + A_{j+\frac12} \dfrac{\hj^{n}}{2} \right) \leq \dfrac{\Dtn}{\Delta x} A_{j+\frac12} \hj^{n}
%$$
%but
%$$ 0 \leq \hj^{n} \bar\uj^{n}+ \hjpn \baru_{j+1}^{n} - A_{j+\frac12} ( \hjpn -\hj^{n}) = \hj^{n} (\bar\uj^{n} + A_{j+\frac12}) - \hjpn( A_{j+\frac12}-\baru_{j+1}^{n})$$
%and the CFL condition imposes that $\dfrac{\Delta x}{\Dtn} \geq A_{j+\frac12}$ and the result.
%\end{remark}

\paragraph{Acknowledgements} S. Hörnschemeyer was supported in part by Deutsche Forschungsgemeinschaft under DFG grant 320021702 / GRK2326. N. Aguillon acknowledges the ANR – FRANCE (French National Research Agency) for its financial support of the project NASSMOM  ANR-23-CE56-0005-01. \\
The authors would like to thank Emmanuel Audusse and Carlos Par\'es for fruitful discussions. In addition SH would like \revised{to} thank Sebastian Noelle for a careful reading and discussions of the manuscript.

\bibliographystyle{alpha}
\bibliography{biblio}

\end{document}

%% file: EOC_table_combined_L1_splitting_until_1600NX_160layers_t=60.0_Fr=0.04.tex
cells $\times$ layers & $h$ error & EOC & $u$ error & EOC \\ 
\hline
50 $\times$ 5 & 6.10e-04 & --- & 8.37e-02 & --- \\
100 $\times$ 10 & 1.98e-04 & 1.62 & 6.05e-02 & 0.47 \\
200 $\times$ 20 & 1.02e-04 & 0.96 & 3.74e-02 & 0.70 \\
400 $\times$ 40 & 5.50e-05 & 0.88 & 2.09e-02 & 0.84 \\
800 $\times$ 80 & 2.88e-05 & 0.94 & 1.11e-02 & 0.92 \\
1600 $\times$ 160 & 1.58e-05 & 0.86 & 5.70e-03 & 0.96 